\documentclass{daj}

\usepackage{amsmath,amsfonts,amssymb,mathtools}
\usepackage{graphicx,float}

\graphicspath{ {./images/} }
\usepackage{amsthm}
\newtheorem{theorem}{Theorem}

\newtheorem{corollary}[theorem]{Corollary}
\newtheorem{lemma}[theorem]{Lemma}
\newtheorem{proposition}[theorem]{Proposition}

\theoremstyle{definition}
\newtheorem{definition}{Definition}[section]

\theoremstyle{remark}
\newtheorem*{remark}{Remark}

\newcommand{\R}{\mathbb{R}}
\newcommand{\Z}{\mathbb{Z}}
\newcommand{\N}{\mathbb{N}}

\newcommand{\tf}{\tilde{F}}
\newcommand{\tfo}{\tilde{f_1}}
\newcommand{\tft}{\tilde{f_2}}
\newcommand{\tfou}{\prescript{}{u}{\tilde{f_1}}}
\newcommand{\tfol}{\prescript{}{l}{\tilde{f_1}}}
\newcommand{\tftu}{\prescript{}{u}{\tilde{f_2}}}
\newcommand{\tftl}{\prescript{}{l}{\tilde{f_2}}}

\newcommand{\tg}{\tilde{g}}
\newcommand{\td}{\tilde{\Delta}_{A_2}}
\newcommand{\tth}{\tilde{\theta}}

 \newcommand{\CC}{C}

\dajAUTHORdetails{%
  title = {Universally Optimal Periodic Configurations in the Plane}, %% please capitalize all significant words
  author = {Douglas P. Hardin and Nathaniel J. Tenpas},
  plaintextauthor = {Douglas P. Hardin and Nathaniel J. Tenpas},
  keywords = {universal optimality, $A_2$ lattice, energy minimization},
}   %%% END \dajAUTHORdetails

\dajEDITORdetails{%
   year={2025},
   number={26},
   received={4 April 2023},   % received date: example: 7 January 2017
   published={10 October 2025},  % published date
   doi={10.19086/da.144978},       % XXX = number of paper, e.g. da006 for paper#6
}   %%% END \dajEDITORdetails

\begin{document}

\begin{frontmatter}[classification=text]
%% EDITOR: this will force the keywords to appear right after the Abstract.
%%   If the abstract is too long and would force the keywords off the
%%   front page, please comment out % [classification=text] above
%%   This way the keywords will be floated on the bottom of the first page
%%   even though the Abstract spills over to the next page.

%%% AUTHOR: Title goes here.  This line is optional.  You must use it
%%   if title has footnote attached or requires nontrivial typesetting,
%%   e.g., inclusion of linebreaks to force nice layout.
%\title{Universally Optimal Periodic Configurations in the Plane}
\author[dph]{Douglas P. Hardin}
\author[njt]{Nathaniel J. Tenpas}

% Abstract
%\begin{abstract}
%    Add your abstract here.
%\end{abstract}
\begin{abstract}
We develop linear programming bounds for the energy of configurations in $\R^d$,     periodic with respect to a lattice. 
In certain cases, the construction of sharp bounds can be formulated as a finite dimensional, multivariate polynomial interpolation problem. We use this framework to show a scaling of the equitriangular lattice $A_2$ is universally optimal among all configurations of the form $\omega_4+ A_2$ where $\omega_4$ is a 4-point configuration in $\R^2$. Likewise, we show a scaling and rotation of $A_2$ is universally optimal among all configurations of the form $\omega_6+L$ where $\omega_6$ is a 6-point configuration in $\R^2$ and $L=\Z \times \sqrt{3} \Z$. 
\end{abstract}
\end{frontmatter}

%\tableofcontents
% Introduction and Overview

\section{Introduction and Overview of Results}
%Let  $\Lambda$ be a lattice in $\R^d$  and let   $F:\R^d  \rightarrow (-\infty,\infty]$ be a  lower-semicontinuous and $\Lambda$-periodic (i.e., $F(\cdot+v)=F$ for all $v\in \Lambda$) potential.  
%We refer to $\Lambda$ as the \textit{periodization lattice} and $F$ as the \textit{periodization kernel}. 
%For a finite multiset $\omega_n=\{x_1,...,x_n\}\subseteq \R^d$ of cardinality $n$, called a \textit{configuration}, we consider the \textit{$F$-energy} of $\omega_n$ defined by
%\[
%E_F(\omega_n):=\sum_{i=1}^n\sum_{\substack{j=1\\ j\neq i}}^n F(x_i-x_j).
%\]
% Without loss of generality, we may assume that $\omega_n$ lies in some specified fundamental domain $\Omega_\Lambda:=\R^d/\Lambda$ since   replacing a point $x\in\omega_n$ with any point in $x+\Lambda$ does not change $E_F(\omega_n)$.
%\\

%The \textit{minimal discrete $n$-point $F$-energy} is    defined as
%\begin{equation}\label{minEndef}
%\mathcal{E}_F(n):= \inf \{E_F(\omega_n)\mid  \omega_n \subseteq \R^d, \, \lvert \omega_n\rvert=n\},
%\end{equation}
%where $\lvert \omega\rvert$ denotes the cardinality of a multiset $\omega$.    

Let   $F:\R^d  \rightarrow (-\infty,\infty]$ be a  lower-semicontinuous potential.  
For a finite multiset $\omega_n=\{x_1,...,x_n\}\subseteq \R^d$ of cardinality $n$, we consider the \textit{$F$-energy}  of $\omega_n$ defined by
\[
E_F(\omega_n):=\sum_{i=1}^n\sum_{\substack{j=1\\ j\neq i}}^n F(x_i-x_j).
\]
 If for some lattice $\Lambda\subseteq \R^d$ (a lattice is a discrete, additive, rank $d$ subgroup of $\R^d$, more on lattices in Section \ref{lattbasics}), $F$ is $\Lambda$-periodic  (i.e., $F(\cdot+v)=F$ for all $v\in \Lambda$), then we also refer to the \textit{$F$-energy} as \textit{periodic energy}.  In this case, without loss of generality, we may assume that $\omega_n$ lies in the flat torus given by some specified fundamental domain $\Omega_\Lambda:=\R^d/\Lambda$, since   replacing a point $x\in\omega_n$ with any point in $x+\Lambda$ does not change $E_F(\omega_n)$.

The \textit{minimal discrete $n$-point $F$-energy} is subsequently defined as
\begin{equation}\label{minEndef}
\mathcal{E}_F(n):= \inf \{E_F(\omega_n)\mid  \omega_n \subseteq \R^d, \, \lvert \omega_n\rvert=n\},
\end{equation} 
%and if we restrict our multisets to some set $X$, we denote minimal energy with the extra restriction as $\mathcal{E}_F(n,X)$. 
and for $n\not \in \N$, the definition is extended by interpolation.
An $n$-point configuration $\omega_n\subset \R^d$  satisfying $E_F(\omega_n)=\mathcal{E}_F(n)$ is called \textit{$F$-optimal}. Note that the lower-semicontinuity of $F$ and the compactness of $\Omega_\Lambda$ in the torus topology imply the existence of at least one $F$-optimal configuration for each cardinality $n>1$. 

We consider potentials generated by  a function  $f: [0,\infty)\to [0,\infty]$ with $d$-rapid decay (i.e. $f(r^2) \in \mathcal{O}(r^{-s} ),  r \rightarrow \infty$, for some $s>d$) using 
\begin{equation}\label{Ffdef}
F_{f,\Lambda}(x):=\sum_{v\in \Lambda}f(|x+v|^2).
\end{equation}
The potential $F_{f,\Lambda}$ has the following  physical interpretation: if $f(r^2)$ represents the energy required to place a pair of unit charge particles at distance $r$ from each other, then $F_{f,\Lambda}(x)$ is the energy required to place such a particle at the point $x$ in the presence of existing particles at points of $\Lambda$.    We write the pair interaction in terms of the distance squared in order to be compatible with the notion  of {\em universal optimality} discussed below.  See \cite{HSS_2014} and \cite{HSSS_2017} for constructions of periodic potentials from interactions $f$ without rapid decay.  \\

The periodization of Gaussian potentials $f_a(r^2):=\exp(-ar^2)$ for $a>0$ leads to a type of lattice theta function (cf. \cite[Chapter 10]{MEbook}) and plays a central role in our analysis.  For convenience,  we write   $F_{a,\Lambda}:= F_{f_a,\Lambda}$ or just $F_a$ when the choice of lattice is unambiguous.  
\begin{definition}\label{UniOptDef} Let $\Lambda$ be a lattice in $\R^d$. \begin{itemize}
    \item[(a)] 
We say that a  configuration $\omega_n\subset \R^d$ is \textit{$\Lambda$-universally optimal} if it is $F_{a,\Lambda}$-optimal for all $a>0$ (cf.,   \cite{Cohn_Kumar_2007}).

\item[(b)] We say  $\Lambda$ is \textit{universally optimal} if for any sublattice 
$\Phi\subseteq \Lambda$, the configuration\footnote{A {\em fundamental domain} for a group $G$ acting on a set $X$  is a subset of $X$ consisting of exactly one point from each $G$-orbit. Note that $X/G$ will be used to denote both a fundamental domain and the set of $G$ orbits in $X$.  In particular, if $\Omega_\Phi$ is a choice of fundamental domain for $\R^d/\Phi$, then $\Lambda \cap \Omega_\Phi$ is a choice of fundamental domain for $\Lambda/\Phi$ in \eqref{genConfig}. The {\em index} of $\Phi$ in $\Lambda$ is   $[\Lambda,\Phi]:=|\Lambda/\Phi|$.}  \begin{equation}\label{genConfig}
 \omega(\Phi,\Lambda):=   \Lambda/\Phi
\end{equation}is $\Phi$-universally optimal; i.e., if the index of $\Phi$ in $\Lambda$ is  $n$, then    any choice of representatives $\Lambda/\Phi$ for $\Lambda/\Phi$ satisfies $E_{F_{a,\Phi}}(\omega_n)=\mathcal{E}_{F_{a,\Phi}}(n)$ for all $a>0$.
\end{itemize}
\end{definition}
If $\omega_n$ is $\Lambda$-universally optimal, then it follows from  a theorem of Bernstein \cite{Bernstein_1929} (see \cite{Cohn_Kumar_2007},\cite{MEbook}) that $\omega_n$ is $F_{f,\Lambda}$-optimal for any $f$ with $d$-rapid decay  that is  completely monotone on $(0,\infty)$.\footnote{Recall that a function $g$ is completely monotone on an interval $I$ if $(-1)^ng^{(n)}\geq 0$ on $I$ for all positive integers $n$.}. Furthermore,  taking $a\to \infty$, it follows that a universally optimal configuration must also form the centers of an optimal sphere packing.

 As discussed in Appendix Section \ref{notionsofunivopt}, it follows from classical results of Fisher \cite{Fisher_1964}  that a lattice is universally optimal in the sense of Definition~\ref{UniOptDef} if and only if it is  universally optimal in the    sense of Cohn and Kumar  \cite{Cohn_Kumar_2007} (also see \cite{CKMRV_2022})  which we review at the end of this section.

 We further show that to establish the universal optimality of $\Lambda$, it is sufficient to prove that there is a  sublattice $\Phi\subseteq \Lambda$ such that $\omega(m\Phi,\Lambda)$ is $m\Phi$-universally optimal for infinitely many $m\in \N$.  
 Observing that the notion of lattice universal optimality in Definition~\ref{UniOptDef} is scale-invariant, it is sufficient to prove the $\Phi$-universal optimality of configurations of the form  
 $\omega(\Phi,\frac{1}{m}\Lambda)= \left(\frac{1}{m}\Lambda\right)\cap\Omega_\Phi$
 for a  sublattice $\Phi$ of a lattice $\Lambda$.
\\

Recently it was shown in \cite{CKMRV_2022} that the $E_8$ and Leech lattices are universally optimal in dimensions 8 and 24, respectively. 
It was also shown in \cite{Cohn_Kumar_2007} that $\mathbb{Z}$ is universally optimal in $\R$.
These 3 cases are the only proven examples of universally  optimal (in the sense of Cohn and Kumar) configurations in $\R^d$. 
However, it was  conjectured in \cite{Cohn_Kumar_2007} that the hexagonal $A_2$ lattice
\[
A_2:=\begin{bmatrix}
   1& 1/2\\
  0& \sqrt{3}/2
    \end{bmatrix}\mathbb{Z}^2,
    %=\begin{bmatrix} u_0 & u_1\end{bmatrix}\mathbb{Z}^2.
\] 
%with   fundamental domain $\Omega_0$ generated by  $u_0:=[1,0]^T, u_1:=[1/2, \sqrt{3}/2]^T$.
 is universally optimal in $\R^2$. Though $A_2$ has long been known to be optimal for circle packing (see \cite{Toth_1940}) and was proved to be universally optimal among lattices in \cite{Montgomery_88}, its conjectured universal optimality among all infinite configurations (of fixed density) surprisingly remains open.  
 
 The  proofs of universal optimality for $\Z$, $E_8$, and the Leech lattice given in  \cite{Cohn_Kumar_2007} and \cite{CKMRV_2022} are based on so-called ``linear programming bounds" originally developed in the context of coding theory for point configurations on the $d$-dimensional sphere 
(e.g., see \cite{Delsarte_1977}, \cite{Levenshtein}, \cite{Yudin_92}) 
 and extended to bounds for the energy and sphere-packing density of point configurations in $\R^d$ in (e.g., see \cite{Cohn_2003}, \cite{Cohn_Kumar_2007}, and \cite{CKMRV_2022}).

In Section~\ref{sec2},  we formulate linear programming bounds (see Proposition~\ref{linprogbounds}) for  lattice periodic configurations in $\R^d$, find sufficient conditions to permit a certain polynomial structure (see Theorem~\ref{genlinpoly}) and develop conditions for the $F$-optimality of  configurations of the form \eqref{genConfig} in terms of polynomial interpolation (see Corollary~\ref{mukd}).

We apply this framework  to the following four families of configurations (arising from scalings of $A_2$) using the notation of \eqref{genConfig}. Taking $L:=\Z\times \sqrt{3}\Z$, we set
\begin{enumerate}
    \item[(a)] $\Phi=A_2$ and $\omega^*_{m^2}:=\omega(\Phi,\frac{1}{m}A_2),$
    %conjectured to be $A_2$-universally optimal.
    \item[(b)]  $\Phi=L$ and $\omega^*_{2m^2}:=\omega(\Phi,\frac{1}{m}A_2),$ 
    %conjectured to be $L$-universally optimal.
     \item[(c)] $\Phi= \sqrt{3}R_{\pi/6}A_2$  and    $\omega^*_{3m^2}:=\frac{1}{\sqrt{3}}R_{-\pi/6}\omega(\Phi,\frac{1}{m}A_2),$ where $R_{\theta}$ denotes   rotation by $\theta$,
     %conjectured to be $A_2$-universally optimal (note $\frac{1}{\sqrt{3}}R_{-\pi/6}\Phi=A_2$).    
     \item[(d)] $\Phi= \sqrt{3}R_{\pi/6}L$ and   $\omega^*_{6m^2}:=\frac{1}{\sqrt{3}}R_{-\pi/6}\omega(\Phi,\frac{1}{m}A_2).$ 
     %conjectured to be $L$-universally optimal (note $\frac{1}{\sqrt{3}}R_{-\pi/6}\Phi=L$).
\end{enumerate}

The $A_2$-universal optimality of $\omega^*_{m^2}$ and $\omega^*_{3m^2}$ as well as the $L$-universal optimality of $\omega_{2 m^2}^*$ and $\omega_{6 m^2}^*$ would follow immediately should the conjectured universal optimality of $A_2$ be true.   Conversely, as discussed above, the universal optimality of $A_2$ would follow if analogous results are established for any of these four families   for infinitely many $m\in \mathbb{N}$.\\

The proofs of universal optimality of two of the base cases, $\omega_2^*$ and $\omega_3^*$, follow immediately from results on theta functions, some classical and some from \cite{Baernstein_1997} (cf. \cite{Su_2015} or \cite{Faulhuber_2023} for proofs in the context of periodic energy). Our main results are the universal optimality of the next two cases, $\omega_4^*$ and $\omega_6^*$, with proofs utilizing the linear programming bounds.

\begin{theorem}\label{mainThm}
The configurations $\omega_{4}^*$ and $\omega_{6}^*$ are $A_2$ and $L$-universally optimal, respectively.
\end{theorem}
  
\begin{figure}
    \centering
    \begin{minipage}{0.45\textwidth}
        \centering
        \includegraphics[width=0.9\textwidth]{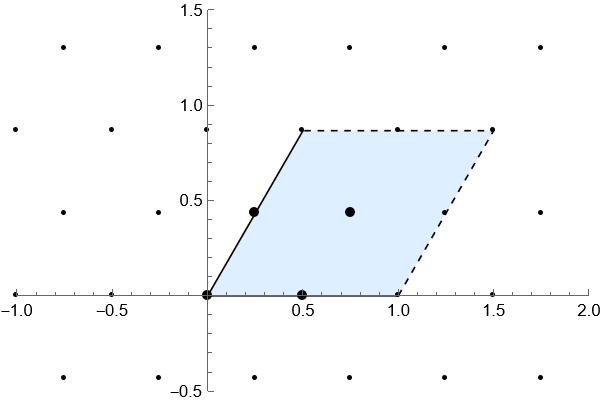} % first figure itself
        \caption{The 4-point  $A_2$-universally optimal configuration   $\omega^*_4$.}
    \end{minipage}\hfill
    \begin{minipage}{0.45\textwidth}
        \centering
        \includegraphics[width=0.9\textwidth]{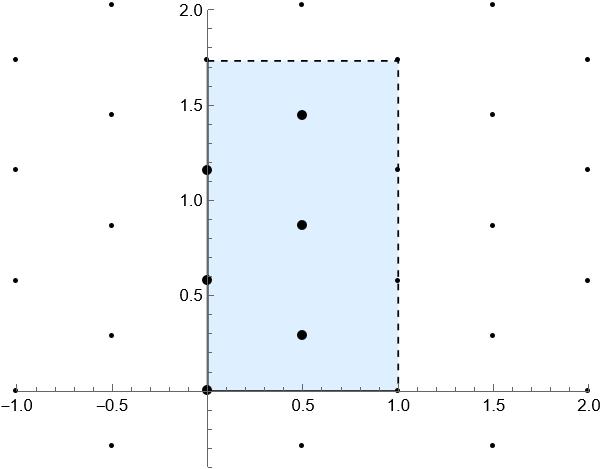} % second figure itself
        \caption{The 6-point  $L$-universally optimal configuration   $\omega^*_6$.}
    \end{minipage}
\end{figure}
%$A_2$ decomposes as $L \sqcup ([1/2,\sqrt{3}/2]^T +L)$. Let $\Omega_L$ be the fundamental region of $L$ generated by $u_0$ and $\ell_1=[0,\sqrt{3}]^T$. 

We can rephrase Theorem \ref{mainThm} in terms of the energies of infinite configurations, for which we follow the notation of  \cite{CKMRV_2022}.
Let $B(x,r)$ be the ball of radius $r>0$ centered at $x$. If $C$ is an infinite multiset in $\R^d$ such that every ball intersects finitely many points, we call it an \textit{infinite configuration}. Define $C_r:C \cap B(0,r) $ and  the \textit{density of $C$} as
\[
\lim_{r\rightarrow \infty} \frac{\vert C_r \vert}{{\rm Vol}(B(0,r))},
\]
assuming the limit exists and is finite. 
%Similarly to above, for a lower semi-continuous map $f:[0,\infty)\rightarrow [0,\infty]$ of $d$-rapid decay, we define the $f$-energy of an $n$-point configuration $\omega_n=\{ x_1,\dots,x_n\}\subseteq \R^d$ as 
%\[
%E_f(\omega_n):=\sum_{\substack{1\leq i,j\leq n \\ i\neq j}}f(\vert x_i-x_j\vert^2).
%\]
%and the infimum over all $n$-point configurations contained in some set $X$ as $\mathcal{E}_f(n,X)$, where we extend the definition to $n\not \in \N$ via linear interpolation. Configurations achieving this infimum are called $f$-optimal on $X$.
Then for a configuration $C$ of density $\rho$, the \textit{lower $f$-energy of $C$} is 
\[
E^{l}_f(C):=\liminf_{r\rightarrow \infty} \frac{E_f(C_r)}{\vert C_r\vert}.
\] 
If the limit exists, we'll write it as $E_f(C)$ and call it the \textit{average $f$-energy of $C$}. A configuration $C'$ of density $\rho$ is \textit{$f$-optimal} if 
\[
E_{f}(C') \leq E^{l}_{f}(C) 
\]
for every configuration $C$ of density $\rho$, and  \textit{universally optimal} if it is $f_a$-optimal for all $a>0$.
Similarly, a configuration $C'$ is universally optimal among $S$ if we further restrict $C$ to elements of $S$. 
We'll also say an infinite configuration $C$ is an \textit{$N$-point $\Lambda$-periodic configuration} if 
\[
C=\cup_{i=1}^{N} (x_i+ \Lambda)
\]
for some set of \textit{representatives} $\omega_N^C=\{x_1,\dots,x_n\}$. 
Then we have the following connection between the $f$-energy of $C$, and the $F_{f,\Lambda}$ energy of $\omega_N^C$ (where we abuse notation by identifying the one-variable function $f$ with the map from $\R^d\rightarrow (-\infty, \infty]$ taking $x$ to $f(\vert x \vert^2)$. 
\begin{proposition}
\label{periodictoaverage}
Let $C$ be an $N$-point $\Lambda$-periodic configuration with $\omega_N^C$ a set of representatives, and $f$ has  $d$-rapid decay. 
%Recall $F_{a,\Lambda}$ is the $\Lambda$-periodization of $f$. 
Then $E_f(C)$ exists and 
\[
E_f(C)= \frac{1}{N}\left(E_{F_{f,\Lambda}} (\omega_{N}^{C})+ N\sum_{0\neq v\in \Lambda} f(v)\right).
\]
\end{proposition}
Thus, Theorem~\ref{mainThm} can be restated as follows: $A_2/2$ is universally optimal among all 4-point $A_2$-periodic configurations, and a rotation and scaling of $A_2$ is universally optimal among all 6-point $L$-periodic configurations.

As motivation for studying the above periodic energy problems arising from the $A_2$ lattice, we review in Section~\ref{Zunivopt} a proof of the universal optimality of $\Z$, which proceeds through the analogous periodic approach.
As far as the authors are aware, this proof is the simplest route to showing the universal optimality of $\Z$.
The main difference between the $\Z$ and $A_2$ cases is the presence of a simple error formula for univariate hermite interpolation that is not available for general bivariate interpolation.
As a result, the most difficult portions of our proof of Theorem \ref{mainThm} involve showing that our proposed interpolants stay below their potentials on the relevant domains. \\
Moreover, we find the small cardinality examples of the main theorem interesting regardless of whether the periodic energy approach leads to a proof of the universal optimality of $A_2$. Such optimality results can often be surprisingly difficult, even in the case of simple potentials with configurations restricted to nice spaces. For example, the case of proving optimality for the Riesz potentials among configurations on $S^2$ is notoriously difficult even for 5 points  (rigorously handled almost completely with computer-assisted calculations in \cite{Schwartz_2016}) and remains open for $n\not \in  \{2,3,4,5,6,12\}$.

%Finally, we present numerical evidence for a proposed interpolant that would verify the $L$-universal optimality of $\omega_8$. 
%%%%%%%%%%%%%%%%%%%%%%%%%%%%%%%%%%%%%%%%%%%%%%%%%%%%%%%%%%%%%%%%%%%%%%%%%%%%%%%%%%%%%%%%%

%The optimality of these 4 and 6-point configurations extend existing results on the optimality of 2 and 3-point configurations arising from $A_2$, pictured below  (cf. \cite{Baernstein_1997} for the key result on the minimum of $F_{A_2,a}$ and \cite{Su_2015} for a proof in the context of periodic energy). The 3-point configuration can be obtained by intersecting $A_2$ with a fundamental domain of the sublattice $\sqrt{3}R_{\pi/6}A_2$. Similarly, an optimal 2-point configuration can be obtained by removing any     point from the optimal 3-point configuration.  
%$A_2$ decomposes as $L \sqcup ([1/2,\sqrt{3}/2]^T +L)$. Let $\Omega_L$ be the fundamental region of $L$ generated by $u_0$ and $\ell_1=[0,\sqrt{3}]^T$. Then Let $C$ be an $A_2$-invariant set. If $\lvert C \cap \Omega^*\rvert=n$, then $\lvert C\cap \Omega_l\rvert=2n$ and we have the following correspondence:
%\begin{observation}
%If $C$ is $A_2$ periodic and $2n$-point $L$ optimal, in the sense of minimizing energy on $\Omega_L$ for $F$, the $L$-periodization of some function $f$, then $C$ is also $n$-point $A_2$ optimal in that it minimizes energy on $\Omega^*$ for the $A_2$-periodization of $f$. The converse fails to hold.
%\end{observation}

\section{Lattices and Linear Programming Bounds for Periodic Energy \label{sec2}} 
\subsection{Preliminaries: Lattices and Fourier Series}
\label{lattbasics}
We first gather some basic definitions and properties of lattices in $\R^d$. 
\begin{definition}
Let  $\Lambda\subset \R^d$.
  \begin{itemize} 
  \item  $\Lambda$ is a {\em lattice in $\R^d$}  if 
   $\Lambda:= V \mathbb{Z}^d=\left\{\sum_{i=1}^da_iv_i\mid a_1,a_2,\ldots,a_d\in \mathbb{Z}\right\}$ for some nonsingular $d\times d$ matrix $V$  with columns $v_1, \ldots v_d$. We refer to $V$ as a {\em generator} for $\Lambda$.   
   \item Once a choice of generator $V$ is specified, we let  $\Omega_\Lambda:=V\mathbb[0,1)^d$  denote the parallelepiped {\em fundamental domain} for $\Lambda$. The {\em co-volume of $\Lambda$} defined by $\left|\Lambda\right|:=|\det V|$ is the volume  of $\Omega_\Lambda$  which is, in fact, the same for any  Lebesgue measurable fundamental domain for $\R^d/\Lambda$ where $\Lambda$ acts on $\R^d$ by translation.
   \item The {\em dual lattice} $\Lambda^*$ of a lattice $\Lambda$ with generator $V$ is the lattice generated by $V^{-T}=(V^T)^{-1}$ or, equivalently,
   $\Lambda^*:=\{v\in \R^d\mid w\cdot v\in \mathbb{Z} \text{ for all } w\in \Lambda    \}.$
   \item We denote by $S_\Lambda$ the {\em symmetry group of $\Lambda$} consisting of isometries on $\R^d$ fixing $\Lambda$  and   denote by $G_\Lambda$ the subgroup of $S_\Lambda$ fixing the origin (and thus can be considered as elements of the orthogonal group $O(d)$).     Note that $G_\Lambda=S_\Lambda/\Lambda$ where we identify $v\in\Lambda$ with  the translation $\cdot+v$.   Further, note that $G_{\Lambda^*}=G_{\Lambda}$  since elements of $O(d)$ preserve inner products. 
   \end{itemize}
\end{definition} 

 Let $\Lambda$ be a lattice in $\R^d$ with generator $V$  and fundamental domain $\Omega_\Lambda$.  We
 let $L^2(\Omega_\Lambda)$ denote the Hilbert space of complex-valued $\Lambda$-periodic functions on $\R^d$ with  inner product $\langle f,g\rangle=\int_{\Omega_\Lambda}f(x)\overline{g(x)}\, dx$. Then $\{e^{2\pi i v\cdot x}\mid v\in \Lambda^*\}$ 
forms an orthogonal basis of $L^2(\Omega_\Lambda)$ yielding the Fourier  expansion of a function $g\in L^2(\Omega_\Lambda)$:
\begin{equation}
    \label{FExp}
g(x)=\sum_{v\in \Lambda^*} \hat{g}_v e^{2\pi i v\cdot x}
\end{equation}
with Fourier coefficients   $\hat{g}_v:=\frac{1}{|\Lambda|}\int_{\Omega_\Lambda} g(x)\, dx$ for $v\in \Lambda^*$
where equality (and the implied unconditional limit on the right hand side) holds in $L^2(\Omega_\Lambda)$.  Of course, elements of $L^2(\Omega_\Lambda)$ are actually equivalence classes of functions.  If $g\in L^2(\Omega_\Lambda)$ contains an element of $C(\R^d)$, then we identify $g$ with its continuous representative and write $g\in L^2(\Omega_\Lambda)\cap C(\R^d)$.  
As will be the case in our applications, if $g\in L^2(\Omega_\Lambda)$ is such that $\sum_{v\in \Lambda^*}|\hat{g}_v|<\infty$, then the right-hand side of \eqref{FExp} converges uniformly and unconditionally to $g$ and so $g\in L^2(\Omega_\Lambda)\cap C(\R^d)$ and \eqref{FExp}   holds pointwise for every $x\in \R^d$. 

\begin{definition}
We say that  $g\in L^2(\Omega_\Lambda)$ is {\em conditionally positive semi-definite (CPSD)} if the Fourier coefficients
$\hat{g}_v\ge 0$ for all $v\in\Lambda^*\setminus\{0\}$ and $\sum_{v\in\Lambda^*}\hat{g}_v<\infty$; and we say that a CPSD $g$ is {\em positive semi-definite (PSD)} if $\hat{g}_0\ge 0$.\footnote{If $g$ is PSD in the above sense, then for any configuration $\omega_n=(x_1,\ldots,x_n)$ the matrix $G=(g(x_i-x_j))$ is positive semi-definite in the sense that $v^TGv\ge 0$ for any $v$ whose components sum to 0.  Conversely, Bochner's Theorem shows that any $g$ with this property is PSD in our sense.}    
\end{definition}
     Note that the product of two PSD functions in $L^2(\Omega_\Lambda)$ is PSD.  

\subsection{Lattice symmetry, symmetrized basis functions, and polynomial structure}
\label{latticesymmetry}
Let $\Lambda$ be a lattice in $\R^d$, $f:[0,\infty)\to[0,\infty]$ have $d$-rapid decay, and  $\sigma\in G_\Lambda$. Since  $\sigma^{-1}\in G_\Lambda$ and $\sigma$ is an isometry, we have  
$$
F_{f,\Lambda}(\sigma x) =\sum_{v\in \sigma^{-1}\Lambda} f(\lvert \sigma x + \sigma v\rvert^2)
 =\sum_{v\in \Lambda} f(\lvert x + v\rvert^2) 
 =F_{f,\Lambda}(x).
$$  Then
$F_{f,\Lambda}$ is also $\Lambda$-periodic and we obtain: 
\begin{proposition}
\label{G2invariance}
Suppose $f:[0,\infty)\to[0,\infty]$ has $d$-rapid decay and  $\Lambda$ is a lattice in $\R^d$. Then for all $\sigma\in G_\Lambda$, $v\in\Lambda$ and $x\in \R^2$, we have $F_{f,\Lambda}(\sigma x+v)=F_{f,\Lambda}(x)$ showing that $F_{f,\Lambda}$ is $S_\Lambda$-invariant.   
\end{proposition}

We next recall that   $g \in L^2(\Omega_\Lambda)$ is $\sigma$-invariant for $\sigma\in G_\Lambda$ if and only if the Fourier coefficients of $g$ are $\sigma$-invariant, as described in the next proposition. 

\begin{proposition}
\label{groupCfs}
Suppose  $g \in L^2(\Omega_\Lambda)$ and $\sigma\in G_\Lambda$.  Then $g(\sigma x)=g(x)$ for a.e. $x\in\R^d$
  if and only if $\hat{g}_{\sigma v}=\hat{g}_{v}$ for all $v\in \Lambda^*$.
\end{proposition}
\begin{proof}
Since $\sigma^{-1}\in G_{\Lambda^*}=G_\Lambda$, we have
$$g(\sigma x)=\sum_{v\in \Lambda^*} \hat{g}_{v} e^{2\pi i v\cdot (\sigma x)}=\sum_{v\in \sigma^{-1}\Lambda^*} \hat{g}_v e^{2\pi i (\sigma v)\cdot (\sigma x)}=\sum_{v\in \Lambda^*} \hat{g}_{\sigma v} e^{2\pi i v\cdot x}.$$ The proposition then follows from uniqueness properties of the Fourier expansion.
\end{proof}

Let  $\Gamma$ be a subgroup of  $ G_\Lambda$.  
For $v\in \Lambda^*$,
let $\CC_v^\Gamma$ be the $\Lambda$-periodic function defined by 
\begin{equation}\label{CCdef}
 \CC_v^{\Gamma}(x) :=
    \frac{1}{|\Gamma|}
    \sum_{\sigma\in \Gamma}
    e^{2\pi i (\sigma v)\cdot x}=
    \frac{1}{|\Gamma(v)|}
    \sum_{v'\in \Gamma(v)}
    e^{2\pi i   v'\cdot x}, 
    \quad x\in \R^d.
\end{equation}
where $\Gamma(v)$ denotes the  orbit $\Gamma(v)=\{\sigma v \mid \sigma\in \Gamma\}.$    We write $\CC_v$ for $\CC_v^\Gamma$ when $\Gamma$ is unambiguous.    
%Since clearly $\CC_{\sigma v}=\CC_{v}$ by definition, we will typically restrict our consideration of $v\in \Phi^*$ to fundamental domains of $\Phi^*$ mod the action of  $G_\Phi$, denoted $\Phi^*/ G$.
If $g\in L^2(\Omega_\Lambda)$ and $g$
is $G_\Lambda$-invariant (i.e., if $g(\sigma \cdot)=g$ for all $\sigma\in G_\Lambda$), then we may rewrite \eqref{FExp} as
\begin{equation}
    \label{FExpG}
g(x)=\sum_{v\in \Lambda^*/\Gamma} |\Gamma(v)| \,\hat{g}_v \,\CC_v^\Gamma(x).
\end{equation}

%\newcommand{\id}{{\rm id}}

 %Let $\id$ denote the identity function on $\R^d$ and $\Gamma=\{\id, -\id\}$.   Then for $v\in \R^d$, $$\CC^\Gamma_v(x)=\cos(2\pi v\cdot x). $$

\bigskip

  We next consider the case of a {\em rectangular lattice} by which we mean a lattice of the form $\Lambda_R=(a_1\Z)\times \cdots (a_d\Z)$ with $a_1,\ldots,a_d>0$ (also referred to as an \textit{orthorhombic lattice}).  The symmetry group of a rectangular lattice in $\R^d$
  contains the subgroup $H$ of order $2^d$ generated by the coordinate reflections \begin{equation}
  \label{Rjdef}    
  R_j(x_1,\dots,x_j,\ldots,x_d)=(x_1,\dots,-x_j,\ldots,x_d),\qquad j=1,2,\ldots, d.
  \end{equation}
  Let   $v\in \Lambda^*=(1/a_1)\Z\times \cdots (1/a_d)\Z$, and note that $v=(k_1/a_1,k_2/a_2,\dots,k_d/a_d)$ for some $k_1,\dots,k_d \in \Z$.  A straightforward induction on $d$ gives  
        \begin{equation}
    \CC_v^{H}(x)=  \prod_{i=1}^{d} \cos(2\pi k_i x_i /a_i)
   = \prod_{i=1}^{d} T_{\vert k_i \vert}(\cos(2\pi x_i /a_i)).
  \end{equation}
Recall   the $\ell$th Chebyshev polynomial of the first kind defined by $\cos(\ell \theta)=T_\ell(\cos \theta)$ for $\ell=0,1,2,\ldots$.  We then have  the following proposition.
\begin{proposition}\label{rectvars}
Let $\Lambda_R=(a_1\Z)\times \cdots (a_d\Z)$ with $a_1,\ldots,a_d>0$.  If $v\in \Lambda_R^*$, then
$v=(k_1/a_1,k_2/a_2,\dots,k_d/a_d)$ for some $k_1,\dots,k_d \in \Z$ and
    \begin{equation}    \CC_v^{H}(x)
 = \prod_{i=1}^{d} T_{\vert k_i \vert}(t_i),
  \end{equation}
where $t_i=\cos(2\pi x_i /a_i)\in [-1,1]$ for $i=1,2,\ldots,d$. 
\end{proposition}

We next deduce a polynomial structure for $\CC_v^{\Lambda}$  for lattices $\Lambda$ that are invariant under the  coordinate reflections $R_j$; i.e., such that $H\subseteq G_\Lambda$.

\begin{proposition}
\label{polystruct}
Let $\Lambda\subseteq \R^d$ be a lattice such that $H\subseteq G_\Lambda$.  Then $\Lambda$ contains a rectangular lattice $\Lambda_R= (a_1\Z)\times \cdots \times (a_d\Z)$ and  
the function $\CC^{G_\Lambda}_{v}(x)$ is a polynomial in the  variables $t_j=\cos(2\pi x_j/a_j)$ for $j=1,2,\ldots,d$ and any $v\in \Lambda^*$. 
\end{proposition}

\begin{proof}
We first show that $\Lambda$ must contain some rectangular  sublattice (i.e., of the form $\Lambda_R=(a_1\Z)\times \cdots \times(a_d\Z)$).  Since $\Lambda$ is full-rank, for each   $j=1,2,\ldots,d$, there is some $w^j\in\Lambda$ such that $a_j:=2w^j\cdot e^j\neq 0$ where $e^j$ denotes the  $j$-th  coordinate unit vector.   Then $a_je^j=w^j-R_jw^j\in \Lambda$,
and so the rectangular lattice $(a_1\Z)\times \cdots \times(a_d\Z)$ is a sublattice of $\Lambda$. 

    Let $v\in \Lambda^*$. Since $\Lambda_R \subseteq \Lambda$, $\Lambda^*\subseteq \Lambda_R^*$, so $v\in \Lambda_R^*$. Let $C=\{\sigma_1,\dots,\sigma_{[G_\lambda:H]}\}$ be a set of right coset representatives of $H$ in $G_\Lambda$, so that $\vert C\vert \vert H\vert=\vert G\vert$.
    Then we have 
   \begin{equation}
       \begin{split}
        \CC^{G_\Lambda}_{v}&=\frac{1}{\vert G_\Lambda \vert}\sum_{g\in G_\Lambda}e^{2\pi i gv\cdot x}\\
        &= \frac{1}{\vert C\vert } \sum_{\sigma \in C} \frac{1}{\vert H \vert}\sum_{h\in H} e^{2\pi i h\sigma v\cdot x } \\
        &= \frac{1}{\vert C\vert } \sum_{\sigma \in C} \CC^H_{\sigma v}.
    \end{split}
    \end{equation} 
 Proposition~\ref{rectvars} implies  $\CC^H_{\sigma v}$ is polynomial in the variables $t_j=\cos(2\pi x_j/a_j)$ and thus  so is $\CC^{G_\Lambda}_{v}$.     
\end{proof}
With $\Lambda$ and $\Lambda_R$ as in Proposition \ref{polystruct}, we consider the change of variables 
\begin{equation}
\label{littlettransformation}
t_i := \cos(2\pi x_i/a_i), \qquad i=1,...,d. 
\end{equation}
We then let $T_{a_1,...,a_d}: \R^d \rightarrow \R^d$ be defined by 
\begin{equation}
\label{Ttransformation}
T_{a_1,\dots,a_d}(x_1,...,x_d):=(t_1,\dots,t_d).
\end{equation}

For any $\Lambda_R$-periodic function $h$ with $H$-symmetry, $\tilde{h}$ will refer to the function defined on $[-1,1]^d$ by 
\[
\tilde{h}(t)=h\left(\frac{a_1 \arccos t_1}{2\pi},\dots,\frac{a_d\arccos t_2}{2\pi}\right),
\]
which ensures $\tilde{h}(t)=h(x)$.  We say that $\tilde{h}$ is (C)PSD   if $h$ is (C)PSD.

It follows by Proposition \ref{polystruct} that the maps \begin{equation}\label{tildeP}P^{\Phi}_v:=\tilde{\CC}^{G_\Phi}_{v}, \qquad v\in \Phi^*,\end{equation}  are polynomials in the variables $t_1,\dots, t_d$.   It then follows that the collection of polynomials $\{P_v^{\Phi} \mid v\in \Phi^*/G_\Phi\}$ is orthogonal with respect to the measure $(1-t_1^2)^{-1/2}\cdots (1-t_d^2)^{-1/2}dt_1\, \cdots\, dt_d$ on $[-1,1]^d$. Furthermore, $\tilde{h}$ is CPSD if and only if its expansion in terms of these polynomials has coefficients that are non-negative and summable.  
%In our proof of Theorem~\ref{mainThm} we shall only need this characterization of CPSD for polynomials $\tilde h$ but state our linear programming bounds in the more general case. 

We shall also write $P_v$ when the choice of $\Phi$ is clear. Similarly,
the $T_{a_1,\dots,a_d}$ image of any subset $D\subseteq \R^d$ will be denoted $\tilde{D}$. In any case where we do so, the choice of rectangular lattice (and hence the choice of $a_i$'s will be clear).\\

%The proof follows by induction. In the base case, we have $v=k/a$ so 
%\[
%\CC^{\Gamma}_{v}(x)=\frac12(e^{2\pi i k/a}+e^{-2\pi i k/a})= \cos(2\pi \vert k \vert/ a)= T_{\vert k\vert }(\cos(2\pi k/a)),
%\]
%merely by the definition of the Chebyshev polynomials. Now in general, let $\Gamma_1\leq \Gamma$ be the subgroup of maps fixing the $d$th coordinate, so 
%$G= \Gamma_1\sqcup \Gamma_1 \sigma'$ where $\sigma'$ is the map taking $(x_1,\dots,x_d)\rightarrow (x_1,x_2,\dots,-x_d)$. 
%Also let $p_{d-1}$ be the map projecting $x$ onto its first $d-1$ coordinates, and likewise for $\sigma \in \Gamma$, $\sigma_{d-1}$ will denote the map on $\R^{d-1}$ by restricting $\sigma$ to acting only the first $d-1$ coordinates. Note that $\{ \sigma_{d-1}: \sigma \in \Gamma_1\}$ is the set of coordinate symmetries for $\R^{d-1}$. Thus, we obtain 

%  \begin{align}
%    \CC_v^\Gamma(x)&=  \frac{1}{\vert \Gamma_1\vert}\sum_{\sigma \in \Gamma_1}\frac{e^{2\pi i \sigma v \cdot x}+ e^{2\pi i \sigma \sigma ' v \cdot x}}{2}\\
%    &= \frac{1}{\vert \Gamma_1\vert}\sum_{\sigma \in \Gamma_1} \cos (2\pi \vert k_d\vert x_d /a_d) e^{2\pi i \sigma_{d-1}p_{d-1(v)}\cdot p_{d-1}(x)}    \\
 %   &= \cos (2\pi \vert k_d\vert x_d /a_d) \frac{1}{\vert \Gamma_1\vert}\sum_{\sigma \in \Gamma_1} e^{2\pi i \sigma_{d-1}p_{d-1(v)}\cdot p_{d-1}(x)}\\
%&=T_{ \vert k_d\vert} (\cos(2\pi x_d /a_d)) \prod_{i=1}^{d-1} T_{\vert k_i \vert}(\cos(2\pi x_i /a_i)
%  \end{align}
%  as desired, and we have applied the inductive hypothesis to obtain the final equality. 

\subsection{Linear Programming Bounds for Periodic Energy}
\label{lpboundssection}

If $g\in L^2(\Omega_\Lambda)$ is CPSD and $\omega_n$ is an arbitrary $n$-point configuration in $\R^d$, then the following fundamental lower bound holds: 
\begin{equation}
    \begin{split}
E_{g}(\omega_n)&=\sum_{x\neq y \in \omega_n} g(x-y) = -ng(0)+\sum_{x,y\in \omega_n} g(x-y)\\
&=-ng(0)+\sum_{v\in \Lambda^*}  \hat{g}_v \sum_{x,y\in \omega_n}  e^{2\pi i v\cdot x}e^{-2\pi i v\cdot y} \\
&=-ng(0)+\sum_{v\in \Lambda^*} \hat{g}_v \left\lvert \sum_{x\in \omega_n} e^{2\pi i v\cdot x} \right\rvert^2 
\\
&\geq n^2\hat{g}_0-ng(0) \label{LPg}.
\end{split}
\end{equation}
For $v\in \R^d$  we refer to 
$$M_v(\omega_n):= \sum_{x\in \omega_n} e^{2\pi i v\cdot x}, $$ as the {\em $v$-moment of $\omega_n$}.  Note that equality holds in \eqref{LPg} if and only if 
\begin{equation}\label{LPgEq}
    \hat{g}_v  M_v(\omega_n) = 0, \qquad v\in \Lambda^*\setminus\{0\}.
\end{equation}
The next proposition   follows immediately from \eqref{LPg} and the condition \eqref{LPgEq} for equality in \eqref{LPg}.  The calculations in \eqref{LPg} are similar to the proof of the linear programming bounds for energy found in \cite[Proposition 9.3]{Cohn_Kumar_2007}  and is closely related to Delsarte-Yudin energy bounds for spherical codes (cf. \cite[Chapters 5.5 and 10.4]{MEbook}).  

\begin{proposition}
\label{linearprogramming}
Let $F:\R^d\to [0,\infty]$ be $\Lambda$-periodic, and suppose $g\in L^2(\Omega_\Lambda)$ is CPSD  such that $g\le F$. 
 Then for any $n$-point configuration $\omega_n $, we  have 
\begin{equation}
\label{linprogbounds}
E_F(\omega_n)\geq E_g(\omega_n) \geq n^2\hat{g}_0-ng(0)
\end{equation}
with equality holding throughout \eqref{linprogbounds} if and only if the following two conditions hold:
\begin{enumerate}
\item[(a)] $g(x-y)=F(x-y)$ for all $x\neq y \in \omega_n$,
\item[(b)]   $\hat{g}_v M_v(\omega_n)=0$, for all $v\in \Lambda^*\setminus \{ 0\}$.
\end{enumerate}
If (a) and (b) hold, then $E_F(\omega_n)=\mathcal{E}_F(n)$. 
\end{proposition}

\begin{remark}
\label{GLambdainvariance}
  If $F:\R^d\to [0,\infty]$ is $\Lambda$-periodic and $G_\Lambda$-invariant and 
$g\in L^2(\Omega_\Lambda)$ is CPSD such that $g\le F$, then the $S_\Lambda$-invariant
function 
\begin{equation*}
    g^{\rm sym}(x):=\frac{1}{|G_\Lambda|}\sum_{\sigma\in G_\Lambda}g((\sigma v)\cdot x), \qquad x\in\R^d 
\end{equation*}
is also CPSD and satisfies $g^{\rm sym}\le F$.  Thus, we may restrict our search for functions $g$ to use in Proposition~\ref{linearprogramming} to those of the form given in \eqref{FExpG} in which case we only need verify the condition that $g\le F$ on the fundamental domain of the action of $S_\Lambda$ on $\R^d$. In particular, when $\Lambda=A_2$, we have the representative set
\[
\Delta_{A_2}:=\{ (x_1,x_2)\mid 0\leq x_1\leq \frac12,0\leq x_2\leq x_1/\sqrt{3}\}
\]
and when $\Lambda=L$, we'll consider the representative set $[0,1/2]\times[0,\sqrt{3}/2]$.
\end{remark}

\subsection{Moments for certain lattice configurations}
We consider moments of configurations $\omega(\Phi,\Lambda):=\Lambda\cap \Omega_\Phi$ obtained by restricting     a lattice $\Lambda$ to the fundamental domain of a sublattice $\Phi$.\footnote{We are aware of similar lattice computations in discrete harmonic analysis (e.g., see \cite{Li_Sun_Xu_2012}), but the authors could not find a reference for this exact result and so include a proof.} 
\begin{theorem}
\label{generalmoments}
Suppose $\Phi$ is a sublattice of a lattice $\Lambda$ in $\R^d$. Let  $\kappa:=|\omega(\Phi,\Lambda) |$ denote the index of $\Phi$ in $\Lambda$.   Then for $v\in \Phi^*$, we have
\begin{equation}\label{MomentForm}
    M_v(\omega(\Phi,\Lambda))=\begin{cases} \kappa,&  v\in \Lambda^*,\\
    0,& \text{otherwise.}
    \end{cases}
\end{equation}
Furthermore, if $G_\Phi\subset G_\Lambda$, then for any $v\in \Phi^*$ and $\sigma \in G_\Phi$, we have $M_{\sigma v}(\omega(\Phi,\Lambda))=M_v(\omega(\Phi,\Lambda))$.
\end{theorem}
\begin{proof}
Let  $\Lambda=V\Z^d$; i.e., $V$ is a generator for $\Lambda$.
 Since $\Phi$ is a sublattice of $\Lambda$, there is some integer $d\times d$ matrix $W$ such that $VW$ is a generator for $\Phi$. Then $W$  can be written in Smith Normal Form as $W=SDT$ where $S$ and $T$ are integer matrices with  determinant $\pm 1$ (equivalently, their inverses are also integer matrices) and $D$ is a diagonal matrix with positive integer diagonal entries $\lambda_1,\ldots, \lambda_d$.  It follows that $\widetilde V=VS$ is a generator for $\Lambda$ and $U=\widetilde V D$ is a generator for $\Phi$.  Choosing the fundamental domains $\Omega_\Lambda=\widetilde V[0,1)^d$ and $\Omega_\Phi=U[0,1)^d$  
 gives
 $$\omega(\Phi,\Lambda)=\{ \widetilde V j\mid j\in [0..\lambda_1]\times\cdots \times [0..\lambda_d]\}, $$ 
 where   $[0..p]:=\{0,1,2,\ldots,p-1\}$ positive integers $p$.
    Let $v\in\Phi^*$ so that $v=U^{-T}k={\widetilde V}^{-T}D^{-1}k$  for some $k=(k_1,k_2,\dots,k_d)\in \Z^d.$  Then $v\cdot ( \widetilde V j)=j\cdot (D^{-1}k)$ and so
 \begin{align*}
     M_v(\omega(\Phi,\Lambda))&=
     \sum_{j\in [0..\lambda_1]\times\cdots \times [0..\lambda_d]} e^{2\pi i j\cdot D^{-1}k}=
     \prod_{\ell=1}^d\left(\sum_{j_\ell=0}^{\lambda_\ell -1}e^{2\pi i \frac{j_\ell k_\ell}{\lambda_\ell}}\right)\\&=\begin{cases}
     \lambda_1\cdots\lambda_d, & k\in D\Z^d,\\
     0,& \text{otherwise},
 \end{cases}
 \end{align*}
 where we used the finite geometric sum formula in the last equality.  Noting that $\kappa=\lambda_1\cdots\lambda_d$ and that $v\in \Lambda^*$ if and only if $k\in D\Z^d$ establishes \eqref{MomentForm}.

 Finally, if   $\sigma\in G_\Phi$ and $G_\Phi\subset G_\Lambda$, then $\sigma v\in \Lambda^*$ if and only if $v\in \Lambda^*$ which completes the proof.
\end{proof}
We define the {\em moment index} of a configuration $\omega=\{x_1,\ldots,x_n\}\subset \R^d$  with respect to a lattice $\Phi\subset \R^d$ by 
\begin{equation}\label{Idef}
   \mathcal{I}_\Phi(\omega):=\{v\in\Phi^*\mid M_v(\omega)=0\}.
\end{equation}

 \begin{remark}
 Let $\Phi$ be a sublattice of a lattice $\Lambda\subset \R^d$  
and $m\in \N$.
 Since $(\frac{1}{m}\Lambda)^*=m\Lambda^*$,      Theorem~\ref{generalmoments} implies
\begin{equation}\label{MomentFormGEN}
  \mathcal{I}_\Phi(\omega(\Phi,\frac{1}{m}\Lambda))=\Phi^*\setminus (m\Lambda^*).
\end{equation}
 
 \end{remark}
 
\subsection{Lattice theta functions}

For $c>0$, the classical   Jacobi theta function of the third type, is defined by
\begin{equation}
\label{thetasmallaformula}
\theta(c;x):=\sum^{\infty}_{k=-\infty} e^{-\pi k^2 c}e^{2\pi i kx}, \qquad x\in \R.
\end{equation}

Via Poisson Summation on the integers, we have
\begin{equation}
\label{thetabigaformula}
\theta(c;x)=c^{-1/2} \sum^{\infty}_{k=-\infty} e^{-\frac{\pi(k+x)^2}{c}},
\end{equation}
and so, in terms of our earlier language for periodizing gaussians by lattices, 
\begin{equation}
F_{a, \mathbb{Z}}(x)= (a/\pi)^{1/2}\theta(\frac{\pi}{a};x).
\end{equation}
%We note that $\theta(c;x)$ is $\mathbb{Z}$-periodic and $\theta(c;x)=\theta(c;-x)$ for any $x$.  , we have \[ F_{a, \mathbb{Z}}(x)= \sqrt{\frac{a}{\pi}}\theta(\frac{\pi}{a};x).\]

We'll also use 
\[
\tilde{\theta}(c;t):=\theta\left (c,\frac{\arccos t}{2\pi}\right), \quad t\in [-1,1].
\]
It follows from the symmetries of $\theta(c,x)$ that for all $x\in \R$,
\[
\tilde{\theta}(c;\cos 2\pi x)=\theta(c;x),
\]
and moreover, as shown below, $\tth$ is absolutely monotone  on $[-1,1]$. First, we recall the Jacobi triple product formula.
\begin{theorem}[Jacobi Triple Product Formula]
Let $z,q \in \mathbb{C}$ with $\lvert q \rvert<1$ and $z\neq 0$. Then
\[
\prod^{\infty}_{r=1}(1-q^{2r})(1+q^{2r-1}z^2)(1+q^{2r-1}z^{-2})=\sum^{\infty}_{k=-\infty} q^{k^2}z^{2k}.
\]
\end{theorem}

Applying the Jacobi triple product with $q=e^{-\pi c}$ and $z=e^{\pi i x}$, gives
\begin{equation}\label{tildethetaProd}
\tilde{\theta}(c;t)= \prod^{\infty}_{r=1}(1-e^{-2\pi rc}) (1+2e^{-2\pi rc}t+ e^{-2(2r-1)\pi c}).
\end{equation}
It's elementary to verify that  $\tilde{\theta}(c;\cdot)$ is entire, and that we may compute derivatives by applications of the product rule to \eqref{tildethetaProd}. Hence, we arrive at the following proposition:
%strictly absolute monotone  on $[-1,1]$.   
%Furthermore, we may differentiate $\log \tilde{\theta}(c;\cdot)$ usin 

\begin{proposition}
\label{absmonotone}
For any $c>0$, the function $\tilde{\theta}=\tilde{\theta}(c; \cdot):[-1,1]\rightarrow(0,\infty)$ is strictly absolutely monotone on $[-1,1]$ and its logarithmic derivative $\tilde{\theta}'/\tilde{\theta}$ is strictly completely monotone on $[-1,1]$.
\end{proposition}
%\begin{proof}
%Our proposition is equivalent to showing that $h'$ is strictly completely monotone, where $h:=\log \tth$. By Equation \ref{tildethetaProd},
%Recall that $\tth$ can be expressed via the Jacobi triple product as
%\[
%\tth(c;t)= \prod^{\infty}_{r=1}(1-e^{-2\pi rc}) (1+2e^{-2\pi rc}t+ e^{-2(2r-1)\pi c}),
%\]
%so 
\[
%h=\sum_{r=1}^{\infty} \log\left[ (1-e^{-2\pi rc}) (1+2e^{-2\pi rc}t+ e^{-2(2r-1)\pi c})\right].
\]
%Let $h_r$ be the $r$th term in this sum.
%It suffices to show that each $h_{r}'$ is strictly completely monotone.
%Indeed, we have
%\[
%h_{r}'=\frac{2e^{-2\pi rc}}{1+2e^{-2\pi rc}t+ e^{-2(2r-1)\pi c}}
%\]
%so 
%\[
%[h_{r}']^{(n)}=\frac{(-1)^n[2e^{-2 \pi rc}]^{n+1}}{1+2e^{-2\pi rc}t+ e^{-2(2r-1)\pi c}}
%\]
%from which the claim follows since $1+2e^{-2\pi rc}t+ e^{-2(2r-1)\pi c}\geq 0$ for all $r,c>0$ and $t\geq-1$. 
%\end{proof}

If $\Lambda_R=(a_1\Z)\times \cdots \times(a_d\Z)$ is a rectangular lattice, then $F_{a,\Lambda_R}(x)$ is a tensor product of such functions: 
\begin{equation}
\begin{split}
  \label{recF}
  F_{a,\Phi}(x)&=\sum_{v\in\Phi}e^{-a|x+v|^2}
  =\sum_{k\in\Z^d}\prod_{i=1}^d e^{-aa_i^2(x_i/a_i +k_i)^2}\\&=\prod_{i=1}^d\left(\sum_{k_i\in\Z} e^{-aa_i^2(\frac{x_i}{a_i} +k_i)^2}\right)
  =\prod_{i=1}^dF_{aa_i^2,\Z}(x_i/a_i).
  \end{split}
\end{equation}
If $\Lambda$ contains a rectangular sublattice $\Lambda_R$, then we may write $F_\Lambda$ as a sum of such tensor products.

\begin{proposition}
\label{thetafnrectlattice}
Suppose $\Lambda$ is a lattice in $\R^d$ that contains a 
rectangular sublattice $\Lambda_R=(a_1\Z)\times \cdots\times (a_d\Z)$ and let $\omega(\Phi,\Lambda)=\Lambda\cap \Omega_{\Lambda_R}$.   Then
\begin{equation}\label{Frec}
    F_{a,\Lambda}(x)
    =\sum_{y\in\omega_{\Lambda_R}} F_{a,\Lambda_R}(x+y).
\end{equation}
\end{proposition}
\begin{proof}
The formula follows immediately from  $\Lambda=\omega(\Phi,\Lambda)+\Phi$.
\end{proof}
\subsection{Polynomial interpolation and linear programming bounds for lattice configurations}
Combining the previous results in this section, we obtain the following general polynomial interpolation framework for linear programming bounds.  For  convenience, we shall write 
\begin{equation}
 W_\Phi := \Phi^*/G_\Phi \text{ and }  \Delta_\Phi:= \R^d/S_\Phi,
\end{equation}
to denote some choice of the respective fundamental domains for a lattice $\Phi\subset \R^d$.
\begin{theorem}
\label{genlinpoly}
Let $\Phi \subset \R^d$ be such that $H\subseteq 
G_\Phi$ where $H$ is the 
coordinate symmetry group (see Sec. 
\ref{latticesymmetry}) and suppose $F:\R^d\to (-\infty,\infty]$ is $S_\Phi$ invariant. By Proposition 
\ref{polystruct},  $\Phi$ contains a rectangular sublattice 
\[
a_1 \Z \times \cdots\times a_d \Z, 
\]
which induces the change of variables $T:=T_{a_1,\ldots,a_d}$ defined in \eqref{Ttransformation} and associated polynomials $P^{\Phi}_{v}$ defined in \eqref{tildeP}. 
Suppose $(c_v)_{v\in W_\Phi}$ is such that (a) 
$c_v\ge 0$, for all nonzero $  v\in W_\Phi$ , (b) $\sum_{v\in W_\Phi}c_v<\infty$, and 
(c)  the continuous function
\[
\tg:= c_0+\sum_{v\in W_\Phi} c_v P^{G_\Phi}_{v}. 
\]
satisfies $\tg\leq \tf$ on $\tilde{\Delta}_\Phi$.
 
Then for any $n$-point configuration $\omega_n=\{x_1,\ldots,x_n\}\subset \R^d$, we have \begin{equation}
    \label{TildeLB}
E_{F}(\omega_n) \geq 
E_{g}(\omega_n)=n^2 c_0- n \tg(1,\ldots, 1),
\end{equation}
where   equality holds   if and only if 
\begin{enumerate}
    \item $\tg(t)=\tf(t)$ for all $t\in T(\{x_i-x_j\mid i\neq j \in\{1,\ldots,n\}\})$ and 
    \item $c_v M_{\sigma v}(\omega_n)=0$ for all $v\in W_\Phi$ and $\sigma \in G_{\Phi}$.
\end{enumerate}  
\end{theorem}

We now consider sufficient conditions   for the $F$-optimality  of configurations of the 
form 
$\omega(\Phi, \frac{1}{m}\Lambda)$,
where $[\Lambda:\Phi]=\kappa$.
%$\mu_{\kappa m^d}=\frac{1}{m}\Lambda \cap \Omega_\Phi$ as in \eqref{mukappa}.  
For such a configuration (when the choices of $\Lambda, \Phi,$ and $\Delta_\Phi$ are clear), we define
\begin{equation}
\label{taudefn}
    \tau_{\kappa m^d}:=(\frac{1}{m}\Lambda) \cap \Delta_\Phi.
    \end{equation}
%which equals  $\mu_{\kappa m^d}\cap\Delta_\Phi$ if we choose $\Delta_\Phi\subset \Omega_\Phi.$

\begin{corollary}
\label{mukd}
Suppose $\Phi$,  $T:=T_{a_1,\ldots,a_d}$, $\tg$, and $F$ are as in Theorem \ref{genlinpoly} and   that $\Phi\subseteq \Lambda$, and $G_\Phi \subseteq G_\Lambda$ for some lattice $\Lambda\subset \R^d$. 
Then the configuration 
$\omega(\Phi,\frac{1}{m}\Lambda)$   defined in Prop. \ref{generalmoments} is $F$-optimal if
\begin{enumerate}
    \item $c_v=0$ for all   $v  \in \left(m \Lambda^*\right)\cap W_\Phi$, and 
    \item  $\tg(t) = \tf(t)$ for all $t  \in \tilde {\tau}_{\kappa m^d}\setminus \{\mathbf{1}\}$ where 
     $\mathbf{1}=(1,1,\ldots,1)\in\R^d$.
\end{enumerate}
\end{corollary}
If such a $\tg$ exists, we refer to it as a `magic' interpolant. 

\subsection{Example: Universal optimality of $\Z$}
\label{Zunivopt}
In this section we review an alternate proof of the universal optimality of $\Z$ that proceeds through the periodic approach.  The main tool used here is the observation in  Proposition~\ref{absmonotone} that the functions  $\tilde{\theta}(c;t)$ are absolutely monotone.   This proof is essentially from \cite{Su_2015}, but H. Cohn and A. Kumar were also aware of this approach \cite{CKprivate}. The proof we give that equally spaced points are universally optimal on the unit interval is equivalent to that of \cite{Cohn_Kumar_2007} showing that the roots of unity are universally optimal on the unit circle. For simplicity, we consider the case of an even number of roots; the odd case follows in a similar manner.  

Let  $\omega^\Z_{2m}=\{j/(2m)\mid j=0,1,\ldots,2m-1\}=\omega(\Z,\Z/(2m))$ 
and $t=\cos(2\pi x)$.  Then 
\begin{itemize}
    \item 
$P_k^{\Z}(t)=T_k(t)$    for $k=0,1,2,\ldots$  
\item $\mathcal{I}(\omega^\Z_{2m})=\Z\setminus (2m\Z)$
\item We have $\tilde {\tau}_{m}=\{\cos(2\pi j/(2m))\mid j=0,1,\ldots, m\}=\{t_{j,m}\mid j=0,1,\ldots, m\}$ where $t_{j,m}:=\cos(\pi (p-j)/m)$. 
\end{itemize}

Recall that the Chebyshev polynomials of the second kind are defined by the relation $$U_k(\cos \theta)\sin \theta=\sin((k+1)\theta), \qquad k=0,1,2,\ldots$$ and form the family of monic orthogonal polynomials with respect to the measure $(1-t^2)dt$ on $[-1,1]$.
These polynomials can be related to  Chebyshev polynomials of the first kind through the relations
$$U_k(t)=\begin{cases}
    2\sum_{j=0}^{\ell}T_{2j+1}& k=2\ell+1\\
    1+2\sum_{j=1}^{\ell}T_{2j}& k=2\ell,
\end{cases}$$
showing that  $U_k$  is PSD for   $k=0, 1,2,\ldots$.   

Observe that the points  $-1<  t_{1,m}<\cdots<t_{m-1,m}<1$  are also  the roots  of $U_{m}$.  It then follows using the Christoffel-Darboux formula that the partial products  $\prod_{\ell=0}^{j}(t-t_\ell)$ have expansions in   $U_0,U_1,\ldots,U_j$ with positive coefficients for  $j=1,\ldots,m-2$ (see \cite[Prop 3.2]{Cohn_Kumar_2007} or \cite[Thm A.5.9]{MEbook}). Likewise, $\prod_{\ell=1}^{m-1}(t-t_\ell)$ is simply $U_{m-1}$ and $(t+1)=(t-t_0)$, both of which are positive definite. Hence, each such partial product is PSD as is any product of such partial products; in particular, with $T=\{t_0,t_0,t_1,t_1,\ldots,t_{m-1},t_{m-1}\}$, the partial products $p_j(T;t)$ defined in \eqref{parprod} are PSD for $j\le 2m$.  

By Proposition~\ref{absmonotone}, the function $\tf_{a,\Z}$ is absolutely monotone on $[-1,1]$ and, since the divided differences of an absolutely monotone function are non-negative, it follows that the interpolant $H_T(\tf_{a,\Z})(t)$ defined in \eqref{HTf} is PSD. Moreover, $H_T(\tf_{a,\Z})(t)$ has degree at most $2m-1$, and the error formula \eqref{ErrForm} shows that $H_T(\tf_{a,\Z})\le \tf_{a,\Z}$ on [-1,1].   
\section{The Linear Programming Framework for the families $\omega_{m^2},\omega_{2m^2},\omega_{3m^2},$ and $\omega_{6m^2}$}
\label{sec3}
We explicitly apply Corollary \ref{mukd} to the four families of periodic problems described in the introduction to obtain bivariate polynomial interpolation problems whose solutions would verify the $A_2$-universal optimality of $\omega_{m^2}$ and $\omega_{3m^2}$, and the $L$-universal optimality of $\omega_{2m^2}$ and $\omega_{6m^2}$. 
Observe that our four families of point configurations are of the form $\mu_{\kappa m^2}$ with the following choices of lattices $ \Phi \subseteq \Lambda$:
\begin{enumerate}
    \item $\omega^*_{m^2}$: $\Phi=A_2, \Lambda=A_2$
    \item $\omega^*_{2m^2}$: $\Phi=L, \Lambda=A_2$
    \item $\omega^*_{3m^2}$: $\Phi=A_2, \Lambda=\frac{1}{\sqrt{3}}A_2^{\pi/6}$
    \item $\omega^*_{6m^2}$: $\Phi=L, \Lambda=\frac{1}{\sqrt{3}}A_2^{\pi/6}$.
\end{enumerate}
It is straightforward to check that in all cases, $\Phi$ and $\Lambda$ satisfy the conditions of Theorem \ref{genlinpoly}. 
Since both choices of $\Phi$, $A_2$ and $L$, contain $L$ as a rectangular sublattice, we will work with the following change of variables to induce our polynomial structure, as described in Proposition \ref{polystruct}:  
\begin{align}
(t_1,t_2):=\left(\cos(2\pi x_1),\cos \left(\frac{2\pi x_2}{\sqrt{3}}\right)\right).
\end{align}
We will use $T$ to denote
the change of variables $T(x)=(t_1(x_1),t_2(x_2))$.
\begin{figure}[H]
    \centering
    \includegraphics[width=10cm]{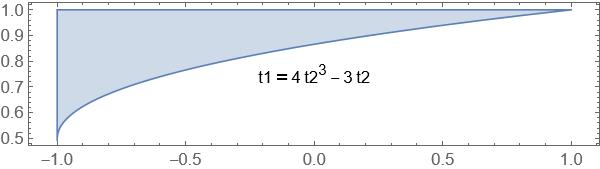}
    \caption{The region $\td$, pictured above, is our region of interpolation for the families $\omega^*_{m^2}$ and $\omega^*_{3m^2}$ (see sec. 3.4).} 
    %\label{fig:my_label}
\end{figure}

Importantly, the maps $F_{a,\Phi}$ are also well-behaved under this change of variables, as seen through decomposing $F_{a,\Phi}$ into $\theta$ functions as described in Proposition \ref{thetafnrectlattice}.

When $\Phi=L$, we obtain 
\begin{align}\label{FaL}
F_{a,L}(x)=\frac{\pi}{\sqrt{3}a}\theta(\frac{\pi}{a};x_1)\theta(\frac{\pi}{3a};\frac{x_2}{ \sqrt{3}}).
\end{align}
As a result,
\begin{align*}
\tf_{a,L}(t_1,t_2)&:=F(\frac{ \arccos(t_1)}{2\pi},\frac{ \sqrt{3}\arccos(t_2)}{2\pi})\\
&=
\frac{\pi}{\sqrt{3}a }\theta(\frac{\pi}{a };\frac{\arccos(t_1)}{2\pi})\theta(\frac{\pi}{3a };\frac{\arccos(t_2)}{2\pi})\\
&=
\frac{\pi}{\sqrt{3}a }
\tth(\frac{\pi}{a };t_1)\tth(\frac{\pi}{3a };t_2)
\end{align*}
for $t_1,t_2 \in [-1,1]$.
Thus, for fixed $t_1\in [-1,1]$, $\tf_{a,L}$ is strictly absolutely monotone as a function of $t_2$ and vice versa. We will use the absolute monotonicity in $t_1$ and $t_2$ repeatedly, and the simplicity of this formula for $\tf$ is one of the main motivations for considering the sublattice $L$.

On the other hand, when $\Phi= A_2$, we arrive at the following formula, 
%a consequence of Proposition \ref{Frec}, 
which also appears in \cite{Baernstein_1997} and \cite{Su_2015}:
\begin{equation}\label{FaA2}
F(x):=F_{a,A_2}(x)=\frac{\pi}{\sqrt{3}a}\left(
\theta(\frac{\pi}{a};x_1) 
\theta(\frac{\pi}{3 a};\frac{x_2}{\sqrt{3}})+ 
\theta(\frac{\pi}{3 a};\frac{x_2}{\sqrt{3}}+\frac{1}{2})
\theta(\frac{\pi}{a};x_1+\frac{1}{2})\right),
\end{equation}
and
\begin{equation}
\begin{split}
\tilde F(t)&:=F(\frac{\arccos t_1}{2\pi},\frac{ \sqrt{3}\arccos t_2}{2\pi})\\
&=\frac{\pi}{\sqrt{3}a}( \tilde{\theta}(\frac{\pi}{ a};t_1)  \tilde{\theta}(\frac{\pi}{ 3a};t_2 ) +\tilde{\theta}(\frac{\pi}{a};-t_1)\tilde{\theta}(\frac{\pi}{ 3a};-t_2)
 ).
\end{split} 
\end{equation}

The next corollary follows immediately from the absolute monotonicity of $\tth$ (see Proposition \ref{absmonotone}). 
\begin{corollary}
\label{evenpartials}
For any nonnegative integers $l_1$ and $l_2$ whose sum $l_1+l_2$ is even, we have on all $[-1,1]^2$ that
\[
\frac{\partial^{l_1+l_2}\tf}{\partial^{l_1} t_1 \partial^{l_2} t_2 }>0.
%\text{ on all } [-1,1]^2
%\quad \forall t_1,t_2 \in [-1,1].
\]
\end{corollary}
Finally, we'll use the following lemma from \cite{Baernstein_1997} (also see \cite{Su_2015}). 
\begin{lemma}
\label{pospartials}
On all of $[-1,1]\times[\frac12,1]$, we have the inequalities
\begin{align}
\frac{\partial \tf}{\partial t_1}
%(t_1,t_2)&
> 0, 
%\quad \forall t_1 \in [-1,1], \forall t_2 \in [\frac12,1]\\
\frac{\partial \tf}{\partial t_2}
%(t_1,t_2)
%&
\geq 0 
%\quad \forall t_1 \in [-1,1], \forall t_2 \in [\frac,1]
\end{align}
where the equality holds if and only if $t_1=-1$, $t_2=\frac12$.  In particular, these inequalities hold on all $\td$. 
\end{lemma}
\begin{proof}
Since even partial deriviatives of $\tf$ are positive and every point $(t_1,t_2)\in \td$ satisfies $t_1\geq -1$ and $t_2\geq \frac12$, it suffices to verify the inequalities 
\[
\frac{\partial \tf}{\partial t_1}(t_1,t_2)> 0 , \frac{\partial \tf}{\partial t_2}(t_1,t_2)\geq 0
\]
at $(-1,\frac1 2)$. See \cite{Baernstein_1997} or \cite{Su_2015}.  
\end{proof}
As observed in \cite{Su_2015} (also see \cite[Chapter 10]{MEbook} 
and \cite{Faulhuber_2023}), this Lemma \ref{pospartials} suffices to proves the $A_2$-universal optimality of the 2 and 3-point configurations discussed in the introduction (see section \ref{interp3msq} for more detail). 
%Indeed, the lemma implies that $\tf$ is minimized at $(-1,1/2)$, and so equivalently, $F$ is minimized at $(1/2, \sqrt{3}/6)$. But then it is elementary to check that every nonzero difference $x-y$ from the 2 and 3 point configurations equals (up to $S_{A_2}$ actions) $(1/2, \sqrt{3}/6)$, and so these configurations must minimize $F$-energy as claimed.\\

%\begin{proof}
%\begin{align}
%F(x)&=\sum_{v\in [u_0,l_0]\mathbb{Z}^2}e^{-a\lvert[ x_1,x_2]^T+v\rvert^2}\\
%&=\sum_{m,n \in \mathbb{Z}}e^{-a\lvert[ x_1,x_2]^T+mu_0+nl_0\rvert^2}\\
%&=\sum_{m,n \in \mathbb{Z}}e^{-a\left[ (x_1+  m)^2+(x_2+  \sqrt{3}n)^2\right]^2}\\
%&=\sum_{m \in \mathbb{Z}}e^{-a (x_1+  m)^2}
%\sum_{n\in \mathbb{Z}}e^{-a(x_2+  \sqrt{3}n)^2}\\
%&=\sum_{m \in \mathbb{Z}}e^{-a (x_1+  m)^2}
%\sum_{n\in \mathbb{Z}}e^{-a(x_2+  \sqrt{3}n)^2}\\
%&=\sum_{m \in \mathbb{Z}}e^{-a ^2 (x_1+ m)^2}
%\sum_{n\in \mathbb{Z}}e^{-3a(\frac{x_2}{\sqrt{3}}+n)^2}\\
%&=\sqrt{\frac{\pi}{a }}
%\theta(\frac{\pi}{a };x_1)\sqrt{\frac{\pi}{a \sqrt{3}}}\theta(\frac{\pi}{3a };\frac{x_2}{  \sqrt{3}})\\
%&=\frac{\pi}{\sqrt{3}a }\theta(\frac{\pi}{a };x_1)\theta(\frac{\pi}{3a };\frac{x_2}{ \sqrt{3}}).
%\end{align}
%\end{proof}

We will make the following choices of fundamental domains 
%$W_\Phi$ and $\Delta_\Phi$.
$\R^2/S_{\Phi}$ and $\Phi^*/ G_{\Phi}$. 
When $\Phi=L$, we take as a choice for  $\Phi^*/ G_{L}$ the set 
\[
W_{L}:=\left\{\begin{bmatrix}
    k_1\\
    k_2/\sqrt{3}
\end{bmatrix} \mid k_1,k_2\in \Z \text{ and }k_1,k_2\geq 0\right\}
\]
and $[0,1/2]\times[0,\sqrt{3}/2]$ for $\R^2/S_{L}$. Likewise for $A_2$, we take the sets 
\[
W_{A_2}:= \left\{
\begin{bmatrix}
    k_1\\
    k_2/\sqrt{3}
\end{bmatrix}\in L^*\mid 
 0\le k_2\le  k_1 \text{ and } k_1\equiv k_2  \pmod 2 \, \right\}
\]
and $\Delta_{A_2}$ (see Sec. \ref{lpboundssection}) for $\Phi^*/ G_{A_2}$ and $\R^2/S_{A_2}$, respectively. 

Finally, the following characterizations of the dual lattices $\Lambda$ will be useful for determining which degree polynomials are available to us for interpolation. First,
$L^*=\{[k_1,k_2/\sqrt{3}]^T \mid k_1,k_2\in \Z\}$. Then $A_2^*=\{ v\in L^*\mid v\cdot e_1 \equiv v\cdot \sqrt{3}e_2 \pmod 2\}$, and $\left(\frac{1}{\sqrt{3}}A_2^{\pi/6}\right)^*=\{ v\in A_2^*\mid v\cdot \sqrt{3}e_2 \equiv 0 \pmod 3\}$. 
Thus, using Theorem \ref{generalmoments}, the (non-redundant) sets of  all $v$ for which $c_v$ may be non-zero in the construction of an interpolant $\tg_a$ are expressed by the index sets
\small
\begin{align}
\mathcal{I}_{m^2}:=\mathcal{I}_{A_2}(\omega_{m^2}^*)\cap W_{A_2}& =  W_{A_2}/(mA_2^*)\\
&=\{[k_1,k_2/\sqrt{3}] \mid k_1,k_2 \geq 0, k_1 \equiv k_2 \bmod 2, [k_1,k_2]\neq m[j_1,j_2]\}\\
\mathcal{I}_{2m^2}:= \mathcal{I}_{L}(\omega_{2m^2}^*)\cap W_{L}&=W_{L}/(mA_2^*)\\
&=
\{[k_1,k_2/\sqrt{3}] \mid k_1,k_2 \geq 0, [k_1,k_2]\neq m[j_1,j_2] \text{ for some } j_1 \equiv j_2 (\bmod 2)\}
\\
\mathcal{I}_{3m^2}
:=\mathcal{I}_{A_2}(\omega_{3m^2}^*)\cap W_{A_2}&=  W_{A_2}/(mA_2^*)
\\
&=\{[k_1,k_2/\sqrt{3}] \mid k_1,k_2 \geq 0, k_1 \equiv k_2 \bmod 2, [k_1,k_2]\neq m[j_1,j_2] \text{ for some } j_2\equiv 0\bmod 3\}\\
\mathcal{I}_{6m^2}:=\mathcal{I}_{L}(\omega_{6m^2}^*)\cap W_{L}&=W_{L}/(mA_2^*)\\
&=
\{[k_1,k_2/\sqrt{3}] \mid k_1,k_2 \geq 0, [k_1,k_2]\neq m[j_1,j_2] \text{ for some } j_1 \equiv j_2 (\bmod 2), j_2 \equiv 0 \bmod 3\}.
\end{align}
\normalsize
\subsection{The Polynomials $P^{L}_{v}$ and  $P^{A_2}_{v}$}
\label{G2polynomials}
When $\Phi=L$, we have already shown that the functions are $P^{L}_v$ are tensors of Chebyshev Polynomials 
\[
P^{L}_v= T_{k_1}(t_1) T_{k_2}(t_2)
\]
where $v=[k_1,k_2/\sqrt{3}]^T$, $k_1,k_2\geq 0$ is an arbitrary element of $W_L$. 
(see Proposition \ref{rectvars}).\\
What can be said in the case when $\Phi=A_2$? These polynomials have been studied extensively (see \cite{Sun_2007}, \cite{Li_Sun_Xu_2012}, and references therein). Of particular importance to  $P^{A_2}_v$ are the polynomials $P_{v'}$ and $P_{v''}$, where $v':=[1,1/\sqrt{3}]^T$ is the shortest non-zero vector in $W_{A_2}$ and $v'':=[2,0]^T$ is the next shortest vector. 
We have
\begin{align}
P_{v'}&=\frac13 (-1+2t_2(t_1+t_2)),\\
P_{v''}&=\frac13 (-1+2 t_1 (t_1-3 t_2+4 t_2^3)).
\end{align} 
Perhaps surprisingly, every other $P_{v}$ can be expressed as a bivariate polynomial in $P_{v'}$ and $P_{v''}$, i.e. for any $v\in A_2^*$, there exist coefficients $c_{i,j}$ (with only finitely many nonzero) such that 
\[
P_{v}=\sum_{i,j\geq 0}c_{i,j} (-1+2t_2(t_1+t_2))^i (-1+2 t_1 (t_1-3 t_2+4 t_2^3))^j
\]
Note that since $P_{v'}$ and $P_{v''}$ contain only monomials of even total degree, the same is true of arbitrary $P_{v}$.
To further understand these bivariate polynomials, we set $\alpha=P_{v'} $, $\beta=P_{v''}$ and introduce a notion of degree,  first given in \cite{Patera_2010}, on polynomials of the form $\alpha^{k_0} \beta^{k_1}, i,j\geq 0$.
\begin{definition}
The $A_2$-degree of $\alpha^{k_0}\beta^{k_1}$ is $2k_0+3k_1$.
\end{definition}
If $v\in W_{A_2}$, then $v=k_0 v' + k_1 v''$ for some unique $k_0,k_1 \geq 0$, and so we can likewise introduce the notion of the $A_2$ degree of $v\in W_{A_2}$ as $2k_0+3k_1$. We will denote the degree function as $\mathcal{D}$ for both polynomials and elements of $W_{A_2}$. 
Now we can introduce an ordering on $\{\alpha^{k_0} \beta^{k_1}: k_0,k_1\geq 0\}$ by $A_2$-degree and break ties via the power of $\alpha$. Then the leading term (by $A_2$-degree) of $P_v$ is $\alpha^{k_0}\beta^{k_1}$. Certainly, this is true for our first polynomials, $P_{0}=1$, $P_{v'}=\alpha$, and $P_{v''}=\beta$, and then an examination of the recursion generating the polynomials shows that the claim holds inductively  (cf. \cite{Li_Sun_Xu_2012}).

\subsection{Interpolation Nodes}
Our final step to applying Corollary \ref{mukd} is to calculate the nodes $\tau_{\kappa m^2}$ for each family.
%at which a hypothetical interpolant $\tg$ must interpolate $\tf$ to meet the conditions of our linear programming bounds (including $0$ to simplify formulas). 
Straight from the definition,
\begin{align}
\tau_{\kappa m^2}&=\omega_{\kappa m^2}\cap \Delta_{A_2}, \hspace{1cm}\kappa=1,3\\
\tau_{\kappa m^2}&=\omega_{\kappa m^2}\cap ([0,1/2]\times[0,\sqrt{3}/2]),\hspace{1cm}\kappa=2,6
\end{align}
%Going straight from the definitions of $\omega_{\kappa m^2}$ 
and so under the $T$ change of variables, we obtain
\begin{align}
\tilde{\tau}_{m^2}&=\left\{ \left(\cos(\frac{\pi k_1}{m}),\cos(\frac{\pi k_2}{m})\right)\mid 0\leq 3k_2\leq k_1\leq m, k_1\equiv k_2 (\bmod 2) \right\}\\
\tilde{\tau}_{2m^2}&=\left\{ \left(\cos(\frac{\pi k_1}{m}),\cos(\frac{\pi k_2}{m})\right)\mid 0\leq k_1,k_2\leq m, k_1\equiv k_2 (\bmod 2) \right\}\\
\tilde{\tau}_{3m^2}&=\left\{ \left(\cos(\frac{\pi k_1}{m}),\cos(\frac{\pi k_2}{3m})\right)\mid 0\leq k_2\leq k_1\leq m, k_1\equiv k_2 (\bmod 2) \right\}\\
\tilde{\tau}_{6m^2}&=\left \{ \left(\cos(\frac{\pi k_1}{m}),\cos(\frac{\pi k_2}{3m})\right)\mid 0\leq 3k_1,k_2\leq 3m, k_1\equiv k_2 (\bmod 2) \right\}.
\end{align}

\subsection{Interpolation Problem for $\omega_{m^2}$} 
With all the machinery now set up, we address the family $\omega_{m^2}^*$ and its base case,  $\omega_{4}^*$.  Recall $v':=[1, 1/\sqrt{3}]$ is the shortest vector in $W_{A_2}$ and   
$P_{v'}=\frac13 (-1+2t_2(t_1+t_2))$. In Section~\ref{section4}, we prove the $A_2$-universal optimality of $\omega_4^*$ by constructing for each $a>0$  a polynomial of the form
$g_a(t_1,t_2):=c_0+c_1 P_{v'}(t_1,t_2)$
with $c_1\ge 0$ such that $g_a\le F_{a,A_2}$ on $\td$.

\begin{figure}[H]
    \centering
    \includegraphics[width=10cm]{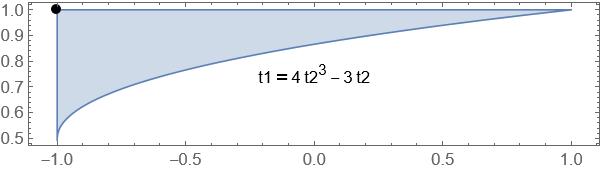}
    \caption{$\tg_a$ must stay below $\tf_a$ on $\td$ with equality at the corner point $(-1,1)$.}
    \label{4ptinterpolationpicture}
\end{figure}

For general $m$, recalling  the background on $G_2$ polynomials in Sec. \ref{G2polynomials}, we note that 
\[
\{ v\in W_{A_2}\mid  \mathcal{D}(v) < 2m\}\subset \mathcal{I}_{m^2}.
\]
The containment holds because if $v=k_0 v'+k_1v''$ and $\mathcal{D}(v) < 2m$, then $k_0< m$, and so $v\not \in mA_2^*$ (i.e. $v \in \mathcal{I}_{m^2}$).
For the $m=2$ base case already discussed, 
our interpolant $\tg_a$ satisfies 
\[
\tg_a \in 
\text{span}\{ P_{v}: v\in W_{A_2}, \mathcal{D}(v)< 2m \}.
\]
\subsection{Interpolation Problem for $\omega^*_{2m^2}$}
Now to the case of $\omega^*_{2m^2}$ with base case $\omega^*_{2}$.
\begin{figure}[H]
    \centering
    \includegraphics[height=6cm]{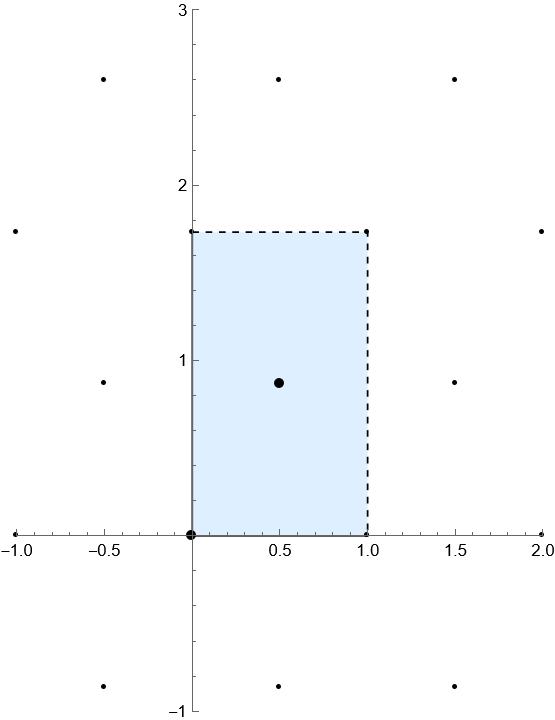}
    \caption{$\omega^*_{2}$, pictured above is $L$-universally optimal, and analogous results hold for any rectangular lattice.}
    %\label{fig:my_label}
\end{figure}

The universal optimality of $\omega^*_{2}$  follows from 
\[
F_{a,L}(x)=\frac{\pi}{\sqrt{3}a}
\theta(\frac{\pi}{a};x_1) 
\theta(\frac{\pi}{3 a};\frac{x_2}{\sqrt{3}}),
\]
which takes its minimum at $(1/2,\sqrt{3}/2)$ for all $a$, as $\theta(c,x)$ takes its minimum at $x=1/2$ for all $c>0$ (see Proposition \ref{absmonotone}). 
Since the $F_{a,L}$ energy of a two-point configuration is determined only by the difference of the two points in the configuration, the universal optimality of $\omega^*_2$ immediately follows, and the same argument can be used to show for any rectangular lattice (cf. \cite{Faulhuber_2023}) that a point at the origin and a point at the centroid of a rectangular fundamental domain yield a 2-point universally optimal configuration. 

For the general case, 
\[
\{ [k_1,k_2/\sqrt{3}]^T: 0\leq k_1,k_2 \text{ and } k_1+k_2< 2m\} \subseteq\mathcal{I}_{2m^2},
\]
and then
\[
\text{span}\{P_{v} \mid v=
[k_1,k_2/\sqrt{3}]^T: 0\leq k_1,k_2 \text{ and } 
k_1+k_2< 2m\}= \mathcal{P}_{2m-1}(t_1,t_2),
\]
where
$\mathcal{P}_n(t_1,t_2)$ is the set of bivariate polynomials of total degree at most $n$. 
For the first non-trivial case, $\omega^*_{8}$, 
we have numerical evidence that for each $a>0$, an interpolant $\tg_a$ exists in $P_{3}$ and satisfies the conditions of Corollary \ref{mukd}. 

\begin{figure}[H]
    \centering
    \begin{minipage}{0.45\textwidth}
        \centering
        \includegraphics[width=0.9\textwidth]{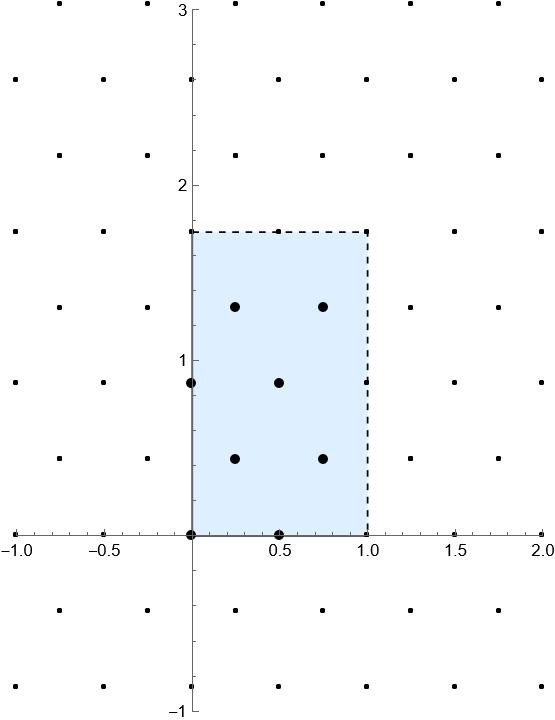} % first figure itself
        \caption{The 8 larger points comprise $\omega^*_8$}
    \end{minipage}\hfill
    \begin{minipage}{0.45\textwidth}
        \centering
        \includegraphics[width=0.9\textwidth]{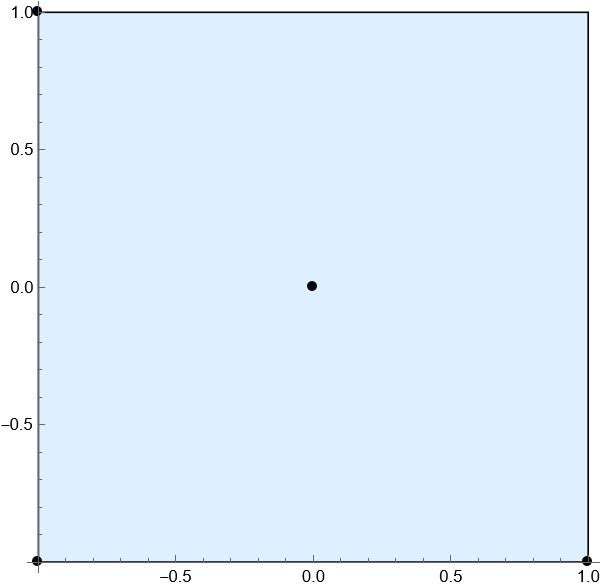} % second figure itself
        \caption{After applying the linear programming framework, it suffices to find for each $a>0$ an interpolant $\tg_a\in \mathcal{C}_{8}$ such that $\tg_a \leq \tf_a$ on $[-1,1]$ with equality at the 4 points shown.}
    \end{minipage}
\end{figure}

%%%%%%%%%%%%%%%%%%%%%%%%%%%%%%%%%%%%%%%%%%%%%%%%%%%%%%%%%%%%%%%%%%%%%%
\subsection{Interpolation Problems for $\omega^*_{3m^2}$}
\label{interp3msq}
Now to the case of $\omega^*_{3m^2}$ with base case $\omega^*_{3}$.
\begin{figure}[H]
    \centering
    \begin{minipage}{0.45\textwidth}
        \centering
        \includegraphics[width=0.9\textwidth]{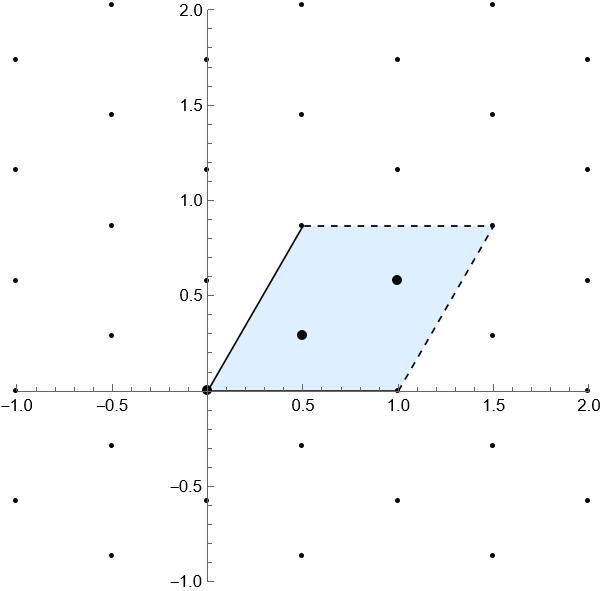} % first figure itself
        \caption{The 3 larger points comprise $\omega^*_3$}
    \end{minipage}\hfill
    \begin{minipage}{0.45\textwidth}
        \centering
        \includegraphics[width=0.9\textwidth]{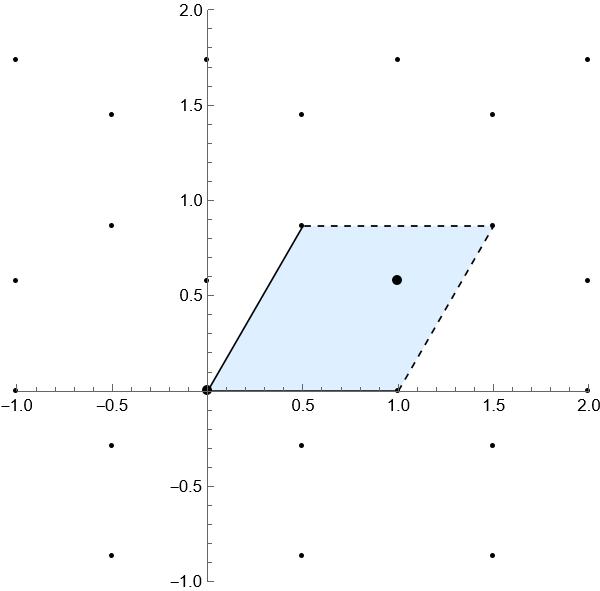}
        \caption{Almost the exact same analysis needed for the $A_2$-universal optimality of $\omega^*_3$ yields $A_2$-universal optimality of the two point honeycomb configuration} 
        %The other is obtained by removing the other non-origin point from $\omega^*_3$.}
    \end{minipage}
\end{figure}
The universal optimality of $\omega^*_{3}$ (cf. \cite{Su_2015} and \cite{Faulhuber_2023}) follows from Lemma \ref{pospartials}, which is used to show %representation
%\[
%F(x):=F_{a,A_2}(x)=\frac{\pi}{\sqrt{3}a}\left(
%\theta(\frac{\pi}{a};x_1) 
%\theta(\frac{\pi}{3 a};\frac{x_2}{\sqrt{3}})+ 
%\theta(\frac{\pi}{3 a};\frac{x_2}{\sqrt{3}}+\frac{1}{2})
%\theta(\frac{\pi}{a};x_1+\frac{1}{2})\right).
%\]
%Baernstein uses the representation to shows that 
a global minimum of $F_{a,A_2}$ occurs at $(1/2,\sqrt{3}/6)$ for all $a>0$. This point, $(1/2,\sqrt{3}/6)$, is also the only difference $x-y$ (up to $S_{A_2}$ action) for $x,y\in \omega^*_3$. Thus for an arbitrary 3-point configuration $\omega_3$, we have
\[
E_F(\omega_3)\geq 6 F(1/2,\sqrt{3}/6)=E_F(\omega^*_3)
\]
and so $\omega^*_3$ is $A_2$-universally optimal. This same line of argument is also used in \cite{Su_2015} to show that the 2-point honeycomb configuration pictured above are $A_2$-universally optimal.

More generally, we suggest invoking $A_2$-degree as in the $m^2$ case to find a nice subset of $\mathcal{I}_{3m^2}$. We have the containment 
\[
\{ v\in W_{A_2}\mid  \mathcal{D}(v) < 3m\}\subset \mathcal{I}_{3m^2}.
\]
The containment holds because if $v=k_0 v'+k_1v''$ and $\mathcal{D}(v) < 3m$, then either $v\not \in mA_2^*$ or $v=m v'\not \in (\frac{1}{\sqrt{3}}A_2^{\pi/6})^*$.

\subsection{Interpolation Problems for $\omega^*_{6m^2}$} \label{6ptInProb}
It remains to consider the family $\omega_6^*$ and its base case $\omega_{6}^*$, whose universal optimality involves our most complex application of the linear programming bounds.
In section \ref{section6}, we prove the $L$-universal optimality of $\omega^*_6$ by constructing for each $a>0$ an interpolant of the 
form 
\[
g_a(t_1,t_2)
=
b_{0,0}+b_{1,0}t_1 +b_{0,1}t_2+ b_{1,1}t_1t_2+b_{0,2}t^{2}_2
\]
where  $b_{i,j}\geq 0$ for $(i,j)\neq 0$. In that section, we will explain in greater detail why such a $\tg_a$ satisfies the conditions of Corollary \ref{mukd}.

\begin{figure}[H]
    \centering
    \includegraphics[height=6cm]{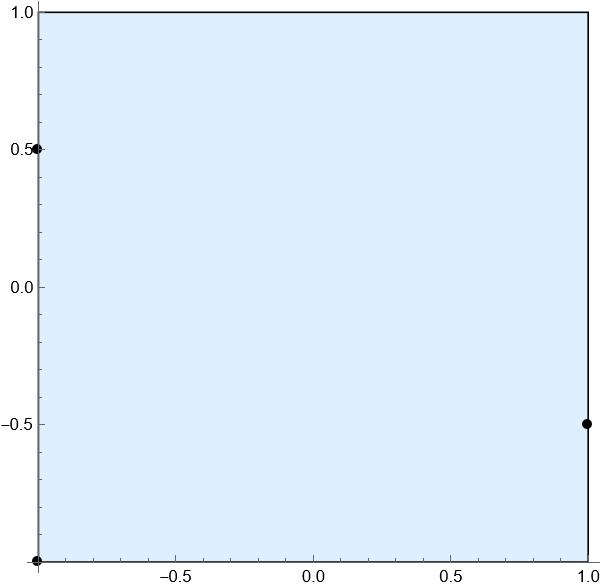}
    \caption{$\tg_a$ must stay below $\tf_a$ on $[-1,1]^2$ with equality at the three points shown}
    %\label{fig:my_label}
\end{figure}
Finally, for arbitrary $m$, we propose a few nice subsets of $\mathcal{I}_{6m^2}$. First, we have the set
\[
\{
[k_1,k_2/\sqrt{3}]^T: 0\leq k_1 <2m, 0\leq k_2 < 3m
\}
\subseteq 
\mathcal{I}_{6m^2},
\]
and 
\[
\text{span}\{P_{v} \mid v=
[k_1,k_2/\sqrt{3}]^T: 0\leq k_1< 2m, 0\leq k_2<3m \}= \mathcal{P}_{2m-1}(t_1)\times \mathcal{P}_{3m-1}(t_2).
\]
Working with such a tensor space of polynomials is natural due to the tensor product nature of 
\[
\tf_{a,L}(x)=\frac{\pi}{\sqrt{3}a}
\tth(\frac{\pi}{a};t_1) 
\tth(\frac{\pi}{3 a};t_2).
\]
Notably, our interpolant, $\tg_a$, for $\omega^*_6$ satisfies $\tg_a \in \mathcal{P}_{1}(t_1)\times \mathcal{P}_{2}(t_2)$. 
%%%%%%%%%%%%%%%%%%%%%%%%%%%%%%%%%
% 4PT SECTION   
%%%%%%%%%%%%%%%%%%%%%%%%%%%%%%%%
\section{$A_2$-universal optimality of $\omega^*_{4}$}
\label{section4}
To prove $\omega^*_{4}$ is $A_2$-universally optimal, it remains to show
for each $a>0$ that there are $c_0,c_1\in \R$ with $c_1\geq 0$ such that the resulting interpolant $\tg_a(t_1,t_2):=c_0+c_1 P_{v'}=c_0+c_1/3(-1+t_2(t_1+t_2))$ satisfies $\tg_a\leq \tf_a$ on $\td$ with equality at $(-1,1)$  or, equivalently,   finding such an interpolant of the form 
\begin{equation}\label{ga4}\
\tg_a(t_1,t_2):= \tf_a(-1,1)+b_1 t_2(t_1+t_2)
\end{equation}
for $b_1\ge 0$. 
\\
Our formulas for $\tg_a$ are defined piecewise\footnote{We suspect that $b_1$ need not be defined piecewise. In fact the choice $b_1= \frac{\partial \tf}{\partial t_1 }(-1,1)$ numerically appears to lead to $\tg_a\leq \tf_a$ for all $a>0$. But our most simple proofs come from this piecewise definition of $b_1$.} in $a$. We set 

\begin{equation}
\label{coeffchoice4pt}
b_1=
\begin{cases}
    2\frac{\partial \tf}{\partial t_1 }(-1,1/2) & \text{if } 0< a\leq 21\\
    \frac{\partial \tf}{\partial t_2 }(-1,1) & \text{if } a> 21.
\end{cases}
\end{equation}

Due to the different expansions used for $\theta$ (see \eqref{thetasmallaformula} and \eqref{thetabigaformula}), we also find it convenient to rescale $\tf$ by a factor of $\sqrt{3}\pi/a$ for small $a$ case.  Defining
\begin{equation}\label{f1f2Def}
\begin{split}
    \tfo(t_1)&:=\begin{cases}
    \tth(\frac{\pi}{a };t_1),& 0<a\le\pi^2\\
    \sqrt{\frac{\pi}{a}}\tth(\frac{\pi}{a };t_1)& a>\pi^2,
    \end{cases}\\
   \tft(t_2)&:=\begin{cases}
    \tth(\frac{\pi}{3a };t_2),& 0<a\le\pi^2\\
    \sqrt{\frac{\pi}{3a}}\tth(\frac{\pi}{3a };t_2)& a>\pi^2.
    \end{cases}     
\end{split}
    \end{equation} 
 
With this rescaling convention, it follows from   \eqref{FaA2} that 
\begin{equation}\label{Ff1f2}
\tf(t_1,t_2) =\tfo(t_1)\tft(t_2)+\tfo(-t_1)\tft(-t_2). 
\end{equation}

\subsection{Constructing magic $\tg_a$}
For all $a>0$, we will establish
\begin{lemma}\label{3rdorderpartial}
For all points $(t_1,t_2)\in \td$,
$\displaystyle\frac{\partial^3 \tf}{\partial t_1 \partial t_{2}^2 }(t_1,t_2) > 0.
$
\end{lemma}
\begin{proof} 
Since even partial derivatives of $\tf$ are positive, it suffices to check the inequality at the minimal $t_1$ and $t_2$ values, when $t_1=-1$ and $t_2=\frac12$. This check is handled in the appendix with large $a$ and small $a$ cases handled separately.
\end{proof}
Likewise, we have 
\begin{lemma}
\label{vertbdry}
Let $\tilde{h}$ be of the form 
$\tf(-1,1)+c_1 t_2(t_1+ t_2)$ such that 
$\tilde{h}(-1,1/2)< \tf(-1,1/2) $ and 
$\frac{\partial \tf-\tilde{h}}{\partial t_2}(-1,1)\leq 0$. 
Then for all $t_2\in [1/2,1]$, $\tf(-1,t_2)\geq \tilde{h}(-1,t_2)$ with equality only when $t_2=1$.
\end{lemma}
\begin{proof}
We abuse notation here and use $\tf,\tilde{h}$ to refer to the one variable functions in $t_2$ obtained by fixing $t_1=-1$.
By assumption on the form of $\tilde{h}$, Lemma \ref{pospartials}, and the two assumed inequalities, we have $\tf(1/2)\geq \tilde{h}(1/2)$, $\tf'(1/2)=\tilde{h}'(1/2)=0$, $\tf(1)= \tilde{h}(1)$, and $\tf'(1)\leq \tilde{h}'(1)$. It follows that there are exists some point in $[1/2,1]$ at which $\tf''\leq \tilde{h}''$. Let $t_2'\leq t_2''$ be such that $t_2'$ and $t_2''$ respectively are the minimal and maximal points in $[1/2,1]$ at which $\tf''\leq \tilde{h}''$. For $t_2\geq t_2''$ we have $\tf''\geq \tilde{h}''$ with equality only at $t_2''$ since $(\tf-\tilde{h}){''}$ is strictly convex (recall $\tf^{(4)}>0$). Thus, we get $\tf\geq \tilde{h}$ by bounding $\tf-\tilde{h}$ below with a tangent line of $\tf-\tilde{h}$ at $1$ and equality holds only if $t_2=1$. Similarly, for $t_2\leq t_2'$, we get the desired inequality with tangent approximation from $\frac12$. For $t_2 \in[ t_2',t_2'']$, we note that $\tf''=\tilde{h}''$ at the endpoints of the interval. Again using the strict convexity of $(\tf-\tilde{h}){''}$, we obtain $(\tf-\tilde{h})''\leq 0$ for the whole interval. Thus, $\tf(t_2)\geq \tilde{h}(t_2)$ by bounding the difference below with its secant line (since we've already established $\tf \geq \tilde{h}$ at the endpoints $t_2',t_2''$), and equality can only hold at $t_2''$ if $t_2''=1$. 
\end{proof}
%%%%%%%%%%%%%%%%%%%%%%%%%%%%%%%
%SMALL A
%%%%%%%%%%%%%%%%%%%%%%%%%%%%%%%%
\subsubsection{Small $a$}
Let $0<a\leq 21$. We will refer to $\tg_a$ as simply $\tg$.
We'll prove the following lemma\footnote{The reason we don't use this approach for all $a>9.6$ is Lemma \ref{inequalitystring} fails at roughly $a=22$. Namely, the terms $\frac{\partial \tf}{\partial t_2}(-1,1),
4(\tf(-1,1)-\tf(-1,1/2))$ both have lead exponential terms on the order of $e^{-a/4}$, while $\frac{\partial \tf}{\partial t_1 }(-1,1/2)$ is on the order of $e^{-a/3}$.} in the appendix:
\begin{lemma}
\label{inequalitystring}
For $0<a\leq 21$, we have \[
\max \left\{ \frac{\partial \tf}{\partial t_2}(-1,1),
\hspace{.25cm}
4(\tf(-1,1)-\tf(-1,1/2))\right\}
\leq 
2\frac{\partial \tf}{\partial t_1 }(-1,1/2) 
\leq \frac{\partial^2 (\tf-\tg)}{\partial t_1 \partial t_2}(-1,1/2).
\]
\end{lemma}
We handle the proof piecewise, splitting into 2 cases, $0<a\leq \pi^2$ and $\pi^2<a\leq21$ depending on which formulas we use for $\tfo$ and $\tft$. 
These 3 inequalities\footnote{Though in the small $a$ case, we have set $b_1=2\frac{\partial \tf}{\partial t_1 }(-1,1/2)$ for simplicity, in fact, we could set $b_1$ to be any element of the (non-empty) interval
$\left[\max\left\{ \frac{\partial \tf}{\partial t_2}(-1,1),
4(\tf(-1,1)-\tf(-1,1/2))\right\},2\frac{\partial \tf}{\partial t_1 }(-1,1/2)\right]$ and the exact same proof would work.}
suffice to show $\tf \leq \tg$.
We certainly have $b_1 > 0$ since $b_1=2\frac{\partial \tf}{\partial t_1 }(-1,1/2)> \frac{\partial \tf}{\partial t_2}(-1,1)>0$ where the first inequality holds by assumption and the next by Lemma \ref{pospartials}. Next, we have 
\[
(\tf-\tg)(-1,1/2)=\tf(-1,1/2)-\tf(-1,1)+\frac14 b_1> 0
\]
and likewise
\[
\frac{\partial (\tf-\tg)}{\partial t_2}(-1,1)=\frac{\partial \tf}{\partial t_2}(-1,1)-b_1 < 0. 
\]

Applying Lemma \ref{vertbdry} and the previous two inequalities to $\tg$, we obtain $\tf\geq \tg$ for all points $(-1,t_2)$ with $t_2\in [1/2,1]$ with equality only at $(-1,1)$, and since $\tf-\tg$ is convex in $t_1$, it remains to show that $\frac{\partial( \tf-\tg)}{\partial t_1}\geq 0$ for all points of the form $(-1,t_2)$, $t_2\in[1/2,1]$ (recall the picture of $\td$, Figure \ref{4ptinterpolationpicture}). 
By Lemma \ref{3rdorderpartial}, $\frac{\partial( \tf-\tg)}{\partial t_1}$ is convex in $t_2$ in $\td$, so we just need to show 
\[
\frac{\partial (\tf-\tg)}{\partial t_1 }(-1,1/2)\geq 0,
    \frac{\partial^2 (\tf-\tg)}{\partial t_1 \partial t_2}(-1,1/2)\geq 0.
\]
But these follow directly from our assumptions on $b_1$. 
Indeed,
\begin{align}
    \frac{\partial (\tf-\tg)}{\partial t_1 }(-1,1/2)&=\frac{\partial \tf}{\partial t_1 }(-1,1/2)- \frac{b_1}{2} = 0\\
    \frac{\partial^2 (\tf-\tg)}{\partial t_1 \partial t_2}(-1,1/2)&=\frac{\partial^2 \tf}{\partial t_1 \partial t_2}(-1,1/2)-b_1 > 0.
\end{align}
%is the same as for $0<a<\pi^2$, but we use estimates on $\tth$ derived from the `large $a$' formula for $\tth$. So for this range it suffices to show in the appendix that Lemma \ref{inequalitystring} holds for $9.6\leq a\leq 21$. in the appendix:

%%%%%%%%%%%%%%%%%%%%%%%%%%%%%%%%%%%%%%%%%%%%%
%intermediate A
%%%%%%%%%%%%%%%%%%%%%%%%%%%%%%%%%%%%%%%%%
%\subsubsection{Intermediate $a$}
%Throughout this intermediate region, we assume $9.6\leq a\leq 21$. Our approach
%
%\\
%%%%%%%%%%%%%%%%%%%%%%%%%%%%%%%%%%%%%%%%
\subsubsection{Large $a$}
\label{sec4ptlargea}
%%%%%%%%%%%%%%%%%%%%%%%%%%%%%%%%%%%%%%%%%%
Throughout, we assume $a>21$ and refer to $\tg_a$ as $\tg$. 
\begin{figure}
    \centering
    \includegraphics[width=10cm]{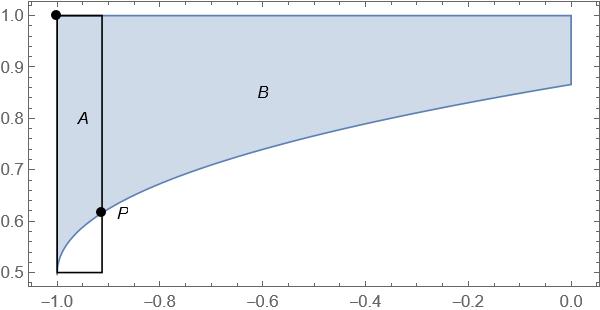}
    \label{4ptproofstrat}
    \caption{The figure depicts our strategy for showing $\tf\geq \tg$ in the large $a$ case. We show $\tf \leq \tg$ on the rectangular region A (which includes points outside of $\td$) in Lemma \ref{4ptlinearization}. The remaining points of $\td$, in region B, are handled by Lemma \ref{largea4ptt1growth}.} 
\end{figure}
We begin by showing that $\tf \ge \tg$ on two segments of the boundary or $\td$. 
\begin{lemma}
\label{4ptboundarysections}
    We have $\tf \geq \tg$ on the set $\{(-1,t_2):t_2\in [1/2,1] \} \cup \{(t_1,1):t_1\in [-1,1] \}$ with equality only at $(-1,1)$.
\end{lemma}
\begin{proof}
For the segment $\{(-1,t_2):t_2\in [1/2,1] \}$, we prove in the appendix that for $a>21$, 
\begin{equation}
\label{derivcomp}
\frac{\partial \tf}{\partial t_1}(-1,1)> \frac{\partial \tf}{\partial t_2}(-1,1).
\end{equation}
It also holds  for $0<a<21$ as an immediate consequence of Lemma~\ref{inequalitystring}. We next show for $a>21$ that $(\tf-\tg)(-1,1/2)>0$, and so using the definition 
\[
b_1=\frac{\partial \tf}{\partial t_2}(-1,1),
\]
for this range of $a$, 
we may apply Lemma \ref{vertbdry} to obtain $\tf\geq \tg$ on $\{(-1,t_2):t_2\in [1/2,1] \}$ with equality only at $(-1,1)$. Now for the other segment, we simply apply \eqref{derivcomp}, our definition of $b_1$, and the convexity of $\tf-\tg$ in $t_1$. 
\end{proof}

Next, we show that $\tf(t_1,t_2) \geq \tg(t_1,t_2)$ in $\td$ if $t_1 \leq \cos(2\pi \frac{\sqrt{3}}{4})$ with equality only at $(-1,1)$, and in fact we'll show the stronger claim that $\tf\geq \tg$ on all of $A:=[-1,\cos(2\pi \frac{\sqrt{3}}{4})]\times[1/2,1]$ (see Figure \ref{4ptproofstrat}) with equality only at $(-1,1)$.\\
Let $H_{\tf}$ denote the Hessian matrix of $\tf$.  It follows from the strict complete monotonicity of the  log derivative of $\tth$ that  $\tilde{f}_i''\tilde{f}_i < (\tilde{f}_i')^2$ for $i\in\{1,2\}$ (see Proposition \ref{absmonotone}), and hence we have
\begin{align*}
\det(H_{\tf}&(t_1,t_2)) 
=(\tfo''(t_1) \tfo(t_1))(\tft''(t_2) \tft(t_2))-\tfo'(t_1)^2\tft'(t_2)^2
+
\\
&(\tfo''(-t_1) \tfo(-t_1))(\tft''(-t_2) \tft(-t_2))-\tfo'(-t_1)^2\tft'(-t_2)^2< 0
\end{align*}
for $t_1,t_2\in[-1,1].$ 

To establish that $\tf\ge \tg$ on a rectangle $R\subset   [-1,1]^2$ with upper left corner point $(c,d)$ (we subdivide the rectangle $A$ into three such rectangles in the proof of Lemma~\ref{4ptlinearization}), we introduce the following auxiliary function 
\begin{equation}\label{linearizedg}
\tg_{c,d}(t_1,t_2):=\tg(t_1, t_2)-b_1 t_1t_2+b_1(ct_2+dt_1-cd),
\end{equation}
and observe that $t_1t_2\leq ct_2+dt_1-cd$ for $t_1\geq c$, $t_2\leq d$,  
  and so
\begin{equation}\tg(t_1,t_2) \leq \tg_{c,d}(t_1,t_2) \quad t_1\geq c,\, t_2\le d
\end{equation} with equality if and only if $t_1=c$ or $t_2=d$. We further observe that 
\[
\det(H_{\tf-\tg_{c,d}})= \det(H_{\tf})-2b_1 \tfo''(t_1)\tft(t_2)< \det(H_{\tf})<0.
\]
Hence,  to verify $\tf\geq \tg$ on the  rectangle $R$, it suffices to show $\tf\geq \tg_{c,d}$ on the boundary of the region by the second derivative test. For the two sides of the rectangle where $t_1=c$ and $t_2=d$, we will have already established $\tf\geq \tg$, and since $\tg=\tg_{c,d}$ on those sides, we immediately obtain $\tf\geq \tg_{c,d}$ there.

On the other two sides, we reduce the $a\geq21$ case to just $a=21$ in the following way. In each case, using truncated series approximations of $\theta$ and $b_1$ developed in the appendix, we find an upper bound on $\tg_{c,d}^*\ge \tg_{c,d}$ with the key feature  that $e^{a/4}\tg_{c,d}^*$ is linear in $a$. Meanwhile, as a lower bound for $\tf$, we truncate the expansions for $f_1$ and $f_2$ from \eqref{thetabigaformula} to obtain
\begin{equation}
    \label{FtLB4pt}
\tf_T(t_1,t_2):=(e^{-ax^2}+e^{-a(x-1)^2})e^{-3a u^2}+ e^{-a((\frac12-x)^2+3(\frac12-u)^2)}<\tf(t_1,t_2), \qquad  t_1,t_2\in[-1,1],
\end{equation}
where
\begin{equation}\label{tuDef}
x=\frac{\arccos(t_1)}{2\pi}, \hspace{.5 cm} u=\frac{\arccos(t_2)}{2\pi}.
\end{equation}
It is straightforward to verify that   $e^{a/4}\tf_T(t_1,t_2)$ is convex in $a$ for any fixed $(t_1,t_2)$ and so   
 the difference $e^{a/4}(\tf_T- \tg_{c,d}^*)$ is also (pointwise) convex in $a$.  Thus, to establish $\tf_T\geq \tg_{c,d}^*$ at some point $(t_1,t_2)$ for all $a\geq 21$ it suffices to show 
\begin{equation}
    \label{partialaCond21}\begin{split}
(\tf_T-\tg_{c,d}^*)(t_1,t_2)\bigg|_{a=21}&\geq 0 \\ \frac{\partial \left[e^{a/4}(\tf_T- \tg_{c,d}^*)(t_1,t_2)\right]}{\partial a}\bigg|_{a=21}&\geq 0.
 \end{split}
\end{equation}
In short, to establish
%To summarize the above procedure for establishing 
$\tf\ge \tg$ on a rectangle $R$ with upper left vertex $(c,d)$ for which we have already established this inequality on the left and upper edges, 
it suffices to establish the inequalities \eqref{partialaCond21} for $(t_1,t_2)$ on the two bottom and right line segments bounding $R$. Moreover, since the above method actually establishes $\tf\geq \tg^*_{c,d}$, we have the strict inequality $\tf>\tg$ on the whole rectangle except for possibly points where $t_1=c$ or $t_2=d$.  We summarize our discussion in the following lemma which will be helpful in the 6-point case.

\begin{lemma}
\label{linearizationlemma}
Let $R:[c,c']\times [d',d]\subseteq [-1,1]^2$ be a rectangle with upper left corner point $(c,d)$, and further suppose that there exist functions $\tg, \tg_{c,d}, \tg_{c,d}^*, \tf_T,$ and $\tf$ of the variables $(a,t_1,t_2)\in (0,\infty)\times [0,1]^2$ with continuous 2nd order partial derivatives which satisfy for all $a\geq a'$:
\begin{enumerate}
    \item $\tg \leq \tg_{c,d}\leq \tg_{c,d}^*$ on $R$ with $\tg =\tg_{c,d}$ if and only if $t_1=c$ or $t_2=d$  \label{as1}
    \item $\tf_T\leq \tf$ on $R$ \label{as2}
    \item $\det{H_{\tf-\tg_{c,d}^*}}<0$ on $R$ \label{as3}
    \item For some $m_1$, $e^{a/m_1}(\tf-\tg_{c,d})$ is pointwise convex in the parameter $a$ \label{as4}
\end{enumerate}
If there is some $a'>0$ such that the inequalities
\begin{equation}
    \label{partialaCond}\begin{split}
(\tf_T-\tg_{c,d}^*)(t_1,t_2)\bigg|_{a=a'}&\geq 0 \\ \frac{\partial \left[e^{a/m_1}(\tf_T- \tg_{c,d}^*)(t_1,t_2)\right]}{\partial a}\bigg|_{a=a'}&\geq 0
 \end{split}
\end{equation}
hold on $\partial R$, then $\tf > \tg$ on $R$ for all $a\geq a'$. 
Further, if for all $a\geq a'$, $\tf \geq \tg$ on some 
$
R'\subseteq \{ (t_1,t_2)\in \partial R: t_1=c \text{ or } t_2=d  \} 
$
and the inequalities \eqref{partialaCond} hold on $\partial R \backslash R'$, then $\tf \geq \tg$ on $R$ for all $a\geq a'$, again with equality only possible if $t_1=c$ or $t_2=d$ and $(t_1,t_2)\in R'$.   
\end{lemma}

%\begin{proof}
%If the inequalities \eqref{partialaCond} hold on all of $\partial R$, we may utilize the last assumption
%to conclude that $e^{a/m_1}(\tf_T-\tg_{c,d}^*\geq 0$ on $\partial R$ for all $a\geq a'$. Thus, by assumptions 1 and 2, we arrive at $\tf- \tg_{c,d}\geq 0$ on $\partial R$ for all $a\geq a'$. Likewise, if the inequalities \eqref{partialaCond} only hold on $R \backslash R'$, then we obtain $\tf- \tg_{c,d}\geq 0$ on $R \backslash R'$. If $\tf \geq \tg$ on $R'$, then using assumption 1 and the definition of $R'$, we obtain $\tf\geq \tg$ on $R'$ for all $a\geq a'$, and so in either case, we have 
%$\tf- \tg_{c,d}\geq 0$ on $\partial R$ for all $a\geq %a'$.

%Now applying the third assumption on the second derivative test extends the inequality $\tf\geq \tg_{c,d}$ to all of $R$, and so by the first assumption, we arrive at $\tf\geq \tg$ on $R$ as desired. 
%\end{proof}

\begin{lemma}
\label{4ptlinearization}
    The inequality  $\tf\geq \tg$ holds on $A=[-1,\cos(2\pi \frac{\sqrt{3}}{4})]\times [\frac12,1]$ with equality only at $(-1,1)$.
\end{lemma}

\begin{proof}
We partition $[-1,\cos(2\pi \frac{\sqrt{3}}{4})]\times [\frac12,1]$ into three subrectangles $R_k:=[-1,\cos(2\pi \frac{\sqrt{3}}{4})]\times [d_{k-1},d_k]$, $k=1,2,3$ where $d_0=1/2$, $d_1=3/5$, $d_2=7/10$, and $d_3=1$ and aim to verify the inequality 
$\tf \geq \tg_{-1,d_k}$ on each $R_k$  using Lemma~\ref{linearizationlemma}, with $\tg_{-1,d_k}$ as in \eqref{linearizedg}, $\tf_T$ as in \eqref{FtLB4pt}, $m_1=4$, and $a'=21$. The specific formulas for each $g_{-1,d_k}^*$ are given in the appendix section \ref{largea4Pt}.
We begin by verifying inequalities \eqref{partialaCond} for $g_{-1,-1}^*$ on the line segments of $R_1$ with $t_1=\cos(2\pi\sqrt{3}/4)$ or $t_2=7/10$, which combined with Lemma~\ref{linearizationlemma} implies $\tf \geq \tg $ on $R_1$. Now having established $\tf\geq \tg$ on the top side of $R_2$, we only need establish inequalities \eqref{partialaCond} for $g_{-1,-1}^*$ on the line segments of $R_2$ where $t_1=\cos(2\pi\sqrt{3}/4)$ or $t_2=3/5$ to get $\tf\ge \tg$ on all $R_2$. In the same fashion, showing inequalities \eqref{partialaCond} on the sides of $R_3$ where $t_1=\cos(2\pi\sqrt{3}/4)$ or $t_2=3/5$ completes the proof by yielding $\tf\geq \tg$ on $R_3$. 
The verification of these inequalities is carried out in the appendix by reducing them to inequalities of the form $h_2(t)-h_1(t)> 0$ on an  interval $(\alpha,\beta)$ where $h_1$ and $h_2$ are increasing functions and  choose $\delta:=(\beta-\alpha)/n$  with sufficiently large  that   we may rigorously verify the inequalities
\begin{equation}\label{linecheck}
   h_1(\alpha+k\delta)<h_2(\alpha+(k-1)\delta),\qquad k=1,2,\ldots, n,
   \end{equation}     thereby reducing our check to a finite number of point evaluations.  
%For the remaining region $[-1,\cos(2\pi \frac{\sqrt{3}}{4})] \times [1/2,7/10]$, we carry out the same procedure with $\tg_{-1,7/10}$ functioning as our upper bound for $\tg$ when $t_2\in [3/5,7/10]$ and  $\tg_{-1,3/5}$ functioning as our upper bound when $t_2\in [1/2,3/5]$. For the former case, when $t_2=7/10$ or $t_1=-1$, 
%$\tg_{-1,7/10}=\tg$, which we have just shown stays below $\tf$ on those segments. So it remains to show that the inequality holds on the sides where $t_1=\cos(2\pi \frac{\sqrt{3}}{4})$ and $t_2=3/5$. Likewise, we only need to show the inequality $\tg_{-1,3/5}\leq \tf_T$ when 
%$t_1= \cos(2\pi \frac{\sqrt{3}}{4})$, $t_2\in[1/2,3/5]$ and when $t_2=1/2$, $t_1\in [-1,\cos(2\pi \frac{\sqrt{3}}{4})]$.
%These computations are carried out in Section~\ref{largea4Pt} of the appendix.     
\end{proof}
Finally, we show that $\tf-\tg$ increases in $t_1$ for every point in $\td$ with $t_1\geq \cos(2\pi \frac{\sqrt{3}}{4})$, thus completing the proof for the $a>21$ case and yielding $\tf\geq \tg$ on all $\td$ with equality only at $(-1,1)$.
\begin{lemma}
\label{largea4ptt1growth}
For all $a\geq 21$ and every $p=(t_1,t_2)\in \td$ with $t_1\geq  \cos(2\pi \frac{\sqrt{3}}{4})$, $\frac{\partial( \tf-\tg)}{\partial t_1}\bigg|_p\geq 0$. 
\end{lemma}

\begin{proof}
Because of the convexity of the difference in $t_1$ and Lemma~\ref{3rdorderpartial}, it suffices to show that at $P=( \cos(2\pi \frac{\sqrt{3}}{4}), \cos(2\pi \frac{\sqrt{3}}{12}))$,
\begin{equation}
\frac{\partial (\tf-\tg)}{\partial t_1}\bigg|_{P}\geq 0, \hspace{1 cm} \frac{\partial^2 (\tf-\tg)}{\partial t_1 \partial t_2}\bigg|_{P}\geq 0
\end{equation}
which is handled in the appendix. 
\end{proof}

%%%%%%%%%%%%%%%%%%%%%%%%%%% 
%%%%%%%%%%%%%%%%%%%%%%%%%%
% 6 POINT UNIVERSAL OPTIMALITY
%%%%%%%%%%%%%%%%%%%%%%%%%%%%%%%%
%%%%%%%%%%%%%%%%%%%%%%%%%%%%%%%%
\section{$L$-universal optimality of $\omega^*_6$}
\label{section6}

We consider the $m=1$ case of the interpolation problem from Section~\ref{6ptInProb}.
In this case we have interpolation conditions at the nodes $\tilde{\tau}_{6}=\{(-1,-1),(1,-\frac12),(-1,\frac12) \}$.  Using the same rescaling convention as in the previous section, we have
$$\tf(t_1,t_2)=\tfo(t_1)\tft(t_2),$$ where $\tfo,\tft$ are as in \eqref{f1f2Def}. For $m=1$, we 
  may choose an interpolant $\tg=\tg_a\in \mathcal{P}_1(t_1)\times \mathcal{P}_2(t_2)$; i.e $\tg$ of the form
 $$
\tg(t_1,t_2)=\sum_{i=0}^1\sum_{j=0}^2b_{i,j}t_1^it_2^j.
$$  We   require that $\tg$ and $\tf$ agree at the three points in $\tilde{\tau}_{6}$ and remark that the condition $\tg \leq \tf$ further requires  
 $\partial\tg/\partial t_2=\partial\tf/\partial t_2$ at the points $(-1,1/2)$ and $(1,-1/2)$ giving a total of 5 linearly independent conditions on $\mathcal{P}_1(t_1)\times \mathcal{P}_2(t_2)$.  

\begin{figure}[H]
    \centering
    \includegraphics[height=6cm]{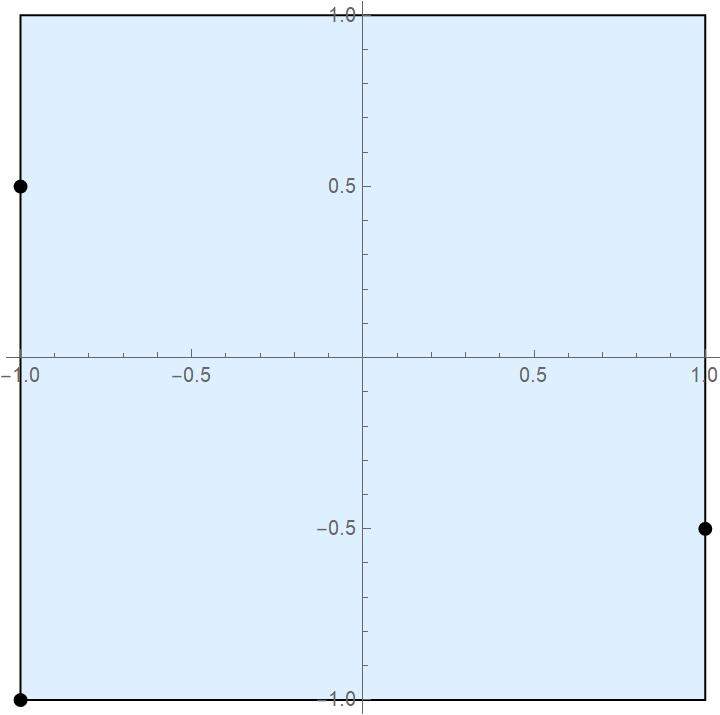}
    \caption{$\tg$ has 5 necessary equality interpolation conditions in order to provide a sharp bound.}
    \label{fig:my_label}
\end{figure}

Noting that $q(t_1,t_2)=(1+t_1)(t_2+1/2)^2$ vanishes on $\tilde{\tau}_{6}$ and that $\frac{\partial q}{\partial t_2}$ vanishes  on $\{(-1,1/2),(1,-1/2)\}$ shows that $\tg$ can be written as
\begin{equation}\label{tgH}
\begin{split}
  \tg(t_1,t_2)&=\frac{(1-t_1)}{2}\tfo(-1) H_{\{-1,\frac{1}{2},\frac{1}{2}\}}(\tft)(t_2)+ \frac{(1+t_1)}{2}\tfo(1) H_{\{-\frac{1}{2},-\frac{1}{2}\}}(\tft)(t_2) + c q(t_1,t_2),    
\end{split}
\end{equation}
where $H_T(f)$ is the Hermite interpolant to $f$ on the node set $T$ which can be expressed in terms of divided differences (see Appendix~\ref{DivDiffSec}).   In particular, 
 $$
 H_{\{-1,\frac{1}{2},\frac{1}{2}\}}(\tft)(t_2)=
 \tft(-1)+\tft[-1,\frac{1}{2}](t_2+1)+\tft[-1,\frac{1}{2},\frac{1}{2}](t_2+1)(t_2-\frac{1}{2}),
 $$
and $H_{\{-\frac{1}{2},-\frac{1}{2}\}}(\tft)(t_2)=\tft(-\frac{1}{2})+\tft'(-\frac{1}{2})(t_2+\frac{1}{2})^2$.
Since $T_1(t)=t$ and $T_2(t)=2t^2-1$, it easily follows that $\tg$ is CPSD if and only if $b_{i,j}\ge 0$ for $(i,j)\neq0$. From \eqref{tgH}, it follows that
$b_{1,2}=-\frac{1}{2}\tfo(-1)\tft[-1,\frac{1}{2},\frac{1}{2}]+c$.  Observing that $q\ge 0$ on $[-1,1]^2$, we choose $c=\frac{1}{2}\tfo(-1)\tft[-1,\frac{1}{2},\frac{1}{2}]$   as small as possible in which case $b_{1,2}=0$.

In addition,  the following derivative equality,   
\begin{equation}
\label{derivative equality}
\tfo(-1)\tft'(1/2)=\tfo(1)\tft'(-1/2).
\end{equation}
proved in \cite{Baernstein_1997} (also see  \cite{Su_2015}) implies that $b_{1,1}=b_{0,2}$. Hence we may   express $\tg$ in the form 
\begin{equation}
\tg(t_1,t_2) = 
    a_{0,0}+a_{1,0}t_1 +a_{0,1}t_2+ a_{0,2}(t_1t_2+t^{2}_2+1/4),
\end{equation}
where $a_{0,0}=b_{0,0}-b_{0,2}/4$ and $a_{i,j}=b_{i,j}$ otherwise. 

From \eqref{tgH}, we then  compute
\begin{equation}
 \label{omega6aForm}
\begin{split}
a_{0,0}&=\frac{\tfo(1)\tft(-1/2)+\tfo(-1)\tft(1/2)}{2} \\
a_{0,1}&=\tfo(-1)\tft'(-1/2)\\
a_{1,0}&=\frac{\tfo(1)\tft(-1/2)-\tfo(-1)\tft(1/2)}{2}+\frac{a_{0,1}}{2}\\
a_{0,2}&=\tfo(-1)\tft[-1,\frac{1}{2},\frac{1}{2}]=\frac49(\tfo(-1)\tft(-1)+a_{0,1}+a_{1,0}-a_{0,0}).
\end{split}   
\end{equation}

The strict absolute monotonicity and positivity of $\tft$ and   $\tfo$ show that the coefficients  $a_{0,0}, a_{0,1}$, and $a_{0,2}$ in \eqref{omega6aForm} are positive.

The next lemma which will be used to prove $a_{1,0}>0$  as well as being a first step in establishing  that $\tg\le \tf$ on $[-1,1]^2$.

 \begin{lemma}
\label{leftboundary}
$\tf(-1,t_2)\geq \tg(-1,t_2)$ for all $ t_2\in[-1, 1]$ with equality only if $t_2\in \{-1,1/2\}$.
\end{lemma}
\begin{proof}
    The result follows from the error formula \eqref{ErrForm} applied to the strictly absolute monotone function 
$\tf(-1,t_2)=\tfo(-1)\tft(t_2)$   for $t_2$ on [-1,1].
\end{proof}
It remains to show that $a_{1,0}>0$. 

\begin{proposition}
\label{nonnegcoeff}
The coefficients $a_{0,0},a_{0,1},a_{1,0}$, and $a_{0,2}$ are positive.  Hence, $\tg$ is CPSD. 
\end{proposition}
\begin{proof}
By Lemma~\ref{leftboundary}, $\tg(-1,-1/2)< \tf(-1,-1/2)$. Moreover, by definition, $\tg(1,-1/2)=\tf(1,-1/2)$. Since $(\tf-\tg)(t_1,-1/2)$ is convex in $t_1$, we must have 
\begin{equation}
a_{1,0}-\frac12 a_{0,2}=\frac{\partial \tg}{\partial t_1}(-1,-1/2)\geq \frac{\partial \tf}{\partial t_1}(-1,-1/2)>0 \label{gt1derivinequality@-1/2}.
\end{equation}
So $a_{1,0}-\frac12a_{0,2}>0$ which implies $a_{1,0}>0$ since   $a_{0,2}> 0$. 
\end{proof}

As in the proof of universal optimality of $\omega_4^*$ the most technical part of our proof is 
to verify   
    $\tg\le\tf$. In the remainder of Section~\ref{section6}, we reduce the proof of this inequality to a number of technical computations and estimates that are carried out in
the Appendices~\ref{smallaAppendix} and \ref{largeaAppendix}.

\subsection{  $\tf\geq \tg$ on $[-1,1]\times([-1,-1/2] \cup[1/2,1])$}
The following lemma, proved in the appendix, establishes  several  necessary inequality conditions  for $\tg \leq \tf$.  We next show the inequality holds on $[-1,1]\times([-1,-1/2] \cup[1/2,1])$.
 
\begin{lemma} \label{necConds}The following derivative conditions hold
    \begin{align}
    \label{necescondit1}
\frac{\partial (\tf-\tg)}{\partial t_1}(-1,-1)> 0,\\
    \label{necescondit2}
    \frac{\partial (\tf-\tg)}{\partial t_1}(-1,1/2) > 0,\\
    \label{necescondit3}
     \frac{\partial (\tf-\tg)}{\partial t_1}(1,-1/2)< 0.
\end{align}
\end{lemma}

%From the construction of $\tg$, we have that $\frac{\partial (\tf-\tg)}{\partial t_2}(p)$ vanishes at $p=(-1,1/2)$ and $p=(1,-1/2)$ and is non-negative at $p=(-1,-1)$ by Lemma~\ref{leftboundary}.  Using these observations, Lemma~\ref{necConds}, and that $(\tf-\tg)$ vanishes on $\tilde{\tau}_6^*$ shows that $\tf-\tg>0$ on deleted neighborhoods of the three points in $\tilde{\tau}_6^*$: 
%\begin{corollary}\label{deletedNbhd}    There exists some $\epsilon>0$ (depending on $a>0$) such that $\tf-\tg>0$ on $$\bigcup_{p\in\tilde{\tau}_6^*}B_\epsilon'(p),$$    where $B_\epsilon'(p)$ denotes the deleted neighborhood $\{ t\in[-1,1]^2\mid 0<\|t-p\|<\epsilon\}$.\end{corollary}

For fixed $t_2$, $\tf(t_1,t_2)-\tg(t_1,t_2)$ is strictly convex  on $[-1,1]$ as a function of  $ t_1 $ since $\tg(t_1,t_2)$ is linear  in $t_1$ and $\tfo$ is strictly absolutely monotone.  The next proposition is an immediate consequence of this observation. 
 
\begin{lemma}
\label{horizslices1/2}
Let $t_2\in[-1,1]$.  If either condition
\begin{enumerate}
    \item[(a)] $\frac{\partial (\tf-\tg)}{\partial t_{1}}(-1,t_2)\geq0$ and $\tf(-1,t_2)\geq \tg(-1,t_2)$  or
    \item[(b)] $\frac{\partial (\tf-\tg)}{\partial t_{1}}(1,t_2)\leq0$
and $\tf(1,t_2)\geq\tg(1,t_2)$
\end{enumerate} holds, then
\begin{equation}\label{horsliceIneq}
\tf(t_1,t_2)\geq \tg(t_1,t_2),\qquad t_1\in [-1,1].
\end{equation}
 
If condition (a) holds, then we have strict inequality   in \eqref{horsliceIneq} for $t_1\neq -1$. If condition (b) holds, then we have strict inequality   in \eqref{horsliceIneq} for $t_1\neq 1$.
\end{lemma}
 
We use the above lemmas to obtain:  
\begin{lemma}
\label{t2>1/2}   
We have $\tf\geq \tg$ on  $ [-1,1]\times [1/2,1]$  with equality only at $(-1,1/2)$.
\end{lemma}

\begin{proof}
We first note that  $\frac{\partial(\tf-\tg)}{\partial t_1}(t_1,t_2)=\tfo'(t_1)\tft(t_2)-a_{1,0}-a_{0,2}t_2$ is (a) strictly increasing in $t_1$ for fixed $t_2$ and (b) strictly convex in $t_2$ for fixed $t_1$. Let $h(t_2):=\frac{\partial(\tf-\tg)}{\partial t_1}(-1,t_2)$. The inequality \eqref{necescondit3} together with (a) implies  
$h(-1/2)< 0$.  Hence, the strict convexity of $h$ together with $h(1/2)>0$ (from \eqref{necescondit2}) implies $h(t_2)=\frac{\partial(\tf-\tg)}{\partial t_1}(-1,t_2) >0$ for $t_2\in[1/2,1].$ Combining this fact with  Lemma~\ref{leftboundary}, we may invoke Lemma~\ref{horizslices1/2} part (a) to complete the proof. 
\end{proof}

Next, we establish that $\tf\geq \tg$ on the right-hand boundary $t_1=1$.  
\begin{lemma}
\label{rightboundary}
We have $\tf(1,t_2)\geq \tg(1,t_2)$ for all $t_2 \in [-1,1]$ with equality only at $t_2=-1/2$.
\end{lemma}
\begin{proof}
Suppose by way of contradiction that there exists $t_2'\in [-1,1]$ such that $t_2'\neq -\frac12$ and $\tf(1,t_2') \leq \tg(1,t_2')$. Then there must be some point $-1/2\neq p\in[-1,1]$ such that $\tfo(1)\tft(p) = \tg(1,p)$. Indeed, either $t_2'$ is such a point, or $\tf(1,t_2') < \tg(1,t_2')$. We have from Lemmas \ref{necConds}, \ref{horizslices1/2}, and \ref{t2>1/2} that $\tf(1,\pm 1)> \tg(1,\pm 1)$, which yield two cases for $t_2'$. If $t_2'<-1/2$, then 
%by the intermediate value theorem applied to $t_2=t_2',-1$, 
there exists $p \in (-1,t_2')$ such that $\tf(1,p)= \tg(1,p)$ by the intermediate value theorem. If $t_2'>-1/2$, instead apply the intermediate value theorem on the interval $[t_2',1]$ to see that $p\in[1/2,1]$.

Then $\tg(1,t_2)$ is the unique quadratic polynomial that interpolates  the  function $\tf(1,t_2)$ at $T=\{p,-1/2,-1/2\}$. Then the error formula   \eqref{ErrForm} gives 
\[
\tf(1,t_2)-\tg(1,t_2)=\tfo(1)\tft^{(3)}(\xi) (t_2-p)(t_2+1/2)^2
\]
for some $\xi\in[-1,1]$.
The positivity of $\tft^{(3)}$ then implies the contradiction
$\tf(1,-1)=\tfo(1)\tft(-1)<\tg(1,-1)$ completing the proof.   
\end{proof}

\begin{lemma}\label{t2<-1/2}
We have $\tf \geq \tg$ on $[-1,1]\times [-1,-1/2]$ with equality only at  $(-1,-1)$ and $(1,-1/2)$. 
\end{lemma}

\begin{proof}
From Lemma~\ref{leftboundary}, we have $\tf\geq \tg$ for $t_1=-1$. By Lemmas \ref{necConds} and \ref{horizslices1/2}, we have the same inequality when $t_2=-1/2$ or $-1$. Finally, by Lemma~\ref{rightboundary}, we have the inequality for $t_1=1$. All of these inequalities are strict %on the boundary of  $[-1,1]\times [-1,1/2]$ 
except for at $(-1,-1)$ and $(1,-1/2)$.

Let     $p=(p_1,p_2)$ be an arbitrary point on the boundary of $[-1,1]\times [-1,-1/2]$ such that $p_1<1$ and $p_2<-1/2$, let $q=(1,-1/2)$,  and 
let $l(s):=p+s(q-p)$, $0\le s\le 1$, parametrize the line segment from $p$  to $q$. Since $u_l:=q-p$ has positive components, it follows that $\tf^l:=\tf\circ l$ is strictly absolutely monotone on $[0,1]$. Also, let $\tg^l:=\tg\circ l$ and note that $\tg^l$ is a polynomial of degree at most 2.   
%We may assume that $p$ is not on the line from $(-1,1)$ to $(1,-1/2)$ and then extend the result to these points by continuity once we have proved $\tf\geq \tg$ elsewhere. 

We claim that   $(\tf^l-\tg^l)(\epsilon)>0$ for all  sufficiently small $\epsilon>0$. Indeed, if $p\neq (-1,-1)$, then 
$(\tf^l-\tg^l)(0)>0$ and the result follows by continuity. If $p=(-1,-1)$, then $\nabla (\tf-\tg)(-1,-1) \cdot u_l> 0$ at   (-1,-1) by Lemmas~\ref{necConds} and \ref{leftboundary} which shows the result in this case.
Similarly, the necessary derivative inequality and equality conditions at $(1,-1/2)$ imply
$\nabla (\tf-\tg)(-1,1/2) \cdot u_l< 0$. Together with the fact that $(\tf^l-\tg^l)(1)=0$, we get
$(\tf^l-\tg^l)(1-\epsilon)>0$ for all $\epsilon$ sufficiently small.

Now supposing for a contradiction that $(\tf^l-\tg^l)(r')<0$ for some $r'\in (0,1)$. By the intermediate value theorem there are points $0<r_1<r'<r_2<1$ such that 
$(\tf^l-\tg^l)(r_1)=(\tf^l-\tg^l)(r_2)=0$. Then  $\tg^l$ is a polynomial of degree at most 2 which interpolates $\tf^l$ for $T=\{r_1,r_2,1\}$  and leads to a contradiction using the error formula \eqref{ErrForm}. Since any point in $(-1,1)\times(-1,-1/2)$  must lie on such a line segment, we conclude that $\tf\geq \tg$ on $(-1,1)\times(-1,-1/2)$. Now to see the inequality must be strict, if  $(\tf^l-\tg^l)(r')=0$ for some $r'\in (0,1)$ $\tg^l$ is a polynomial of degree at most 2 which interpolates $\tf^l$ for $T=\{r',r',1\}$, and again we obtain a contradiction with the error formula.  
\end{proof}
 
Thus, we have proved $\tf \geq \tg$ on $[-1,1]^2$ whenever $t_2\geq 1/2$ or $t_2 \leq -1/2$. Our proof of the inequality for the critical region $-1/2 \leq t_2\leq 1/2$  is more delicate and requires different approaches for $a$ small and $a$ large. 

\subsection{The critical region for small $a$ ($a<\pi^2$)}
For $a<\pi^2$, we take a linear approximation approach.
Let $$L_{\pm 1}(t_1,t_2):=(\tf-\tg)(\pm 1,t_2)+ (t_1\mp 1)\frac{\partial(\tf -\tg)}{\partial t_1}(\pm 1,t_2)$$ denote the tangent approximation of $\tf-\tg$ for fixed $t_2$ about $t_1=\pm 1$.    
Since  $(\tf-\tg)(t_1,t_2)$ is strictly convex in $t_1$ for fixed $t_2$,
we have 
\small
\begin{equation}
    \label{linearAppLB}
    (\tf-\tg)(t_1,t_2)\ge \max \{L_{-1}(t_1,t_2),L_{-1}(t_1,t_2)\}\ge
    \min \{L_{-1}(-1,t_2),L_{-1}(0,t_2),L_{1}(0,t_2),L_{1}(1,t_2)\},
\end{equation}
\normalsize
where the second inequality uses that $L_{\pm 1}(t_1,t_2)$ is a linear polynomial in $t_1$ for fixed $t_2$.
Note the first inequality in \eqref{linearAppLB} is strict if $-1<t_1<1$.

Now $L_{-1}(-1,t_2)=(\tf-\tg)(-1,t_2)\ge 0$ by Lemma~\ref{leftboundary} and 
$L_{1}(1,t_2)=(\tf-\tg)(1,t_2)\ge 0$ by Lemma~\ref{rightboundary}. 
In fact, we shall next prove that $L_{-1}(0,t_2)\ge L_{1}(0,t_2)$ so that the minimum on the right-hand side of \eqref{linearAppLB} is non-negative if $L_{1}(0,t_2)$ is non-negative.

\begin{lemma}\label{philemma}
   If  $t_2\in(-1,1)$, then $L_{-1}(0,t_2)> L_{1}(0,t_2)$.  
\end{lemma}
\begin{proof}
Since $\tg$ is affine in $t_1$, we have
\begin{equation}\label{Lpm2}
    L_{\pm 1}(0,t_2):= \tf(\pm 1,t_2)  \mp \frac{\partial \tf }{\partial t_1}(\pm 1,t_2)-\tg(0,t_2).
    \end{equation}
Then, the error formula \eqref{ErrForm} applied to $\tf(\cdot,t_2)$ (or the Lagrange remainder formula) gives
$$\tf(0,t_2)=\tf(\pm 1,t_2)  \mp \frac{\partial \tf }{\partial t_1}(\pm 1,t_2)+\frac{1}{2}\tfo''(\chi_\pm)\tft(t_2),$$
where $-1<\chi_{-}<0<\chi_{+}<1$ which with \eqref{Lpm2} and the absolute monotonicity of $\tfo$ implies
$$L_{-1}(0,t_2)=\tf(0,t_2)-\frac{1}{2}\tfo''(\chi_{-})\tft(t_2)> \tf(0,t_2)-\frac{1}{2}\tfo''(\chi_{+})\tft(t_2)=L_{1}(0,t_2).$$
\end{proof}

Hence, if $$\phi(t_2):=L_{1}(0,t_2)\ge 0,$$ then \eqref{Lpm2} and Lemma~\ref{linearAppLB} show $(\tf-\tg)(t_1,t_2)>0$ for all $-1<t_1<1.$
So it suffices to show 
$\phi\ge 0$ on $[-1/2,1/2]$ to prove that $\tg\le \tf$ on the critical region. We can express $\phi(t_2)$ as 
\[
\phi(t_2)=(\tfo(1)-\tfo'(1))\tft(t_2)-a_{0,0}-a_{0,1}t_2-a_{0,2}(t_{2}^2+1/4).
\]
Using our technical bounds on $\tth$, we show the following lemma in the appendix:
\begin{lemma}
\label{phithirdderiv}
For $a<\pi^2$, 
$\tfo(1)-\tfo'(1)\geq 0$.
\end{lemma}
Thus, $\phi^{(3)}(t_2) \geq 0$, and so its 2nd degree Taylor polynomial at $t_2=-1/2$ yields the following lower bound for $t_2\geq -1/2$:
\[
\phi(t_2)\geq A+B(t_2+1/2)+\frac{C}{2}(t_2+1/2)^2
\]
where $A=\phi(-1/2)$, $B=\phi'(-1/2)$, and $C=\phi''(-1/2)$.
In the appendix (see Section \ref{proofoflinapproxbound}), we prove
\begin{lemma}\label{ABCLemma}
For $a<\pi^2$, $A,C>0$. If $B<0$, then $B^2-2AC<0$. 
\end{lemma}

It follows from Lemma~\ref{ABCLemma} that
$A+B(t_2+1/2)+\frac{C}{2}(t_2+1/2)^2>0$ for $t_2\ge -1/2$ completing the proof that $\tg\le\tf$ in the case $a<\pi^2$, and moreover, showing that $\tf=\tg$ only at our interpolation points $(-1,1),(-1,1/2),(1,-1/2)$.

%%%%%%%%%%%%%%%%%%%%%%%%%%%%%%%%%%%%%%%%%%%%%%%%%%%%%
\subsection{The critical region for $a\geq 9.6$}
%%%%%%%%%%%%%%%%%%%%%%%%%%%%%%%%%%%%%%%%%%%%%%%%%%%%%
\begin{figure}[H]
    \centering
    \includegraphics[width=10cm]{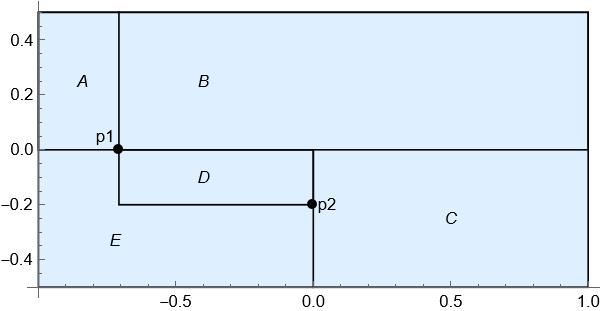}
    \caption{The figure depicts our proof strategy for showing $\tf \geq \tg$ on the critical region when $a$ is large. The points $p_1$ and $p_2$ are located at $(-\sqrt{2}/2,0),(0,-1/5)$, respectively.}
    \label{proofstrategy6}
\end{figure}
To complete the proof of universal optimality of $\omega^*_6$, it remains to show that $\tf \geq \tg$ on the critical region when $a>\pi^2$. In fact, we will show the inequality is strict on the interior of the critical region. We split the region into several subregions as in Figure \ref{proofstrategy6}.
The inequality $\tf \geq \tg$ for Subregions A,B,C,D, and E from Figure  \ref{proofstrategy6} is handled in Lemmas \ref{linearization1}, \ref{t1deriv}, \ref{gradientcomp}, \ref{linearization2}, and \ref{logderivprop}, respectively.

To prove $\tf\geq \tg$ on the regions $A$ and $D$, we apply Lemma~\ref{linearizationlemma}.   Here we use
\[
\tg_{c,d}(t_1,t_2):=\tg(t_1,t_2)+a_{0,2}(-t_1t_2 +ct_2+dt_1-cd)
\]
for $c,d\in [-1,1]$.
Approximating the coefficients $a_{i,j}$ for $a\geq a':=9.6$,
we obtain $\tg_{c,d}^*$ such that $\tg_{c,d}^*\ge\tg_{c,d}$ on the relevant subrectangle and $e^{a/3}\tg_{c,d}^*$ is linear in $a$. See Section~\ref{proofsoflinearization} for the  construction of $\tg_{c,d}^*$ in the different subrectangles. 
As a lower bound for $\tf$, we use
\begin{equation}
\label{fapprox}
\tf_T:=(e^{-ax^2}+e^{-a(x-1)^2})e^{-3a u^2}
\end{equation}
where  $x$ and $u$ are given in \eqref{tuDef}.
Analogously to the 4-point case, it is straightforward to verify that these choices of $\tf_T$, $\tg_{c,d}$, and $\tg_{c,d}^*$ satisfy   conditions 1--4 of Lemma \ref{linearizationlemma} with $m_1=3$ and $a'=9.6$. 

\begin{lemma}
\label{linearization1}
We have $\tf \geq \tg$ on $A=[-1,-\sqrt{2}/2]\times[0,1/2]$ with equality only at $(-1,1/2)$.
\end{lemma}

\begin{proof}
First, we show the inequality for $[-1,-\sqrt{2}/2]\times[1/4,1/2]$. Since we already have $\tf\ge \tg$ when  $t_2=1/2$ or $t_1=-1$, it suffices by Lemma \ref{linearizationlemma} to show inequalities \eqref{partialaCond} on the 2 segments when $t_2=1/4$ or $t_1=-\sqrt{2}/2$, which we handle in the appendix Section  \ref{proofsoflinearization}. Now having $\tf\geq \tg$ on the segment of $[-1,-\sqrt{2}/2]\times[0,1/4]$ when $t_2=1/4$, we again show inequalities  \eqref{partialaCond} with $\tg_{-1,1/4}$ on the segments when $t_2=0$ or $t_1=-\sqrt{2}/2$ to complete the proof.

%In this region, we have
% $\tg_{-1,.5}\geq \tg$ with equality when $t_2=1/2$ or $t_1=-1$. As in the 4-point case, it suffices to show that $\tg_{-1,.5} \leq \tf$ on the region, and it further suffices to show this inequality holds on the boundary.
%For the sides of the boundary where $t_1=-1$, or $t_2=1/2$, we have $\tg_{-1,1/2}=\tg$ and so $\tg_{-1,.5} \leq \tf$ by lemmas \ref{leftboundary} and \ref{horizslices1/2}. For the other two sides of the boundary, $t_1=-\sqrt{2}/2$ and $t_2=\frac14$, we proceed analogously to the 4-point case. Approximations of the $a_{i,j}$ coefficients (listed in the appendix) yield a relatively simple upper bound, call it $\tg_{-1,1/2}^*$, for $\tg_{-1,1/2}$. 

%In the appendix section \ref{proofsoflinearization}, we prove these inequalities for the boundary segments $\{(\sqrt{2}/2,t_2)\mid \frac14\leq t_2 \leq \frac12\}$ and
%$\{(t_1,\frac14)\mid -1\leq t_1 \leq -\sqrt{2}/2\}$, which completes the proof that $\tf \geq \tg$ on $[-1,-\sqrt{2}/2]\times[\frac14,\frac12]$.
%\\

%To obtain $\tf \geq \tg$ for $[-1,-\sqrt{2}/2]\times[0,\frac14]$, we conduct the same procedure with $\tg_{-1,1/4}$ instead of $\tg_{-1,1/2}$.
% Two sections of the boundary are immediate: when $t_1=-1$ or $t_2=1/4$, $\tg=\tg_{-1,1/4}$, and we have already established $\tf\geq \tg$ on those segments. The remaining two segments of the boundary are handled exactly as described above.
\end{proof}

\begin{lemma}
\label{t1deriv}
 We have $\tf > \tg$ on $B=[-\sqrt{2}/2,1]\times [0,1/2]$. 
\end{lemma}

\begin{proof}
By the convexity of $\tf-\tg$ in $t_1$, it suffices to show:
\begin{enumerate}
    \item $\tf(-\sqrt{2}/2,t_2)\geq \tg(-\sqrt{2}/2,t_2)$ for all $t_2\in [0,1/2]$
    \item $\frac{\partial (\tf-\tg)}{\partial t_1}(-\sqrt{2}/2,t_2)\geq 0$ for $t_2\in[0,1/2]$.
\end{enumerate}
The first of these follows from Lemma~\ref{linearization1}. To prove the second, it actually suffices to just show that $\frac{\partial (\tf-\tg)}{\partial t_1}(-\sqrt{2}/2,0)\geq 0$, which is handled in the appendix section \ref{proofoft1deriv}. This sufficiency follows from the same reasoning as Lemma \ref{t2>1/2} and holds because $\frac{\partial (\tf-\tg)}{\partial t_1}\geq 0$ is convex in $t_2$ and satisfies $\frac{\partial (\tf-\tg)}{\partial t_1}(-\sqrt{2}/2,-1/2)< 0$ (due to the necessary condition $\frac{\partial (\tf-\tg)}{\partial t_1}(1,-1/2)< 0$). 
\end{proof}

\begin{lemma}
\label{gradientcomp}
We have $\tf \geq \tg$ on $C=[0,1]\times [-1/2,0]$ with equality only at $(1,-1/2)$.
\end{lemma}
\begin{proof}
We claim that for this portion of the critical region, it suffices to show at each point that 
\[
L_1(t_1,t_2):=\frac{\tft'(t_2)\tfo(t_1)}{\tfo'(t_1)\tft(t_2)}-\frac{\frac{\partial \tg}{\partial t_2}(t_1,t_2)}{\frac{\partial \tg}{\partial t_1}(t_1,t_2)}> 0,
\]
%Indeed, 
since this would imply that $\tg$ increases along the level curves of $\tf$ as $t_1$ increases. 
%First, we note that if we take a point $(t_1',t_2')$ on the boundary of the region where $t_2=-\frac12$ or $t_1=1$, then by the implicit function theorem, the equation $\tf=\tfo(t_1')\tft(t_2')$ implicitly defines a function $t_2=p(t_1)$ with negative slope from $(t_1',t_2')$ to a point on the portion of the boundary with $t_1=0$ or $t_2=0$. At any point on the curve, the vector $[\frac{\tft'}{\tft},-\frac{\tfo'}{\tfo}]^T$ is parallel and in the positive $t_1$ direction of $p$, so $\tg$ increases along $p$ if 
%\[
%\frac{\tft'}{\tft}\frac{\partial \tg}{\partial t_1}-\frac{\partial \tg}{\partial t_2}\frac{\tfo'}{\tfo}> 0
%\]
%as needed. Thus, if the previous inequality holds, 

Thus, $\tf-\tg$ is minimized along the right and bottom boundaries of the region, where we have already showed $\tf-\tg\geq 0$ in the previous section with equality only at $(1,-1/2)$. The inequality $L_1(t_1,t_2)>0$ for $(t_1,t_2) \in  [0,1]\times [-1/2,0]$ is proved in the appendix section \ref{proofofgradientcomp}. 
\end{proof}

\begin{lemma}
\label{linearization2}
We have $\tf > \tg$ on $D=[-\sqrt{2}/2, 0]\times [\frac15,0]$.
\end{lemma}
\begin{proof}
We first show inequalities \eqref{partialaCond} hold for $\tg_{-\sqrt{2}/2,0}$ on each line segment on the boundary of $[-\sqrt{2}/2, 0\times[-.1 , 0]$ except the $t_2=0$ segment (where we already have $\tf >\tg$). Then we repeat the process with $\tg_{-\sqrt{2}/2,-.1}$ on each segment of $[-\sqrt{2}/2,0]\times[-.2,-.1]$ except the $t_2=-.1$ segment. 
The precise calculations are carried out in the appendix section \ref{proofsoflinearization}.

%The proof proceeds very similarly to that of Lemma~\ref{linearization1}. We use $\tf_T$ as in \eqref{fapprox}, and on $[-\sqrt{2}/2, 0\times[-.1 , 0]$, we use $\tg_{-\sqrt{2}/2,0}$ for an upper bound on $\tg$. For 
%$[-\sqrt{2}/2,0]\times[-.2,-.1]$, we similarly use $\tg_{-\sqrt{2}/2,.1}$
%The precise calculations are carried out in the appendix section \ref{proofsoflinearization}. 
\end{proof}

\begin{lemma}
\label{logderivprop}
We have $\tf > \tg$ on $E=[-1,0]\times [-.5,-.2] \cup [-1,-\sqrt{2}/2]\times[-.5,0]$.
\end{lemma}
\begin{proof}
%Let 
%\[
%S= [-1,0]\times [-.5,-.2] \cup [-1,-\sqrt{2}/2]\times[-.5,0].
%\]
We extend the domain of the $\log$ function so that $\log(t)=\infty$ for $t\leq0$. Note that this definition  and the fact that $\tf>0$ on all of $[-1,1]^2$ imply $\log(\tf/\tg)>0$ on $E$ is equivalent to $\tf>\tg$ on $E$. Since we have already established that this inequality holds on $\partial E$, it suffices to show $\log(\tf/\tg)$ takes no finite local minima on $S$, which we'll do by showing that 

\begin{equation}
\label{logderiveq}
\frac{\partial }{\partial t_1}\log\left(\frac{\tf}{\tg}\right)<0
%= \frac{\tfo'}{\tfo}-\frac{\frac{\partial \tg}{\partial t_1}}{\tg}< 0
\end{equation}
on all of $E$ where $\tg>0$, or equivalently, that if $\tg(t_1,t_2)>0$, then  
\[
\frac{\tfo(t_1)}{\tfo'(t_1)}-\frac{\tg(t_1,t_2)}{\frac{\partial \tg(t_1,t_2)}{\partial t_1}}=\frac{\tfo(t_1)}{\tfo'(t_1)}-t_1-\frac{\tg(0,t_2)}{a_{1,0}+a_{0,2}t_2}>0
\]
since each of $\tfo,\tfo', \frac{\partial \tg}{\partial t_1}> 0$ on $E$ (see Equation \ref{gt1derivinequality@-1/2}).   
Notably, $\frac{\tfo(t_1)}{\tfo'(t_1)}-t_1$ is a function only in $t_1$, while $\frac{\tg(0,t_2)}{a_{1,0}+a_{0,2}t_2}$ depends only on $t_2$. Let 
\[
L_2(t_1,t_2):=\frac{\tfo(t_1)}{\tfo'(t_1)}-t_1-\frac{\tg(0,t_2)}{a_{1,0}+a_{0,2}t_2}.
\]
We will next establish that on all of $E$, $L_2$ is decreasing in $t_1$ and $t_2$. Thus to show $L_2> 0$ on all of $E$, we need only check that $L_2(-\sqrt{2}/2,0),L_2(0,-.2)> 0$, which is handled in the appendix section \ref{proofoflogderivprop}.
To see that $L_2$ is decreasing in both $t_1$ and $t_2$, observe by Proposition \ref{absmonotone}
 that
 \[
\frac{\partial L_2}{\partial t_1}= \left(\frac{\tfo}{\tfo'}\right)'-1=\frac{(\tfo')^2-\tfo''\tfo}{(\tfo')^2}-1=-\frac{\tfo''\tfo}{(\tfo')^2}-1<0.
\]
Similarly,\[
\frac{\partial L_2}{\partial t_2}=\frac{b_{0,0}a_{0,2}-a_{0,1}a_{1,0}-a_{0,2}t_2(2a_{1,0}+a_{0,2}t_2)}{(a_{1,0}+a_{0,2}t_2)^2}
\]
whose sign depends only on $N_2(t_2):=b_{0,0}a_{0,2}-a_{0,1}a_{1,0}-a_{0,2}t_2(2a_{1,0}+a_{0,2}t_2)$. Now 
\[
N_2'(t_2)=-2a_{0,2}(a_{1,0}+a_{0,2}t_2)=-2a_{0,2}\left(\frac{\partial \tg}{\partial t_1}(t_1,t_2)\right)<0
\]
for $t_2\geq -1/2$ (see Equation \ref{gt1derivinequality@-1/2}). So the negativity of $N_2(t_2)$ and (thus $\frac{\partial L_2}{\partial t_2}$) follows from checking $N_2(-1/2)<0$, which we handle using our coefficient bounds. 
\end{proof}

%\begin{lemma}
%For all $(t_1,t_2)$ such that $-.7 \leq t_1 \leq -.1$, $-.2 \leq t_2\leq 0$, $\tf \geq \tg$.
%\end{lemma}
%Similar to the previous lemma, we note that in the relevant region,
%\[
%(t_1-(-.1))(t_2-(-.2)) \leq 0
%\]
%and thus
%\[
%t_1 t_2 \leq -.1t_2 -.2t_1-.02.
%\]
%Therefore, 
%\[
%\tg \leq \tg -a_{02}t_1t_2+a_{02}(-.1t_2 -.2t_1-.02).
%\]
%Let $\tg_{-.1,-.2}$ denote the right side of this inequality. As before, it suffices to show that $\tg_{-.1,-.2} \leq \tf$ on the region, and it further suffices to show the inequality holds on the boundary.
%We again bound $\tg_{-.1,-.2}$ above with $\tg^{'}_{-.1,-.2}$, where the coefficients $a_{i,j}$ are linear in $a$, up to a factor of $e^{-a/3}$. 
%Then we again have $e^{a/3}(\tf-\tg^{'}_{-.1,-.2})$ is convex in $a$. So to show $\tf-\tg_{-.1,-.2}\geq 0$ at some point for an arbitrary $a\geq 10$, it suffices to show $e^{a/3}(\tf-\tg^{'}_{-.1,-.2})\geq 0$ for $a=10$, with 
%\[
%\frac{\partial e^{a/3}(\tf-\tg^{'}_{-.1,-.2})}{\partial a}_\bigg|_{a=10} \geq 0.
%\]
%On each component of the boundary, these inequalities involve just one variable ($t_1$ or $t_2$), and so we carry these out in a computer assisted manner in the appendix. 

%\printbibliography

\appendix
\section*{Appendix}
\section{Divided differences and univariate  interpolation}\label{DivDiffSec}
We review some basic results concerning one-dimensional polynomial interpolation (e.g., see \cite[Section 5.6.2]{MEbook}).
Let $f\in C^m[a,b]$ for some $a,b$ be given along with some multiset
\[T=\{t_0,t_1,...,t_m\}\subseteq [a,b].
\]
 Then there exists a unique polynomial $H_T(f)(t)$ of degree at most $m$ (called a {\em Hermite interpolant of $f$}) such that for each $\alpha\in T$, we have $H^{(\ell)}_T(f)(\alpha)=f^{(\ell)}(\alpha)$ for $0\leq \ell < k_\alpha$ where $k_\alpha$ denotes the multiplicity of $\alpha$ in $T$.  
Let $f[t_0,...,t_m]$  denote the coefficient of $t^m$ in $H_T(f)(t)$. This coefficient is called the $\textit{m-th divided difference}$ of $f$ for $T$. 
Then $H_T(f)$, may be expressed as
\begin{equation}\label{HTf}
H_T(f)(t)=\sum_{k=0}^{m}f[t_0,t_1,...,t_k]p_k(T;t)
\end{equation} where the {\em partial products} $p_k$ are defined  by
\begin{equation}\label{parprod}
p_0(T;t):=1 \text{ and } p_j(T;t):=\prod_{i<j}(t-t_i), \hspace{.2cm} j=1,2,...,m.
\end{equation}
Then a generalization of the mean value theorem  implies that there is some $\xi \in [a,b]$ such that 
\begin{equation}\label{genMVT}
\frac{f^{(m)}(\xi)}{m!}=f[t_0,t_1,...,t_m].
\end{equation}

Putting these together, we arrive at the classical \textit{Hermite error formula}: 
\begin{equation}\label{ErrForm}
f(t)-H_T(f)(t)=f[t_0,t_1,...,t_m,t]\prod_{i=0}^{m}(t-t_i)=\frac{f^{(m+1)}(\xi)}{(m+1)!}\prod_{i=0}^{m}(t-t_i).    
\end{equation}
In the case that $f$ is absolutely monotone on $[a,b]$, such as with $\tfo$ and $\tft$, then the sign of $f(t)-H_T(f)(t)$   equals the sign of  $\prod^{m}_{i=0} (t-t_i)$.

\section{Equivalence of Different Notions of Universal Optimality}
\label{notionsofunivopt}
In this section, we prove that the definition of a lattice $\Lambda$ being universally optimal given in the introduction is equivalent to 
that given in \cite{CKMRV_2022}. We will use the language and definitions given after the statement of Theorem \ref{mainThm} in the introduction. 
We'll also need the following classical result\footnote{This result actually holds for a larger class of potentials that attain negative values.  However, here we only really need this for nonnegative potentials such as $f_a$.} from the statistical mechanics literature (cf. \cite{Fisher_1964} or \cite{Lewin_2022}).      
\begin{lemma}
\label{basicfisherlemma}
Let $f:[0,\infty)\rightarrow [0,\infty]$ be a lower semi-continuous map  of $d$-rapid decay and $\Omega \subset\R^d$ be a bounded, Jordan-measurable set.  Then for any $\rho>0$, $N_k\rightarrow \infty$ and $\ell_k\rightarrow \infty$ such that $\frac{N_k}{\ell_{k}^d Vol(\Omega)}\rightarrow \rho$, the limit
\[
\lim_{k\rightarrow \infty} \frac{\mathcal{E}_{f}(N_k, \ell_k \Omega)}{N_k}=C_{f,d,\rho}
\] 
exists and is independent of $\Omega$. 
\end{lemma}
The following proposition shows the equivalence of the different notions of universal optimality: 
\begin{proposition}
\label{defnequivalences}
Let $\Lambda\subseteq \R^d$ be a lattice of some density $\rho>0$. Fix $f:[0,\infty)\rightarrow [0,\infty]$ as a lower semi-continuous map  of $d$-rapid decay. For an arbitrary sublattice $\Phi\subseteq \Lambda$, let $F_{\Phi}:= F_{f,\Phi}$. Then the following are equivalent:
\begin{enumerate}
    \item[(1)] As an infinite configuration of density $\rho$, $\Lambda$ is $f$-optimal. 
    \item[(2)] For every sublattice $\Phi\subseteq \Lambda$, the configuration $\Lambda \cap \Omega_{\Phi}$ is $F_{\Phi}$-optimal.
    \item[(3)] There is some sublattice $\Phi\subseteq \Lambda$ such that $\Lambda \cap \Omega_{m\Phi}$ is $F_{m\Phi}$ optimal for infinitely many $m\in \mathbb{N}$. 
\end{enumerate}
\end{proposition}
\begin{proof}
First, we'll prove (1) implies (2). 
Let $\Lambda$   be a lattice of density $\rho$ satisfying condition (1). 
If $\Phi$ is of index $n$, then for an arbitrary $n$-point configuration $\omega_n$, we define $E_{\Phi}(\omega_n):= E_{F_{\Phi}}(\omega_n)$, and the $n$-point $\Phi$-periodic configuration $C_n= \omega_n + \Phi$. Note that $C_n$ has density $\rho$. By assumption $E_{f}(\Lambda)\leq E_{f}(C_n)$. Noting that $\Lambda$ is also an $n$-point $\Phi$-periodic configuration, we apply Proposition \ref{periodictoaverage} to obtain
\begin{equation*}
E_{F,\Phi}(\Lambda \cap \Omega_{\Phi})=N E_f(\Lambda)-N \sum_{0\neq v\in \Lambda}f(\lvert v \rvert^2)
\leq N E_f(C_n)-N \sum_{0\neq v\in \Lambda}f(\lvert v \rvert^2)
=E_{F,\Phi}(\omega_n).
\end{equation*}
Since $\omega_n$ was arbitrary, we conclude $\Lambda \cap \Omega_{\Phi}$ is $F_{\Phi}$-universally optimal as desired.\\

Clearly (2) implies (3) so it remains to show (3) implies (1). Assume $\Lambda$ is generated by some matrix $V_\Lambda$. Let $\Phi\subseteq \Lambda$ be  of index $\kappa$, generated by some matrix $V_\Phi$ and $\{N_k\}\rightarrow \infty$ be our increasing sequence of scalings for which $\Lambda$ yields an $N_k\Phi$-universally optimal configuration.
%We will tweak our notation so that $\Omega_\Lambda=V[0,1]^d$, ensuring compactness and thus that $\mathcal{E}_f(N,\Omega_\Lambda)$ is actually achieved by the $f$-energy of some multiset. 
By Lemma \ref{basicfisherlemma},
certainly $C_{f,d,\rho}$ is a lower bound for  $E^{l}_f(C)$ for any $C$ of density $\rho$, simply by the definition of average energy. Thus, we just have to show $E_f(\Lambda)\leq C_{f,d,\rho}$. To do so, 
we note that $f$ satisfies the so-called \textit{weakly tempered inequality} (cf. \cite{Fisher_1964}), that is, there exist some $\epsilon, R_0, c>0, $ such that for any two $N_1,N_2$ point configurations $\mu_{N_1}=\{x_1,\dots,x_{N_1}\}, \mu_{N_2}=\{ x'_1,\dots, x'_{N_2} \}$, which are separated by distance at least $R\geq R_0$, we have 
\[
2\sum_{i=1}^{N_1}\sum_{j=1}^{N_2}f(\vert x_i-x'_j\vert^2) \leq \frac{N_1N_2c}{R^{d+\epsilon}} .
\]
In other words, the interaction energy between the two sets decays like $R^{d+\epsilon}$.
Now set 
\[
\alpha_k=\frac{1}{N_k^{\epsilon/(2(d+\epsilon))}},
\]
and define for each $k\geq 1$, the configuration
$\theta_k$ as a $\kappa N_k^d$- point configuration which is $f$-optimal on the set $(1-\alpha_k)\bar{\Omega_{N_k\Phi}}$, Also define
$\omega_{\kappa N_{k}^d}:=\Lambda \cap \Omega_{N_k\Phi}$.

We claim the following inequality string holds, which would suffice to prove our desired result:

\[
\label{ineqstringfisher}
E_f(\Lambda)=\lim_{N_k \rightarrow \infty}
\frac{E_{N_K\Phi} (\omega_{\kappa N_k^d})}{ \kappa N_k^d }
\leq
\lim_{k\rightarrow \infty} \frac{E_{ N_K \Phi} (\theta_k)}{\kappa N_k^d}
\leq 
\lim_{k\rightarrow \infty} \frac{E_f(\theta_k)}{\kappa N_k^d}
=C_{f,d,1}. 
\]

To obtain the first equality, we apply Proposition \ref{periodictoaverage} to the lattices $N_k \Phi$ and the configurations
$\omega_{\kappa N_{k}^d}$, yielding 

\begin{align}
E_f(\Lambda)&=\frac{1}{ \kappa N_k^d }\left(E_{N_K\Phi} (\omega_{\kappa N_k^d})+ \kappa N_{k}^d\sum_{0\neq v\in N_k\Phi} f(\vert v \vert^2)\right)\\
&=\frac{E_{N_K\Phi} (\omega_{\kappa N_k^d})}{ \kappa N_k^d }+ \sum_{0\neq v\in N_k\Phi} f(\vert v \vert^2).
\end{align}
Since 

\[
\lim_{N_k\rightarrow \infty} \sum_{0\neq v\in N_k\Phi} f(\vert v \vert^2)=0,
\] 
we have 
\[
E_f(\Lambda)=\lim_{N_k \rightarrow \infty}
\frac{E_{f,N_K\Phi} (\omega_{\kappa N_k^d})}{ \kappa N_k^d }
\]
as needed.
Our first inequality 
\[\lim_{N_k \rightarrow \infty}
\frac{E_{N_K\Phi} (\omega_{\kappa N_k^d})}{ \kappa N_k^d }
\leq
\lim_{k\rightarrow \infty} \frac{E_{ N_K \Phi} (\theta_k)}{\kappa N_k^d}
\]
follows immediately from our assumption of condition (2). 

To obtain our next inequality, 
\[\lim_{k\rightarrow \infty} \frac{E_{ N_K \Phi} (\theta_k)}{\kappa N_k^d}
\leq 
\lim_{k\rightarrow \infty} \frac{E_f(\theta_k)}{\kappa N_k^d},
\]
we first observe 
\[
E_{N_k \Phi} (\theta_k) = E_f(\theta_k) +\sum_{x\neq y \in \theta_k} \sum_{v\neq 0 \in N_k\Phi} f(x-y+v),
\]
so it suffices to show $\displaystyle{\sum_{x\neq y \in \theta_k} \sum_{v\neq 0 \in N_k \Phi} f(x-y+v)\in o(N_k^d)}$ as $N_k\rightarrow \infty$. 
We claim that there exists some $m>0$ such that for all $v\in N_k \Phi$,
\[
d(\theta_k, \theta_k+v) \geq m\alpha_k \lvert v\rvert,
\]
which is proved analogously to \cite[Theorem 8.4.1]{MEbook}. 
%To see why, let $S=[0,1]^d$, and note that 
%if $\phi$ is the map with 
%with the matrix $V_\Phi=[v_1,\dots v_d]$, $k\in \N$
%and $v=b_1v_1+\dots b_dv_d\in \Phi$,  we have
%\begin{align}
%\theta_k&\subseteq V_\phi((1-\alpha_k)S)\\
%\theta_k+v&\subseteq V_\phi((1-\alpha_k)S+b),
%\end{align}
%where $b=[b_1,...,b_d]^T\in \Z^d$.
%Now for $x \in (1-\alpha_k)S, y\in (1-\alpha_k)S+b$,

%\begin{align}
%\lvert x-y\rvert& \geq \lVert x-y\rVert_{\infty}\geq \max_{i=1,...,d} (\vert b_i \vert -(1-\alpha_k)) \\
%&\geq \max_{i=1,...,d} \alpha_k\vert b_i \vert= \alpha_k \Vert b\Vert_\infty
%\end{align}

% Now applying the bi-lipschitz property of $V_\Phi$ and equivalence of norms on $\R^d$, we have for some $m_1,m_2,m_3>0$
%\[
%d(\theta_k,\theta_k+v) 
%\geq 
%d(V_\phi((1-\alpha_k)S), V_\phi((1-\alpha_k)S+b))
%\geq
%m_1d( (1-\alpha_k)S, (1-\alpha_k)S+b)
%\geq 
%m_1\alpha_k \Vert b \Vert_{\infty}
%\geq 
%m_1m_2 \alpha_k\Vert v \Vert_{\infty}
%\geq 
%m_1m_3m_2\alpha_k \lvert v\rvert . 
%\]

%Thus, with $m=m_1m_2m_3$,
%\[
%d(\theta_k,\theta_k+v) \geq m_1m_2m_3 \alpha_k\lvert v \rvert
%= \frac{m \lvert v \rvert }{N_k^{\epsilon/(2(d+\epsilon))}}
%\]
%as desired. 
Returning to 
$\displaystyle{\sum_{x\neq y \in \theta_{k}} \sum_{v\neq 0 \in N_k \Phi} f(x-y+v)}$
for $k$ large enough, we can use the weakly tempered definition to obtain:  
\begin{align*}
\sum_{x\neq y \in \theta_k} \sum_{v\neq 0 \in N_k \Phi} f(\vert x-y+v\vert^2) 
%=
%\sum_{v\neq 0 \in N_k \Phi}
%\sum_{x\in \theta_k}
%\sum_{y\in \theta_k+v/\{ x+v\}}  f(\vert x-y\rvert^2)
&\leq 
\sum_{v\neq 0 \in N_k \Phi}
\sum_{x\in \theta_k}
\sum_{y\in \theta_k+v}  f(\vert x-y\rvert^2) \\ 
&\leq 
%\sum_{x\neq y \in C_{N_k}}
\sum_{v\neq 0 \in N_k \Phi} \frac{N_k^{2d} c}{d(\theta_k,\theta_k+v)^{d+\epsilon}}\\
&\leq \sum_{v\neq 0 \in N_k \Phi} 
\frac{N_k^{2d} c
N_k^{\epsilon/2}}{(m\lvert v \rvert)^{d+\epsilon}}\\
%&=\frac{c N_k^{(2d+\epsilon/2)}}{m^{d+\epsilon}}\sum_{v\neq 0 \in N_k \Phi} \frac{1}{\lvert v \rvert^{d+\epsilon}}\\
%&=\frac{c N_k^{(2d+\epsilon/2)}}{m^{d+\epsilon}}\sum_{v\neq 0 \in  \Phi} \frac{1}{\lvert N_k v \rvert^{d+\epsilon}}\\
%&=\frac{c N_k^{(2d+\epsilon/2)}}{m^{d+\epsilon}N_k^{d+\epsilon}}\sum_{v\neq 0 \in  \Phi} \frac{1}{\lvert v \rvert^{d+\epsilon}}\\
&=\frac{c N_k^{(d-\epsilon/2)}}{m^{d+\epsilon}}\sum_{v\neq 0 \in  \Phi} \frac{1}{\lvert v \rvert^{d+\epsilon}}
\end{align*}
and this last quantity is of order $o(N_k^d)$ since the sum converges and no term but $N_k^{d-\epsilon/2}$ depends on $k$.  We should note we treat $y\in \theta_k+v/ \{ x+v\}$  in the multiset sense, decreasing the cardinality of $x+v$ in $\theta_k+v$ by one.

The final equality $\displaystyle{\lim_{k\rightarrow \infty} \frac{E_f(\theta_k)}{\kappa N_k^d}
=C_{f,d,1}}$ is immediate from Lemma \ref{basicfisherlemma}, which we can apply by the definition of $\theta_k$ and the fact that $\alpha_k\rightarrow 0$. 
\end{proof}

%%%%%%%%%%%%%%%%%%%%%%%%%%%%%%%%%%%%%%%%%%%%%%%%%%%%%%%%%%%%%%%%%%%%%%%%%%%%%%%%%%%%%%%%%%%%%%%%%%%%%%%%%5
%%%%%%%%%%%%%%%%%%%%%%%%%%%%%%%%%%%%%%%%%%%%%%%%%%%%%%%%%%%%%%%%%%%%%%%%%%%%%%%%%%%%%%%%%%%%%%%%%%%%%%%%%5
%%%%%%%%%%%%%%%%%%%%%%%%%%%%%%%%%%%%%%%%%%%%%%%%%%%%%%%%%%%%%%%%%%%%%%%%%%%%%%%%%%%%%%%%%%%%%%%%%%%%%%%%%%

\section{Technical Estimates and Computations for $a<\pi^2$}
\label{smallaAppendix}
%%%%%%%%%%%%%%%%%%%%%%%%%%%%%%%%%%%%%
\subsection{$\theta$ estimates}
%%%%%%%%%%%%%%%%%%%%%%%%%%%%%%%%%%%%
Recall that for $a<\pi^2$, we define 
$\tilde{f}_1(t_1)=\tth(\frac{\pi}{a };t_1)$ and
$\tilde{f}_2(t_2)=\tth(\frac{\pi}{3a };t_2)$. 
%We then aim to find an interpolant for $\tfo(t_1)\tft(t_2)$ in the $L$ case and $\tfo(t_1)\tft(t_2)+\tfo(-t_1)\tft(-t_2)$ in the $A_2$ case.

For $0<a<\pi^2$, we use truncations of the formula 
\[
\theta(\frac{\pi}{a};x):=\sum^{\infty}_{k=-\infty} e^{-d k^2 }e^{2\pi i kx}=1+\sum_{k\geq 1} 2e^{-d k^2}\cos(2\pi kx).
\]
where $d:=\frac{\pi^2}{a}>1$, to obtain bounds on $\theta$. Thus, we will use
\begin{align}
\tfo(t_1,j)=f_1(x_1,j)&:=1+\sum_{k= 1}^{j} 2e^{-d k^2}\cos(2\pi k x_1)\\
\tft(t_2,j)=f_2(\frac{x_2}{\sqrt{3}},j)&:=1+\sum_{k= 1}^{j} 2e^{\frac{-d k^2}{3}}\cos(\frac{2\pi k x_2}{\sqrt{3}})
\end{align}

We first bound the tails of these series:
\begin{equation}
    \begin{split}
    \left\lvert\sum_{k\geq 3} 2e^{-d k^2}\cos(2\pi kx) \right\rvert &
    %&\leq \sum_{k\geq 3} 2e^{-d k^2}\\
    \leq2\sum_{k\geq 0} e^{-d (k+3)^2}
    \leq 2e^{-4d}e^{-5}\sum_{k\geq 0}e^{-(k^2+6k)}\\
    &\leq 2e^{-4d}e^{-5}\sum_{k\geq 0}e^{-(7k)}
    = e^{-4d}\frac{2}{e^5(1-e^{-7})}< \frac{e^{-4d}}{50}.
\end{split}
\end{equation}
Similarly, we have 
\begin{align}
    \left\lvert\sum_{k\geq 5} 2e^{-d/3 k^2}\cos(2\pi kx) \right\rvert &\leq 
    e^{-16d/3}\frac{2}{e^3(1-e^{-11/3})}<\frac{e^{-16d/3}}{5}.
\end{align}

Hence, for $t_1,t_2\in[-1,1]$, we have
\begin{equation}
 \label{6ptsmallthetavalbounds}   
\begin{split}
 \tfo(t_1,2) -\frac{e^{-4d}}{50} < \tfo(t_1)    &< \tfo(t_1,2) +\frac{e^{-4d}}{50},\\
\tft(t_2,4)- \frac{e^{-16d/3}}{5} < \tft(t_2) &< \tft(t_2,4)+ \frac{e^{-16d/3}}{5}.
\end{split}
\end{equation}
 
We also need bounds on the derivatives of $\tfo$ and $\tft$. With $x=\frac{\arccos(t)}{2\pi}$, then $\tilde{\theta}(c,t)=\theta(c,x)$, and so by the chain rule
\[
%\frac{\partial \tth}{\partial t}(\frac{\pi}{a};t)
\tth'(t;\frac{\pi}{a})
=\left(-\sum_{k\geq 1} 2e^{-d k^2}(2\pi k)\sin (2\pi kx)\right)\frac{-1}{2\pi \sqrt{1-t^2}}=\frac{\sum_{k\geq 1} 2ke^{-d k^2}\sin (2\pi k x)}{\sin(2\pi x)}.
\]
For $t=\pm 1$ ($x=0,\frac12$), we   use L'Hopital's rule to obtain
\[
\tth'(1;\frac{\pi}{a})
%=\frac{\sum_{k\geq 1} 2ke^{-d k^2}(2\pi k)\cos (2\pi k x')}{(2\pi)\cos(2\pi x')}
=\sum_{k\geq 1} 2k^2e^{-d k^2}
\]
and  
\[
\tth'(-1;\frac{\pi}{a})=\sum_{k\geq 1} (-1)^{n+1}2k^2e^{-d k^2}.
\]
Then we again bound the tails by comparison with geometric series. For example, using that $d>1$ and for all $k\geq 0$,
$(k+3)^2 \leq 9e^{k}$, we obtain
\begin{align*}
\sum_{k\geq 3}2k^2 e^{-dk^2}
%&=\sum_{k\geq 0}2(k+3)^2 e^{-d(k+3)^2}\\
=e^{-9d}2\sum_{k\geq 0}(k+3)^2 e^{-d(k^2+6k)}
\leq e^{-4d}e^{-5} 2\sum_{k\geq 0}9e^{k}e^{-(k^2+6k)}\\
\leq e^{-4d}18e^{-5}\sum_{k\geq0}e^{-6k}
=\frac{18e^{-4d}}{e^{5}(1-e^{-6})}
< \frac{e^{-4d}}{8}. 
\end{align*}

In the $d/3$ case, since $(k+6)^2\leq 36e^{k/3}$, we analogously have
\begin{align*}
\sum_{k\geq 6}2k^2 e^{-d/3k^2}
%&
=\sum_{k\geq 0}2(k+6)^2 e^{-d/3(k+6)^2}
%&=e^{-25d/3}2e^{-11/3}\sum_{k\geq 0}(k+6)^2 e^{-d/3(k^2+12k)}\\
%&\leq e^{-25d/3}2e^{-11/3}\sum_{k\geq 0}36e^{k/3} e^{-1/3(13k)}\\
%&\leq e^{-25d/3}72e^{-11/3}\sum_{k\geq 0} e^{-4k}\\
%&=e^{-25d/3}\frac{72}{e^{11/3}(1-e^{-4})}\\
&< 2e^{-25d/3}.
\end{align*}
When $t=\pm \frac12$, we have $\sin 2\pi x=\sqrt{3}/2$ and so 
\[
\left\lvert\frac{\sum_{k>j} 2ke^{-d k^2}\sin (2\pi k x)}{\sin(2\pi x)}\right\rvert\leq 
\frac{4}{\sqrt{3}}\sum_{k>j} \left\lvert ke^{-d k^2}\sin (2\pi k x)\right\rvert \leq \sum_{k>j} 4ke^{-d k^2}\leq  \sum_{k>j} 4k^2e^{-d k^2}
\]
and we can apply the previous bounds for $\sum_{k\geq 6}2k^2e^{-dk^2/3}$
Thus, we have the following bounds for $d>1$, $t_1=\pm 1$, and $t_2\in\{-1,-\frac12,\frac12,1\}$, where $\tfo'(t_1,j)$ and $\tft'(t_2,j)$ indicate the truncation of the sums involved in $\tfo'(t_1),\tft'(t_2)$ after $j$ terms:
\begin{equation}
\label{f1pf2p4ptlowerbounds}
\begin{split}
\tfo'(t_1,2)-\frac{e^{-4d}}{8}&< \tfo'(t_1)< \tfo'(t_1,2)+\frac{e^{-4d}}{8}\\
\tft'(t_2,5)-4e^{-25d/3}&< \tft'(t_2)< \tft'(t_2,5)+4e^{-25d/3}.
\end{split}    
\end{equation}

Finally, we need bounds for $\tft(\pm\frac12)$.
Again, using the chain rule, we obtain
\begin{align}
\tth''(d/3,-1/2)&=\left(\sum_{k\geq 2}2ke^{-\frac{d}{3}k^2}
\frac{\cot(2\pi x)\sin(2\pi kx)-k\cos(2\pi kx)}{\sin^2(2\pi x)}\right)
\bigg|_{x=1/3}\\
&=\sum_{k\geq 2}-8/3k e^{-\frac{d}{3}k^2}\left (k \cos(2\pi k/3)+\frac{1}{\sqrt{3}}\sin(2\pi k/3)\right).
\end{align}
Likewise, 
\[
\tth''(d/3,1/2)=\sum_{k\geq 2}-8/3k e^{-\frac{d}{3}k^2}\left (k \cos(2\pi k/3)-\frac{1}{\sqrt{3}}\sin(2\pi k/3)\right).
\]
Note 
\[
\left\lvert-8/3k e^{-\frac{d}{3}k^2}\left (k \cos(2\pi k/3)+\frac{1}{\sqrt{3}}\sin(2\pi k/3)\right) \right \rvert \leq \frac83 k(k+1)e^{-\frac{d}{3}k^2}, 
\]
so using the fact that $(k+6)(k+7)\leq 42e^{k/3}$ for $k\geq 0$ yields 
\begin{align*}
\left\lvert \sum_{k\geq 6} -8/3k e^{-\frac{d}{3}k^2}\left (k \cos(2\pi k/3)+\frac{1}{\sqrt{3}}\sin(2\pi k/3)\right)\right\rvert \leq \sum_{k\geq 6}\frac83 k(k+1)e^{-\frac{d}{3}k^2}\\
%= \frac83 \sum_{k\geq 0}(k+6)(k+7)e^{-\frac{d}{3}(k+6)^2}\\
\leq e^{-25d/3}\frac83 e^{-11/3}\sum_{k\geq 0}(k+6)(k+7)e^{-\frac{d}{3}(k^2+12k)}
%&\leq e^{-25d/3}\frac83 e^{-11/3}\sum_{k\geq 0}42e^{k/3}e^{-(k^2+12k)/3}\\
\leq e^{-25d/3}\frac83 e^{-11/3}\sum_{k\geq 0}42e^{-4k}
%&=\frac{112e^{-25d/3}}{e^{11/3}(1-e^{-4})}
< 5e^{-25d/3}.
\end{align*}
Thus, we obtain our final bounds 
\begin{align}\label{f2pp4ptlowerbound}
\tft''(\pm \frac12,5)-5e^{-25d/3}<\tft''(\pm \frac12)< \tft''(\pm \frac12,5)+5e^{-25d/3}.
\end{align}
As a final remark, it is straightforward to check that  the leftmost lower bounds in \eqref{6ptsmallthetavalbounds}, \eqref{f1pf2p4ptlowerbounds}, and \eqref{f2pp4ptlowerbound} are positive for all $d>1$.
%%%%%%%%%%%%%%%%%%%%%%%%%%%%%%%%%%%%%%%%%%%%%%%%%%%%%
\subsection{$\tf \geq \tg$ for small $a$ and 4 points}
%%%%%%%%%%%%%%%%%%%%%%%%%%%%%%%%%%%%%%%%%%%%%%%%%%%%%
First, we prove that 
$\frac{\partial^3 \tf}{\partial t_1 \partial t_{2}^2 }(-1,1/2) \geq 0$ to complete the proof of Lemma \ref{3rdorderpartial} for $a<\pi^2$. We'll use the notation $\tfou(t_1), \tfol(t_1)$ to denote the upper and lower bounds given in the previous section and likewise for $\tft$. 
We  have 
\begin{align}
\frac{\partial^3 \tf}{\partial t_1 \partial t_{2}^2 }(-1,1/2) \nonumber
&\geq \tfol'(-1)\tftl''(1/2)-\tfou'(1)\tftu''(-1/2)\\
&=e^{-4d}\left[-440 e^{-16d/3} - 130 e^{-4d/3} + 96\right] > 0.
\end{align}
To prove this final inequality, and several others later in the section, we use the following elementary lemma that reduces to  verifying the inequality at $d=1$ which is easily checked in the case above. 
\begin{lemma}
\label{smallaexponentialhelper}
Let 
$h(d)= a_1 e^{c_1 d}+\dots+ a_n e^{c_n d} $, where 
the $c_i$'s are increasing and there is some $j$ such that $a_i\leq 0$ for $i<j$ and $a_i\geq 0$ for $i>j$. Then $h(d)> 0$ for all $d\geq 1$ if and only if $h(1)> 0$. 
\end{lemma}
\begin{proof}
    Note $h(d)\geq 0$ if and only if $h(d) e^{-c_j d}\geq 0$, and we have  
    \[
h(d)e^{-c_j d}= a_1 e^{(c_1-c_j) d}+\dots+ a_{j-1} e^{(c_{j-1}-c_j) d}+
a_j+ \dots + a_n e^{(c_n-c_j) d}. 
    \]
    By our assumptions on the $a_{i}$'s and $c_{i}$'s, for $i<j$, $a_i$ and $c_i-c_j$ are both negative, so $a_i e^{(c_i-c_j)}$ is nondereasing. For $i> j$, both $a_i$ and $c_i-c_j$ are nonnegative, and so again $a_i e^{(c_i-c_j)}$ is nondecreasing. Thus, $h(d)e^{-c_j d}>0 $ is nondecreasing, which suffices for the desired result.  
\end{proof}

%In the accompanying Mathematica notebook, we expand the products and simplify, resulting in an analytic expression whose positivity is easy to verify. 
Next, we prove Lemma \ref{inequalitystring} for $a\le \pi^2$.
\begin{proof}
First, we'll show
\[ 
2\frac{\partial \tf}{\partial t_1 }(-1,1/2) 
< \frac{\partial^2 (\tf-\tg)}{\partial t_1 \partial t_2}(-1,1/2).
\]
Using the bounds from the previous section,
\begin{align}
2\frac{\partial \tf}{\partial t_1 }(-1,1/2) &\leq 2
\left [ \tfou'(-1)\tftu(1/2)-\tfol'(1)\tftl(-1/2)
\right],\\
\frac{\partial^2 \tf}{\partial t_1 \partial t_2 }(-1,1/2)&=\tfo'(-1)\tft'(1/2)+\tfo'(1)\tft'(-1/2)\\
    &\geq \tfol'(-1)'\tftl(1/2)+\tfol'(1)\tftl(-1/2),
\end{align}
from which we obtain 
\[
\frac{\partial^2 \tf}{\partial t_1 \partial t_2 }(-1,1/2)- 2\frac{\partial \tf}{\partial t_1 }(-1,1/2) \geq 
e^{-4d}
\left[ 
95/2 - 8/5 e^{-7d/3} - 191/2 e^{-4d/3} - e^{-d/3}/2
\right]
> 0 
\]
%We show the desired inequality in a similar fashion as with $\frac{\partial^3 \tf}{\partial t_1 \partial t_{2}^2 }(-1,1/2)$ in the Mathematica file. 

It remains to show
\[
4(\tf(-1,1)-\tf(-1,1/2)), \frac{\partial \tf}{\partial t_2}(-1,1)
<
2\frac{\partial \tf}{\partial t_1 }(-1,1/2). 
\]
Just as above, we obtain 
\begin{align*}
 2\frac{\partial \tf}{\partial t_1 }(-1,1/2)-4(\tf(-1,1)-\tf(-1,1/2)) \geq \\
 e^{-4d}\left[31/2 - 219/10 e^{-16d/3} - (16 e^{-3d})/25 - 8/5 e^{-7d/3} - 
 427/10 e^{-4d/3} - (4 e^{-d/3})/25
 \right]
> 0. 
\end{align*}
Similarly, 
\begin{align*}
2\frac{\partial \tf}{\partial t_1 }(-1,1/2)-\frac{\partial \tf}{\partial t_2}(-1,1)
\geq \\
e^{-4d}\left[
47/2 - 18 e^{-25 d/3} - 8 e^{-13 d/3} - (18 e^{-3d})/25 - 
 8/5 e^{-7d/3} - 127/2 e^{-4d/3} - (2 e^{-d/3})/25
 \right]> 0. 
\end{align*}

%Again with our bounds we have:
%\begin{align}
% 4(\tf(-1,1)-\tf(-1,1/2)) &=4 \left[ \tfo(-1)(\tft(1)-\tft(-1/2))+ \tfo(1)(\tft(-1)-\tft(-1/2))\right] \\  
% &\leq 
% 4\left[ 
%\tfou(-1)(\tftu(1)-\tftl(-1/2)+\tfou(1)(\tftu(-1)-\tftl(-1/2))
% \right]\\
% 2\frac{\partial \tf}{\partial t_1 }(-1,1/2)
% &\geq 
%2
%\left [ \tfol'(-1)\tftl(1/2)-\tfou'(1)\tftu(-1/2)
%\right].
%\end{align}
%The proof that the corresponding difference is nonnegative is in the notebook. 

%Finally, 
%\begin{align}
%    \frac{\partial \tf}{\partial t_2}(-1,1) \leq
%    \tfou(-1)\tftu'(1)-\tfol(1)\tftl'(-1)
%\end{align}
%which we compare in Mathematica with the upper bound for $2\frac{\partial \tf}{\partial t_1 }(-1,1/2)$ given above.
\end{proof}

%%%%%%%%%%%%%%%%%%%%%%%%%%%%%%%%%%%%
\subsection{$\tf \geq \tg$ for small $a$ and 6 points}
%%%%%%%%%%%%%%%%%%%%%%%%%%%%%%%%%%%%
%%%%%%%%%%%%%%%%%%%%%%%%%%
\subsubsection{Satisfying Necessary Conditions}
%%%%%%%%%%%%%%%%%%%%%%%%%
Recall we aim to show 
\begin{align}
\frac{\partial (\tf-\tg)}{\partial t_1}(-1,-1)> 0\\
\frac{\partial (\tf-\tg)}{\partial t_1}(-1,1/2) > 0\\
\frac{\partial (\tf-\tg)}{\partial t_1}(1,-1/2)< 0.
\end{align}
\begin{proof}
First, we'll compute bounds for $a_{1,0},a_{0,2}$.
Recall
\begin{align}
2a_{1,0}&=\tfo(1)\tft(-1/2)-\tfo(-1)\tft(1/2)+\tfo(-1)\tft'(1/2)\\
\frac94 a_{0,2}&= \tfo(-1)(\tft(-1)-\tft(\frac12)+\frac32 \tft'(\frac12)).
\end{align}

Thus, using our bounds,
\begin{align}
\tfol(1)\tftl(-1/2)-\tfou(-1)\tftu(1/2)+\tfol(-1)\tftl'(1/2)
<
2a_{1,0}\\
2a_{1,0}
> \tfou(1)\tftu(-1/2)-\tfol(-1)\tftl(1/2)+\tfou(-1)\tftu'(1/2)\\
\tfol(-1)(\tftl(-1)-\tftu(1/2)+\frac32 \tftl'(1/2)) < 
\frac94 b_{0,2}
< \tfou(-1)(\tftu(-1)-\tftl(1/2)+\frac32 \tftu'(1/2)).
\end{align}
Call these upper and lower bounds $a^{u}_{i,j},a^{l}_{i,j}$ respectively. Now using those bounds, we compute
\begin{align}
\frac{\partial (\tf-\tg)}{\partial t_1}(-1,-1)&=
\tfo'(-1)\tft(-1) - (a_{1,0}-a_{0,2})\\
&>  \tfol'(-1)\tftl(-1)-a^{u}_{1,0}+a^{l}_{0,2}\\
&\geq e^{-3d}\left[-2 + 1123/100 e^{-4d/3} - (2429 e^{-d})/200 + 2 e^{2d/3}\right] > 0.
\end{align}
This final inequality is shown by checking that $-2 + 1123/100 e^{-4d/3} - (2429 e^{-d})/200 + 2 e^{2d/3}$ is positive with positive derivative at $d=1$ and then applying Lemma \ref{smallaexponentialhelper} to show its derivative is nonnegative for all $d>1$.   

The other conditions are similar. We must check the positivity of the lower bound 
\begin{align}
    \frac{\partial (\tf-\tg)}{\partial t_1}(-1,1/2) \geq
   -(1629/200) e^{-13d/3}-(2429 e^{-4d})/200-2 e^{-3d}+8 e^{-7d/3} > 0
 \end{align}
 by checking at $d=1$ and applying Lemma \ref{smallaexponentialhelper}. 
 Finally, the negativity of the upper bound 
\begin{align}
    \frac{\partial (\tf-\tg)}{\partial t_1}(1,-1/2)\leq 
    -e^{-3d}( 2 + (7397 e^{-7d/3})/1800 + 1621/200 e^{-4d/3} - (2429 e^{-d})/200) < 0
\end{align}
follows from checking positivity of $2 + (7397 e^{-7d/3})/1800 + 1621/200 e^{-4d/3} - (2429 e^{-d})/200$ and its derivative at $d=1$, and then applying Lemma \ref{smallaexponentialhelper} to its derivative, $e^{-d}(2429/200 - (51779 e^{-4d/3})/5400 - (1621 e^{-d/3})/150)$, as we did with the bound of $\frac{\partial (\tf-\tg)}{\partial t_1}(-1,-1)$. 
\end{proof}

\subsubsection{Computations for the Linear Approximation Bound}
\label{proofoflinapproxbound}
We have the expansions
\begin{align}
    A=\phi(-1/2)
    &=\left(\tfo(1)-\tfo'(1)\right)\tft(-1/2)-a_{0,0}+\frac{a_{0,1}}{2}-\frac{a_{0,2}}{2}\\
    &= \left(\frac12\tfo(1)-\tfo'(1)\right)\tft(-1/2)+\frac16\tfo(-1)\tft'(1/2)-\tfo(-1)\left(\frac29 \tft(-1)+\frac{5}{18}\tft(1/2)\right)\\
    B=\phi'(-1/2)&=(\tfo(1)-\tfo'(1))\tft'(-1/2)-a_{0,1}+a_{0,2}\\
    &=a_{0,2}-\tfo'(1)\tft'(-1/2)\\
    &=\frac49 \tfo(-1)\left(\tft(-1)-\tft(1/2)+\frac32 \tft'(1/2)\right)-\tfo'(1)\tft'(1/2)\\
    C=\phi''(-1/2)&=\left(\tfo(1)-\tfo'(1)\right)\tft''(-1/2)-2a_{0,2}\\
    &=\left(\tfo(1)-\tfo'(1)\right)\tft''(-1/2)-\frac89\tfo(-1)\left(\tft(-1)-\tft(1/2)+\frac32 \tft'(1/2)\right).
\end{align}

It remains to use these expansions to prove Lemmas \ref{phithirdderiv} and \ref{ABCLemma}.
\begin{proof}
To show Lemma \ref{phithirdderiv}, we prove the stronger statement  $\tfo(1)-2\tfo'(1)\geq 0$.
From our aforementioned bounds,
\begin{align}
\tfo(1)-2\tfo'(1)&\geq \tfol(1)-\tfou'(1) \geq 1 - (1427 e^{-4d})/100 - 2 e^{-d}> 0.
\end{align}

Onto Lemma \ref{ABCLemma}, where we must first show $A,C>0$. We have 
\begin{align}
2A&=\left(\tfo(1)-2\tfo'(1)\right)\tft(-1/2)+\frac13\tfo(-1)\tft'(1/2)-\tfo(-1)\left(\frac49 \tft(-1)+\frac{5}{9}\tft(1/2)\right)\\
&\geq \left(\tfol(1)-2\tfou'(1)\right)\tftl(-1/2)+\frac13\tfol(-1)\tftl'(1/2)-\tfou(-1)\left(\frac49 \tftu(-1)+\frac{5}{9}\tftu(1/2)\right)\\
&\geq 
2951/600 e^{-16d/3}+4879/600 e^{-13d/3}-(2429 e^{-4d})/200+2 e^{-3d} > 0.
\end{align}
The final quantity, when multiplied by $e^{4d}$ is convex in $d$, so we just check the value and derivative of this product at $d=1$. 
Similarly, we calculate
\begin{align}
    C&=\left(\tfo(1)-\tfo'(1)\right)\tft''(-1/2)-2b_{0,2}\\
    &\geq \left(\tfol(1)-\tfou'(1)\right)\tftl''(-1/2)-2a^{u}_{0,2}\\
    &\geq (3687 e^{-7d})/25 - 688/45 e^{-19d/3} - 9377/225 e^{-16d/3} - 
  24 e^{-3d} + 16 e^{-7d/3} > 0
\end{align}
by applying Lemma \ref{smallaexponentialhelper} to  $- 688/45 e^{-19d/3} - 9377/225 e^{-16d/3} - 
  24 e^{-3d} + 16 e^{-7d/3}$. 
Finally, with the additional assumption that $B<0$, we must show $A^2-BC<0$. To bound $B^2$ above, we bound $B$ below (since $B<0$). We have
\begin{align}
B&=b_{0,2}-\tfo'(1)\tft'(-1/2)\\
&\geq 14309/450 e^{-16d/3} - 65/4 e^{-13d/3}.
\end{align}
We have arrived at the following lower bounds for $A,B,C$, with the $A$ and $C$ bounds shown to be positive: 
\begin{align}
    A_l&:=2951/600 e^{-16 d/3}+ 4879/600 e^{-13 d/3} - 2429/200 e^{-4 d} + 
 2 e^{-3 d}\\
    B_l&:=14309/450e^{-16d/3} - 65/4e^{-13d/3}\\
    C_l&:= 3687/25e^{-7d} - 688/45e^{-19d/3} - 9377/225e^{-16d/3} - 
 24 e^{-3 d} + 16e^{-7d/3}.
\end{align}
We now
show $B_l^2-2A_lC_l<0$, which is equivalent to $B_l/C_l-2A_l/B_l>0$. To do so, we plug in $d=1$ and see the inequality holds there. Then, we claim $B_l/C_l$ is increasing in $d$, while $2A_l/B_l$ is decreasing in $d$. The sign of the derivative of $B_l/C_l=\frac{e^{13d/3}}{e^{13d/3}C_l}$ depends only on the sign of 
\[
e^{3d}(B_le^{13 d/3})' C_l - B_l (e^{13 d/3} C_l)' \geq -520 + (400652 e^{-d})/225 - (114472 e^{-d/3})/75 + 520 e^{2d/3}> 0
\]
where we obtain the final inequality by checking positivity of  $-520 + (400652 e^{-d})/225 - (114472 e^{-d/3})/75 + 520 e^{2d/3}$ and its derivative  at $d=1$, and applying Lemma \ref{smallaexponentialhelper} to its derivative. Likewise, the derivative of $2A_l/B_l$ depends only on the sign of  
\begin{align*}
e^{3d}(2e^{13d/3} (A_le^{13d/3})' B_l - A_l (e^{13d/3} Bl)') \leq\\
-(130/3) + (182785597 e^{-7d/3})/540000 - (
 34756561 e^{-2d})/67500 + (4626181 e^{-d})/21600 <0,
\end{align*}
and the final inequality depends on the same checks as with the previous case. All of these checks at $d=1$ and algebraic simplifications are verified in an accompanying Mathematica notebook. 
\end{proof}

%%%%%%%%%%%%%%%%%%%%%%%%%%%%%%%%%%%%%%%%%%%%%%%%
\section{Technical Estimates and Computations for large $a$}
\label{largeaAppendix}
%%%%%%%%%%%%%%%%%%%%%%%%%%%%%%%%%%%%%%%%%%%%%%%
Throughout, assume that $a\geq 9.6$. We'll set $\epsilon=\frac{1}{1000}$ so that for all $a \geq  9.6$:
\begin{align}
    \epsilon > 2\sum_{n\geq 1} e^{-an} \text{ and }
    \epsilon > 5e^{-a}.
\end{align} Then set 
%$\delta=\frac15\geq  e^{-5a/12}(4+9\epsilon)$ for all $a\geq 9.6$ and
$\epsilon_2=\frac{1}{100}> 4(1+\epsilon)^2\sum_{n\geq 1} e^{-2(9.6)n/3}$. \\

In the large $a$ case, it is preferable to use the following formula for $\theta$ because of its rapid convergence:
\[
\theta(c;x) =c^{-1/2} \sum^{\infty}_{k=-\infty} e^{-\frac{\pi(k+x)^2}{c}}.
\]
Thus, we'll use the formulas
\begin{align}
f_1(x_1)= \sqrt{\frac{a}{\pi}}\theta(\frac{\pi}{a};x_1)=\sum^{\infty}_{k=-\infty} e^{-a(k+x_1)^2},\\
f_2(x_2)=\sqrt{\frac{3a}{\pi}}\theta(\frac{\pi}{3a},x_2)=\sum^{\infty}_{k=-\infty} e^{-3a(k+x_2)^2}.
\end{align}
%yielding in the $A_2$ case
%\[
%F(x)=\tfo(x_1)f_2(x_2)+\tfo(\frac12-x_1)f_2(\frac12-x_2),
%\]
%and in the $L$ case,
%\[
%F(x)=\tfo(x_1)f_2(x_2). 
%\]

%%%%%%%%%%%%%%%%%%%%%%%%%%%%%%%
\subsection{Basic Lemmas and Other Estimates}
%%%%%%%%%%%%%%%%%%%%%%%%%%%%%
We first establish a couple basic workhorse lemmas bounding $\theta$ and $\tth'$.

\begin{lemma}
\label{basicvalues}
For $x=\arccos(t)/(2\pi)\in[0,1/2]$,
\begin{align}
e^{-a x^2}+e^{-a(x-1)^2} &< \theta(\frac{\pi}{a};x)
< 
%(1+\epsilon) (e^{-a x^2}+ e^{-a(x-1)^2})
%=
(1+\epsilon)e^{-ax^2}(1+e^{-a(1-2x)})\leq 2(1+\epsilon)e^{-ax^2}
\\
e^{-a x^2} &< \theta(\frac{\pi}{a};x)\leq e^{-a x^2}(1+ 2\sum_{n\geq 1}e^{-a (n^2-2nx)}).
\end{align}

\end{lemma}
\begin{proof}
Recall $\theta(\frac{\pi}{a};x)=\sum_{n\in \mathbb{Z}}e^{-a(n+x)^2}$.
The lower bounds follow from simply truncating the series. To obtain the first upper bound, observe
\begin{align}
\sum_{n\in \mathbb{Z}}e^{-a(n+x)^2}
%&=\sum_{n\geq 0}e^{-a(n+x)^2} + \sum_{n\leq -1}e^{-a(n+x)^2}\\
&= e^{-ax^2}\sum_{n\geq 0}e^{-a(n^2+2 n x)}+e^{-a(x-1)^2}\sum_{n\geq 0}e^{-a(n^2+2(1-x)n)}\\
%\leq e^{-ax^2} \sum_{n\geq 0} e^{-a n^2}+e^{-a(x-1)^2}\sum_{n\geq 0}e^{-a(n^2)}
&\leq (e^{-ax^2}+e^{-a(1-x)^2})\sum_{n\geq 0}e^{-a n}
< (1+\epsilon)(e^{-ax^2}+e^{-a(1-x)^2}).
\end{align}
The second upper bound follows in a similar fashion using the fact that for $n\geq 1$, the $n$th term is at least as large as the $-(n+1)$th term. 
\end{proof}

In the remainder of  this section we shall use the dependent variables as in \eqref{tuDef}:
\begin{equation}\label{tux1x2}
x=x_1=\frac{\arccos(t_1)}{2\pi}, \hspace{.5 cm} u=x_2/\sqrt{3}=\frac{\arccos(t_2)}{2\pi}.
\end{equation}

Lemma~\ref{basicvalues} implies that for $t_2\in [ -\frac12,1]$, we have:
\begin{equation}
\begin{split}
\label{basicvaluesapplication}
e^{-3a u^2}&< \tft(t_2)< (1+\epsilon)e^{-3a u^2}\\
1&< \tfo(1),\tft(1)< 1+\epsilon\\
2e^{-a/4}&< \tfo(-1)< 2(1+\epsilon)e^{-a/4}\\
2e^{-3a/4}&< \tft(-1)< 2(1+\epsilon)e^{-3a/4}.
\end{split}
\end{equation}
These particular bounds follow immediately  except for the first, where we use that if $t_2 \in [-1/2,1]$, then $u\in[0,1/3]$ and so 
\[
2\sum_{n\geq 1}e^{-3a (n^2-2nx)}\leq 2\sum_{n\geq 1}e^{-3a(n-2n(1/3))}=2\sum_{n\geq 1}e^{-a}< \epsilon.
\]

%%%%%%%%%%%%%%%%%%%%%%%%%%%%%%%%%%%%%%%%%%%%%%%%%%%%%%%%%%%%
\begin{lemma}
\label{basicderivs}
For  $x=\frac{\arccos(t_1)}{2\pi}\in (0,1/2)$, 
\begin{align}
\frac{a e^{-a x^2}(x -(1-x)e^{-a(1-2x)})}{\pi \sin(2\pi x)}
\leq
\tth'(\frac{\pi}{a};t_1)
\leq
\frac{a x e^{-a x^2}}{\pi \sin(2\pi x)}.
\end{align}
\end{lemma}
\begin{proof}
Using $t_1=\cos(2\pi x)$, we have 
\[
\tth'(\frac{\pi}{a};t_1)=\frac{\sum_{n\in \mathbb{Z}}-2a(n+x)e^{-a(n+x)^2} }{-2\pi \sin(2\pi x)}=\frac{a\sum_{n\in \mathbb{Z}}(n+x)e^{-a(n+x)^2} }{\pi \sin(2\pi x)}.
\]
Let $s_n=(n+x)e^{-a(n+x)^2}$.
Now to obtain the lower bound, we verify that for $n\geq 1$, $s_n+s_{-n-1}\geq 0$. Thus, $s_0+s_{-1}$ yields a lower bound. Similarly, we check $s_n+s_{-n}\leq 0$, so $s_{0}$ yields an upper bound. It will be independently useful that $s_0\geq s_{-1}$ for $\frac14\leq x\leq \frac12$. Indeed, in this case, taking $v=\frac12 -x$,
\[
ae^{-ax^2}(x-(1-x)e^{-a(1-2x)}=
ae^{-ax^2}(1/2-v-(1/2+v)e^{-2av})
\]
and 
$(1/2-v-(1/2+v)e^{-2av})$ is concave in $v$ for all 
$a\geq 9.6$, $v\in[0,\frac14]$, and so it suffices to check the inequality for $v=0,\frac14$, which are both immediate. 
\end{proof}

As a consequence of Lemma \ref{basicderivs}, we obtain
for $a\geq 9.6$ and $t_2\in [-1/2,1/2]$ that  
\begin{equation}
\label{thetaderivin[-1/2,1/2]}
\tft'(t_2)>
\frac{3au e^{-3a u^2}(1-\epsilon)}{\pi \sin(2\pi u)}.
\end{equation}
Indeed, for such $t_2$, $u\in [1/6,1/3]$, so $u-1\geq -5u$ and $1-2u\geq \frac13$ which gives
\[
\tth'(\frac{\pi}{3a},t_2)\geq \frac{3a e^{-a u^2}(u-(1-u)e^{-3a(1-2u)})}{\pi \sin(2\pi u)}\geq  \frac{3au e^{-a u^2}(1-5e^{-a})}{\pi \sin(2\pi u)}> \frac{3au e^{-a u^2}(1-\epsilon)}{\pi \sin(2\pi u)}. 
\]

In particular, 
$\tft'(\frac12)> \frac{(1-\epsilon)a e^{-a/12}}{\sqrt{3 \pi}}$, and $\tft'( -\frac12)> \frac{2(1-\epsilon)a e^{-a/3}}{\sqrt{3 \pi}}$.

We also need to obtain bounds on $\tth'(\frac{\pi}{a};\pm1)$. For $a\geq 9.6$, 
\begin{equation}
\label{thetaderivpm1}
\begin{split}
 \frac{a}{2\pi^2}(a-2)e^{-a/4}
 &< \tth'(\frac{\pi}{a};-1)
 <
 \frac{a}{2\pi^2}(a-2+\epsilon)e^{-a/4}\\
\frac{(1-\epsilon_2)a}{2\pi^2}
 &< \tth'(\frac{\pi}{a};1)
 <
 \frac{a}{2\pi^2}.
\end{split}
\end{equation}
For $a\geq 21$,
\begin{equation}
\label{thetaderivpm1>21}
 \frac{(1-\epsilon)a}{2\pi^2}
 < \tth'(\frac{\pi}{a};1)
 <
 \frac{a}{2\pi^2}.
\end{equation}

  We first have 
 \[
 \tth'(\pi/a,-1)=\frac{a}{2\pi^2}\sum_{n\in \mathbb{Z}}\left[ 2a(n+1/2)^2-1\right]e^{-a(n+1/2)^2}.
 \]
 We get an easy lower bound by just taking the $n=0,-1$ terms.
 For an upper bound, we bound the tail:
 \begin{align}
2\sum_{n\geq1}\left[ 2a(n+1/2)^2-1\right]e^{-a(n+1/2)^2}
%\leq e^{-a/4}4\sum_{n\geq1}\left[ a(n+1/2)^2-1\right]e^{-a(n^2+n)}\\
&\leq e^{-a/4}4\sum_{n\geq1}\left[ a(2n)^2\right]e^{-a(n^2+n)}\\
&\leq e^{-a/4}16\sum_{n\geq 1}an^2e^{-an^2}e^{-an}
\leq e^{-a/4}16/1000\sum_{n\geq 1}e^{-an}\leq \epsilon e^{-a/4},
\end{align}
where we have used that $an^2e^{-an^2}\leq \frac{1}{1000}$ since
$be^{-b}\leq \frac{1}{1000}$ for $b\geq 9.6$.
Thus, we obtain the bounds 
\[
 \frac{a}{2\pi^2}(a-2)e^{-a/4}
 \leq \tth'(\pi/a,-1)
 \leq
 \frac{a}{2\pi^2}(a-2+\epsilon)e^{-a/4}.
\]

Next, 
\[
 \tth'(\pi/a,1)=\frac{a}{2\pi^2}\sum_{n\in \mathbb{Z}}\left[ 1-2an^2\right]e^{-an^2}.
 \]
 By just using the $n=0$ term, we get an easy upper bound. Now bounding the tail, we have

 \begin{align}
      \left\lvert 2 \sum_{n\geq 1}\left[ 1-2an^2\right]e^{-an^2}\right\rvert
      &\leq
      %2\sum_{n\geq1}2an^2 e^{-an^2}\\
      %&=
      4 \sum_{n\geq 1}an^2 e^{-an^2}
      \leq 4\sum_{n\geq 1} e^{-2an^2/3}
      \leq 4\sum_{n\geq 1} e^{-2an/3}
      \leq 4\sum_{n\geq 1} e^{-2(9.6)n/3}
      &< \epsilon_2
 \end{align}
 since $be^{-b}\leq e^{-2b/3}$ for $b\geq 9.6$  
Thus, we have 
\begin{align}
\frac{(1-\epsilon_2)a}{2\pi^2}
 \leq &\tth'(\frac{\pi}{a};1)
 \leq
 \frac{a}{2\pi^2}
\end{align}
and when $a\geq 21$, we obtain that the tail is at most $4\sum_{n\geq 1} e^{-14n}<\epsilon$ in the same manner, and so in this case, 
\[
\frac{(1-\epsilon)a}{2\pi^2}
 \leq \tth'(\frac{\pi}{a};1)
 \leq
 \frac{a}{2\pi^2},
\]
as desired.

We finally need bounds on $\tth''(\frac{\pi}{3a},\pm \frac{1]}{2})$.  First, we have
\[
\tth''(\frac{\pi}{3a},t_2)=\frac{3a}{\pi^2 \sin(2\pi u)^2}\sum_{n\in \mathbb{Z}}e^{-3a(n+u)^2} \left( -\frac12 + \pi\cot(2\pi u)(n+u)+3a(n+u)^2  \right)
\]
so that
\small
\begin{align*}
\tth''(\frac{\pi}{3a},\frac12)&=\frac{4a}{\pi^2 }\sum_{n\in \mathbb{Z}} 
e^{-3a(n+\frac16)^2} 
\left( -\frac12 + \frac{\pi}{\sqrt{3}}(n+\frac16)+3a(n+\frac16)^2  \right)\\
&\geq    \frac{4a}{\pi^2 }e^{-a/12} \left( -\frac12 + \frac{\pi}{6\sqrt{3}}+a/12  \right)\\
\tth''(\frac{\pi}{3a},-\frac12)&=\frac{4a}{\pi^2 }\sum_{n\in \mathbb{Z}} 
e^{-3a(n+\frac13)^2} 
\left( -\frac12 - \frac{\pi}{\sqrt{3}}(n+\frac13)+3a(n+\frac13)^2  \right)\\
%&=    
%\frac{4a}{\pi^2 }e^{-a/3} \left( -\frac12 - \frac{\pi}{3\sqrt{3}}+a/3  \right) +
%\frac{4a}{\pi^2 }
%\sum_{n\neq 0}
%e^{-3a(n+\frac13)^2} 
%\left( -\frac12 - \frac{\pi}{\sqrt{3}}(n+\frac13)+3a(n+\frac13)^2  \right)\\
&\leq 
\frac{4a}{\pi^2 }e^{-a/3} \left( -\frac12 - \frac{\pi}{3\sqrt{3}}+a/3  \right) +
\frac{4a}{\pi^2 }
e^{-a/3}\sum_{n\neq 0}
e^{-3a(n^2+2n)} \left(
\left| -\frac12\right|  +\left|\frac{\pi}{\sqrt{3}}(n+\frac13)\right|+\left|3a(n+\frac13)^2\right|  \right)\\
%&\leq 
%\frac{4a}{\pi^2 }e^{-a/3} \left( -\frac12 - \frac{\pi}{3\sqrt{3}}+a/3  \right) +
%\frac{4a}{\pi^2 }
%e^{-a/3}2\sum_{n\geq 1}
%e^{-3an^2+2n} 8an^2\\
&\leq 
\frac{4a}{\pi^2 }e^{-a/3} \left( -\frac12 - \frac{\pi}{3\sqrt{3}}+a/3  \right) +
\frac{4a}{\pi^2 }
e^{-a/3}16\sum_{n\geq 1}
e^{-an^2} an^2\\
&\leq 
\frac{4a}{\pi^2 }e^{-a/3} \left( -\frac12 - \frac{\pi}{3\sqrt{3}}+a/3  \right) +
\frac{4a}{\pi^2 }
e^{-a/3}16\sum_{n\geq 1}
e^{\frac{-an^2}{2}}\\
&\leq 
\frac{4a}{\pi^2 }e^{-a/3} \left( -\frac12 - \frac{\pi}{3\sqrt{3}}+a/3  \right) +
\frac{4a}{\pi^2 }
e^{-a/3}
=\frac{4a}{\pi^2 }e^{-a/3} \left( \frac12 - \frac{\pi}{3\sqrt{3}}+a/3  \right),
\end{align*}
\normalsize
where for the first inequality we have just thrown away all the terms except for which $n=0$ (these are certainly all positive), and for the second string of inequalities, we have used that $be^{-b}\leq e^{-b/2}$ for $b\geq 9.6$ and then used a comparison with a geometric series to obtain $16\sum_{n\geq 1}
e^{\frac{-an^2}{2}}\leq 16\sum_{n\geq 1}
e^{\frac{-an}{2}}\leq 1$.

This leads to the bounds for $\tft''$:
\begin{equation}
\label{theta2ndderiv}
\begin{split}
\tft''(\frac12)
&\geq    \frac{4a}{\pi^2 }e^{-a/12} \left( -\frac12 + \frac{\pi}{6\sqrt{3}}+a/12  \right)
\\
%\tth''(\frac{\pi}{3a},-\frac12)
%&=\frac{4a}{\pi^2 }\sum_{n\in \mathbb{Z}} 
%e^{-3a(n+\frac13)^2} 
%\left( -\frac12 - \frac{\pi}{\sqrt{3}}%(n+\frac13)+3a(n+\frac13)^2  \right)
%\\
\tft''(-\frac12) 
&\leq 
\frac{4a}{\pi^2 }e^{-a/3} \left( \frac12 - \frac{\pi}{3\sqrt{3}}+a/3  \right). 
\end{split}
\end{equation}

As in the case of $a\leq 9.6$, we will use $\tfou$ and $\tfol$ to denote the bounds for $\tfo$ produced in this section and likewise for $\tft$. 
%%%%%%%%%%%%%%%%%%%%%%%%%
\subsection{Intermediate $a$ and 4 points}
%%%%%%%%%%%%%%%%%%%%%%%%%%

\subsubsection{Calculations for Lemma \ref{3rdorderpartial}}
First, we show Lemma \ref{3rdorderpartial} for $a\geq 9.6$
\begin{proof}
We need to show 
\[
\tth'(\frac{\pi}{a};-1)\tth''(\frac{\pi}{3a},\frac12)-\tth'(\frac{\pi}{a};1)\tth''(\frac{\pi}{3a},-\frac12)> 0.
\]
\end{proof}
Using the bounds of lemmas \ref{thetaderivpm1} and \ref{theta2ndderiv}, we have:
\begin{align}
    \tth'(\frac{\pi}{a};-1)\tth''(\frac{\pi}{3a},\frac12)-\tth'(\frac{\pi}{a};1)\tth''(\frac{\pi}{3a},-\frac12)\geq
    \\
    \frac{a(a-2)e^{-a/4}}{2\pi^2}
      \frac{4a}{\pi^2 }e^{-a/12} \left( -\frac12 + \frac{\pi}{6\sqrt{3}}+a/12  \right)
    -
    \frac{a}{2\pi^2}
    \frac{4a}{\pi^2 }e^{-a/3} \left( \frac12 - \frac{\pi}{3\sqrt{3}}+a/3  \right)\\
    =\frac{a^2 e^{-a/3} (18 + 3 a^2 + 2 a (-18 + \sqrt{3} \pi))}{18\pi^4}.
\end{align}
The inner expression is quadratic in $a$. It is straightforward to check that it's positive with positive slope at $a=9.6$ and convex.

%%%%%%%%%%%%%%%%%%%%%%%%%%%%%%%%
%Inequality String
%%%%%%%%%%%%%%%%%%%%%%%%%%%%%
\subsubsection{Proof of Lemma \ref{inequalitystring}}

Next, we'll prove Lemma \ref{inequalitystring} for $9.6<a\leq 21$.
\begin{proof}
To obtain
\[ 
2\frac{\partial \tf}{\partial t_1 }(-1,1/2) 
< \frac{\partial^2 (\tf-\tg)}{\partial t_1 \partial t_2}(-1,1/2),
\]
we first use the bounds from inequalities \eqref{thetaderivin[-1/2,1/2]}, \eqref{thetaderivpm1} and Lemma \ref{basicvalues} to arrive at the following two inequalities:
\begin{align}
    2\frac{\partial \tf}{\partial t_1 }(-1,1/2) &\leq 2\left[
    \frac{a}{2\pi^2}(a-2+\epsilon)e^{-a/4}(1+\epsilon)e^{-a/12}-(1-\epsilon_2) \frac{a}{2\pi^2}e^{-a/3}  \right]\\
    &=\frac{e^{-a/3}}{\pi^2}\left[ 
    (a-2+\epsilon)(1+\epsilon)-(1-\epsilon_2)
    \right],\\
    \frac{\partial^2 (\tf-\tg)}{\partial t_1 \partial t_2}(-1,1/2)&\geq 
    %\frac{ae^{-a/12}(1-\epsilon)}{\pi\sqrt{3}}\frac{a e^{-a/4}(a-2)}{2\pi^2}
    %+
    %\frac{2ae^{-a/3}(1-\epsilon)}{\pi\sqrt{3}}\frac{a(1-\epsilon_2)}{2\pi^2}\\
     \frac{a^2 e^{-a/3}(1-\epsilon)(a-2\epsilon_2)}{2\pi^3 \sqrt{3}}.
\end{align}
Factoring out $\frac{ae^{-a/3}}{\pi^2}$ from each term, it suffices to show
\[
\frac{a(1-\epsilon)(a-2\epsilon_2)}{\sqrt{3}\pi}-\left( 2(a-2+\epsilon)-(1-\epsilon_2)\right)> 0
\]
and this is certainly true by just checking the  value and first derivative of this difference at $a=9.6$ are positive since it is quadratic in $a$ and convex. \\

Next we handle
\[
4(\tf(-1,1)-\tf(-1,1/2))
<
2\frac{\partial \tf}{\partial t_1 }(-1,1/2).
\]
We have from Lemma \ref{basicvalues} and inequality \eqref{thetaderivpm1} that
\begin{align}
    4(\tf(-1,1)-\tf(-1,1/2))& \leq 4(2(1+\epsilon)^2(1+e^{-a/2})e^{-a/4}-3e^{-a/3})\\
    2\frac{\partial \tf}{\partial t_1 }(-1,1/2)&\geq 2\left( \frac{a(a-2)e^{-a/4}}{2\pi^2}e^{-a/12}-\frac{a}{2\pi^2}(1+\epsilon)e^{-a/3}\right)\\
    &=\frac{a e^{-a/3}(a-3-\epsilon)}{\pi^2}.
\end{align}
Factoring out  $e^{-a/4}$, it suffices to show 
\[
 e^{-a/12}\left( 12+ \frac{a(a-3-\epsilon)}{\pi^2}\right)-8(1+\epsilon)^2(1+e^{-a/2})\geq 0
\]
which holds if 
 
\[
 e^{-a/12}\left( 12+ \frac{a(a-3-\epsilon)}{\pi^2}\right)-8(1+\epsilon)^2(1+e^{-9.6/2})> 0.
\]
We can check that on $[9.6,21]$, this final quantity is either increasing or concave, implying it doesn't have local minima, so it suffices to check the inequality at the endpoints $a=9.6,21$. \\

Finally, we need to show $\frac{\partial \tf}{\partial t_2}(-1,1)
<
2\frac{\partial \tf}{\partial t_1 }(-1,1/2)$. We already have a lower bound for
$2\frac{\partial \tf}{\partial t_1 }(-1,1/2)$, and from \eqref{thetaderivpm1} and Lemma \ref{basicvalues}, we compute
\begin{align}
\frac{\partial \tf}{\partial t_2}(-1,1) \leq \frac{3ae^{-a/4}}{2\pi^2}
     \left( 
  (2+2\epsilon)-(3a-2)e^{-a/2}\right).
\end{align}
Factoring $\frac{a e^{-a/4}}{2\pi^2}$, it suffices to show 
\[   
2(a-3-\epsilon)e^{-a/12}- 3(2+2\epsilon)+3(3a-2)e^{-a/2}> 0,
\]
and it's 
 straightforward to check $2(a-3-\epsilon)e^{-a/12}$ is concave for $a\leq 21$. Since also for
$a\in[9.6,21]$ and any constant $b$ satisfying $b\geq a$, we have
\[
3(3a-2)e^{-b/2}\leq 
3(3a-2)e^{-a/2},
\]
it then suffices to check 
\begin{align*}
    2(a-3-\epsilon)e^{-a/12}  -3(2+2\epsilon)+3(3a-2)e^{-21/2}&> 0 \hspace{1cm} a=21,11\\
    2(a-3-\epsilon)e^{-a/12}  -3(2+2\epsilon)+3(3a-2)e^{-11/2}&> 0 \hspace{1cm} a=9.6,
\end{align*}
which completes the proof.
\end{proof}

%%%%%%%%%%%%%%%%%%%%%%%%%%%%
\subsection{Large $a$ and 4 points}\label{largea4Pt}
%%%%%%%%%%%%%%%%%%%%%%%%%%%
Now we assume $a\geq 21$, and first present the proof of Lemma \ref{4ptboundarysections}.
\begin{proof}
Using the bounds on $\tth$ given in \eqref{thetaderivpm1} and Lemma \ref{basicvalues}, we obtain
 \begin{align}
     0\leq \frac{ae^{-a/4}}{2\pi^2} \left(
     (a-2)-(2+2\epsilon)e^{-a/2}
     \right)
     &\leq \frac{\partial \tf}{\partial t_1}(-1,1)
     \leq 
     \frac{ae^{-a/4}}{2\pi^2}(a-2+\epsilon)(1+\epsilon)
     \\
     \label{b1bounds}
    0\leq \frac{3(2-4\epsilon)ae^{-a/4}}{2\pi^2}
    %\leq 
    %\frac{3ae^{-a/4}}{2\pi^2}
     %\left( 
     %2(1-\epsilon)-(1+\epsilon)(3a-2+\epsilon)e^{-a/2}
     %\right) 
     &\leq \frac{\partial \tf}{\partial t_2}(-1,1)
     \leq 
     \frac{3ae^{-a/4}}{\pi^2}
  (1+\epsilon).
 \end{align}
 %Where for the first nontrivial $\frac{\partial \tf}{\partial t_2}(-1,1)$ inequality, we have used that 
 %\[
%(1+\epsilon)(3a-2+\epsilon)e^{-a/2}
%\leq
%3a(1+\epsilon)e^{-a/2}
%\leq 
%2\epsilon
% \]
% as verified in Mathematica.
 Thus,
 \begin{align*}
    \frac{\partial \tf}{\partial t_1}(-1,1)-\frac{\partial \tf}{\partial t_2}(-1,1) &\geq
    \frac{ae^{-a/4}}{2\pi^2}
    \left(
    a-2- 3(2+2\epsilon)-e^{-a/2}(2+2\epsilon)\right)\\
    &=\frac{ae^{-a/4}}{2\pi^2}\left( a-8-6\epsilon- e^{-a/2}(2+2\epsilon)\right)
 \end{align*}
 which is easily seen to be positive for $a\geq 21$.

Next, using  Lemma \ref{basicvalues} and equation \eqref{thetaderivpm1}, 
\begin{equation}
\begin{split}
(\tf-\tg)(-1,1/2)&=\tf(-1,1/2)-\tf(-1,1)+\frac{b_1}{4}\\
&\geq 3e^{-a/3}- 2(1+\epsilon^2)e^{-a/4}-2(1+\epsilon)^2 e^{-3a/4}             +  \frac{3(2-4\epsilon)ae^{-a/4}}{2\pi^2}\\
&=e^{-a/4}\left[ 3e^{-a/12}-2(1+\epsilon)^2(1+e^{-a/2}) +
  \frac{3(2-4\epsilon)a}{2\pi^2}
\right]\\
&\geq e^{-a/4}\left[ 3e^{-a/12}-2(1+\epsilon)^3 +  \frac{3(2-4\epsilon)a}{2\pi^2}
\right]>0.
\end{split}
\end{equation}
The quantity in the brackets is convex in $a$, so it suffices to verify the positivity  of the value and derivative of this quantity at $a=21$ which are straightforward computations.  
\end{proof}

Now we present the remaining components of the proof of Lemma \ref{4ptlinearization}.
%We shall verify several inequalities of the form $h_1(t)<h_2(t)$ for $t\in [\alpha, \beta]$ where $h_1$ and $h_2$ are both increasing functions by choosing $\delta=(\beta-\alpha)/n$ such that 
%\begin{equation}\label{linecheck}
%   h_1(\alpha+k\delta)<h_2(\alpha+(k-1)\delta),\qquad k=1,2,\ldots, n.
%   \end{equation}
%These inequalities are then rigorously verified in the %Mathematica notebook.   

\begin{proof}
Recall we have 
\[
\tf_T:=(e^{-ax^2}+e^{-a(x-1)^2})e^{-3a u^2}+ e^{-a((\frac12-x)^2+3(\frac12-u)^2)}
\]
and note that for $c<t_1<0$, $0<t_2<d$, with $t_2+c<0$, we have the following upper bounds for $\tg$
\begin{align}
\tg(t_1,t_2)\leq \tg_{c,d}(t_1,t_2)&:= \nonumber\tf(-1,1) + b_1 t_2^2 +b_1(d t_1+c t_2-cd)\\
&\leq \tfou(-1)\tftu(1)+\tfou(1)\tftu(-1)+b_1^l t_2(t_2+c)+b_1^l(d t_1)-cd b_1^u\nonumber 
\\
&\leq 2(1+\epsilon)^3 e^{-a/4}+b_1^l t_2(t_2+c)+b_1^l(d t_1)-c d b_1^u=:\tg_{c,d}^*(t_1,t_2) \label{gfinal}
\end{align}
where
\[
\tfou(-1)\tftu(1)+\tfou(1)\tftu(-1)= 2e^{-a/4}(1+\epsilon)^2(1+e^{-a/2})\leq 2e^{-a/4}(1+\epsilon)^3,
\]
and $b_1^l$ and $b_1^u$ are the bounds on  
$\frac{\partial \tf}{\partial t_2}(-1,1)$ given in \eqref{b1bounds}.

We then show that inequalities \eqref{partialaCond} hold with the choices:
\begin{align}
\tg_{-1,1}^*  \text{ on the segments } \{(\cos(2\pi \sqrt{3}/4, t_2): t_2 \in [.7,1]\} \text{ and } \{(t_1,.7): t_1 \in (-1, \cos(2\pi \sqrt{3}/4)\} \label{examplepartialacond}\\
\tg_{-1,.7}^* \text{ on the segments } \{(\cos(2\pi \sqrt{3}/4, t_2): t_2 \in [.6,.7]\} \text{ and } \{(t_1,.6): t_1 \in (-1, \cos(2\pi \sqrt{3}/4)\} \\
\tg_{-1,.6}^* \text{ on the segments } \{(\cos(2\pi \sqrt{3}/4, t_2): t_2 \in [.5,.6]\} \text{ and } \{(t_1,.5): t_1 \in (-1, \cos(2\pi \sqrt{3}/4)\},
\end{align}
thus permitting the application of Lemma \ref{linearizationlemma}. 
Here, we'll handle the case of \eqref{examplepartialacond}, and leave the other (similar) cases to the Mathematica notebook \cite{Mathematica}. 
By definition, we have 
\[
\tg_{-1,1}^*(t_1,t_2) = 2(1+\epsilon)^3 e^{-a/4}+b_1^l t_2(t_2-1)+b_1^l t_1 + b_1^u.
\]
Similarly, we have 
\[
\frac{ \partial [e^{a/4} \tg_{-1,1}^*(t_1,t_2)]}{\partial a} =\frac{3(2-4\epsilon)}{2\pi^2} t_2(t_2-1)+\frac{3(2-4\epsilon)}{2\pi^2} t_1 + \frac{3}{\pi^2}  (1+\epsilon),
\]
and since $t_2\geq 1/2$ on $\td$, it is immediate from the formulas that $ \tg_{-1,1}^*$ and $\displaystyle{\frac{\partial \left[e^{a/4} \tg_{c,d}^*(t_1,t_2)\right]}{\partial a}}$ are increasing in $t_1$ and $t_2$ on the two line segments. 
\end{proof}

Likewise, we can decompose $\tf_T$ into
\begin{align*}
\tf_t(t_1,t_2)&= (e^{-ax^2}+e^{-a(x-1)^2})e^{-3a u^2}+ e^{-a[(\frac12-x)^2 -3(\frac12-u)^2]}\\
&=e^{-a(x^2+3u^2)}+  e^{-a\left[(\frac12-x)^2+3(\frac12-u)^2 \right]} + e^{-a\left[(x-1)^2+3u^2\right]}.
\end{align*}
Since $u$ decreases as $t_2$ increases 
on $[-1,1]$ with $u(1)=0$, $u(-1)=1/2$, 
we have 
 $(e^{-ax^2}+e^{-a(x-1)^2})e^{-3a u^2}$ increasing in $t_2$ 
 and  
 $e^{-a(\frac12-x)^2} e^{-3a(\frac12-u)^2}$ decreasing in $t_2$. 
By the same reasoning, 
$e^{-ax^2}e^{-3a u^2}$ increases in $t_1$, while $e^{-a(\frac12-x)^2} e^{-3a(\frac12-u)^2}$
and $e^{-a(x-1)^2}e^{-3a u^2}$ are decreasing in $t_1$. 
Finally, 
\begin{align*}
\frac{ \partial [e^{a/4} \tf_T(t_1,t_2)]}{\partial a}\bigg|_{a=21}&=
-(x^2+3u^2-1/4)e^{-21(x^2+3u^2-1/4)}
\\ 
&\hspace{.5cm}-\left[(\frac12-x)^2+3(\frac12-u)^2-1/4 \right] e^{-21\left[(\frac12-x)^2+3(\frac12-u)^2-1/4 \right]} 
\\
&\hspace{.5cm} -\left[(x-1)^2+3u^2-1/4\right] e^{-21\left[(x-1)^2+3u^2-1/4\right]}.
\end{align*}
To break this function into a difference of increasing functions, we need the following elementary lemma, which can be proved by just checking the derivative:
\begin{lemma}
\label{partialatechnical}
Let $n_1,n_2$ be constants, and consider the function 
$
\phi(\gamma):=(n_1+\gamma) e^{-21(n_2+\gamma)}
$
for $\gamma\in \R$. Then
$\phi$ is increasing as a function of $\gamma$ for $\gamma \leq (1-21 n_1)/21$.
\end{lemma}
Now take 
$(x^2+3u^2-1/4-1/28)e^{-21(x^2+3u^2-1/4)}$ as a function of $3u^2$. Using the fact that $t_1<0$ and $t_2 \geq 1/2$ on the whole rectangle $A$ (so $1/4<x\leq 1/2$ and $0\leq 3u^2\leq 1/12$), we obtain 
\[
3u^2\leq 1/12=(1-21(1/4-1/4-1/28))/21\leq (1-21(x^2-1/4-1/28))/21,
\]
and so we may apply Lemma \ref{partialatechnical} to observe $(x^2+3u^2-1/4-1/28)e^{-21(x^2+3u^2-1/4)}$ is increasing as a function of $3u^2$. Since $3u^2$ is decreasing as a function of $t_2$, we finally apply the chain rule to see
$-(x^2+3u^2-1/4-1/28)e^{-21(x^2+3u^2-1/4)}$ is increasing as a function of $t_2$ on all of $A$. 
In the same way, we can check that it is increasing as a function of $t_1$, along with the following, analogous claims: 
\begin{itemize}
\item The quantity $-\left[(\frac12-x)^2+3(\frac12-u)^2-1/4 -4/7\right] e^{-21\left[(\frac12-x)^2+3(\frac12-u)^2-1/4 \right]}$ is decreasing in $t_1$ and $t_2$.
\item The quantity $-\left[(x-1)^2+3u^2-1/4-2/5\right] e^{-21\left[(x-1)^2+3u^2-1/4\right]}$
 is decreasing in $t_1$ and increasing in $t_2$.
\end{itemize}

In summary, we can take the decomposition
\begin{align*}
\frac{ \partial [e^{a/4} \tf_T(t_1,t_2)]}{\partial a}\bigg|_{a=21}&=
-(x^2+3u^2-1/4-1/28)e^{-21(x^2+3u^2-1/4)}-\frac{1}{28}e^{-21(x^2+3u^2-1/4)}
\\ 
&\hspace{.5cm}-\left[(\frac12-x)^2+3(\frac12-u)^2-1/4-4/7 \right] e^{-21\left[(\frac12-x)^2+3(\frac12-u)^2-1/4 \right]}\\
&\hspace{.5cm}-\frac47 e^{-21\left[(\frac12-x)^2+3(\frac12-u)^2-1/4 \right]} 
\\
&\hspace{.5cm} -\left[(x-1)^2+3u^2-1/4-2/5\right] e^{-21\left[(x-1)^2+3u^2-1/4\right]}-\frac25e^{-21\left[(x-1)^2+3u^2-1/4\right]}.
\end{align*}
where each term is either increasing or decreasing in $t_1$ and $t_2$ on each of our line segments. \\
%DO WE NEED TO SAY SOMETHING ABOUT HOW WE CHOP UP THE INTERVALS...

Finally, we present the proof of Lemma \ref{largea4ptt1growth}.
\begin{proof}
    Recall it remains to show that at $P=( \cos(2\pi \frac{\sqrt{3}}{4}), \cos(2\pi \frac{\sqrt{3}}{12})):=(t_1',t_2')$,
\[
\frac{\partial (\tf-\tg)}{\partial t_1}\bigg|_{P}\geq 0, \hspace{1 cm} \frac{\partial^2 \tf-\tg}{\partial t_1 \partial t_2}\bigg|_{P}\geq 0.
\]

For the first inequality, we first compute using Lemmas \ref{basicderivs} and \ref{basicvalues}, $x=\frac{\sqrt{3}}{4}$, and $u=\frac{\sqrt{3}}{12}$ that 
\begin{align}
\frac{\partial \tf}{\partial t_1}\bigg|_{P}&\geq \frac{a}{\pi \sin(2\pi x)}\left( e^{-ax^2}(x-(1-x)e^{-a(1-2x)}) e^{-3a u^2} -(\frac12-x)e^{-a(\frac12-x)^2}2(1+\epsilon)e^{-3a(1/2-u)^2} \right)\\
&= \frac{ae^{-a/4}}{\pi \sin(2\pi x)}\left((x-(1-x)e^{-a(1-2x)}) -(\frac12-x)2(1+\epsilon)e^{-a(1-x-3u)} \right)\\
&\geq  \frac{38ae^{-a/4}}{\pi 41}
\end{align}
using the fact that $x-(1-x)e^{-a(1-2x)}\geq \frac{39}{100}$,  $(\frac12-x)2(1+\epsilon)e^{-a(1-x-3u})\leq \frac{1}{100}$, and $\sin(2\pi x)\leq \frac{41}{100}$
 for $a\geq 21$.
On the other hand, 
since $\cos(2\pi u)\leq \frac{62}{100}$,
\begin{align}
\frac{\partial \tg}{\partial t_1}\bigg|_{P}&\leq\frac{62b_1}{100} \\
&\leq \frac{a e^{-a/4} 3 (1+\epsilon)62}{100\pi^2}.
\end{align}
Thus, 
\begin{align}
\frac{\partial (\tf-\tg)}{\partial t_1}\bigg|_{P}&\geq
\frac{ae^{-a/4}}{\pi}\left( \frac{38}{41}-\frac{3(1+\epsilon)62}{100\pi}\right)>0,
\end{align}
as the final inner quantity is positive. 

It remains to show $\frac{\partial^2 \tf-\tg}{\partial t_1 \partial t_2}\bigg|_{P}> 0$ in much the same fashion. 
%Using the same bound on $b_1$, 
%\begin{align}
%\frac{\partial^2 \tg}{\partial t_1 \partial t_2}\bigg|_{P}&= b_1 \leq \frac{3a e^{-a/4}  (1+\epsilon)}{ \pi^2}.
%\end{align}

Using Lemma \ref{basicderivs},
\begin{align}
\frac{\partial^2 \tf}{\partial t_1 \partial t_2}\bigg|_{P}&\geq \tfo'(\cos(2\pi \sqrt{3}/4))\tft'(\cos(2\pi \sqrt{3}/12))\\
&\geq \frac{39 ae^{-a x^2}}{41 \pi }\frac{3a e^{-3 a u^2}(u-(1-u)e^{-3a(1-2u)})}{\pi \sin(2\pi u)}\\
&\geq \frac{5\times3\times39 a^2e^{-a/4}}{4\times41 \pi^2 }(u-(1-u)e^{-3a(1-2u)})\\
&\geq \frac{14\times5\times3\times39 a^2e^{-a/4}}{100\times4\times41 \pi^2 }
\end{align}
since $\sin(2\pi u)\leq \frac{4}{5}$ and $u-(1-u)e^{-3a(1-2u)}\geq \frac{14}{100}$.
Thus, 
\begin{align}
\frac{\partial^2 \tf-\tg}{\partial t_1 \partial t_2}\bigg|_{P}&\geq
\frac{14\times5\times3\times39 a^2e^{-a/4}}{100\times4\times41 \pi^2 }- \frac{a e^{-a/4} 3 (1+\epsilon)}{ \pi^2}\\
&=\frac{a e^{-a/4}}{\pi^2}\left(\frac{14\times5\times3\times39a}{100\times4\times41}-3(1+\epsilon)\right).
\end{align}
The inner quantity is increasing in $a$ and so it suffices to check its positivity at $a=21$. 
\end{proof}

%%%%%%%%%%%%%%%%%%%%%%%%%%%%%%
\subsection{Large $a$ and 6 points}
%%%%%%%%%%%%%%%%%%%%%%%%%%%%%%
%%%%%%%%%%%%%%%%%%%%%%%%%%%%%%%
\subsubsection{Coefficient Bounds}
%%%%%%%%%%%%%%%%%%%%%%%%%%%%%%%%%
Again, we take $a\geq 9.6$. Our first task is using estimates on $\theta$ to bound the coefficients of $\tg$. 
We obtain
\begin{lemma}
\label{abounds}
\begin{align}
   0\leq \frac{2(1-\epsilon)a}{\sqrt{3}\pi} \leq & e^{a/3 }a_{0,1}\leq \frac{2(1+\epsilon)a}{\sqrt{3}\pi}\\
   %%%%%%%%%%%%%%
0\leq 1/2(-1-6\epsilon) +\frac{(1-\epsilon)a}{\sqrt{3}\pi}
\leq& e^{a/3 } a_{1,0}\leq 1/2(-1+3\epsilon)+\frac{(1+\epsilon)a}{\sqrt{3}\pi}  \\
%%%%%%%%%%%%%%%%%%
0\leq
\frac32 \leq& e^{a/3 }a_{0,0}\leq \frac{3}{2}(1+3\epsilon)\\
%%%%%%%%%%%%%%%%%%%%%
\frac89\left(-(1+\epsilon)+\frac{\sqrt{3}a(1-\epsilon)}{2\pi}\right)
\leq& 
e^{a/3 }a_{0,2}
\leq \frac89 \left( \frac98 \epsilon_2-(1+\epsilon) +\frac{\sqrt{3}a(1+\epsilon)}{2\pi}\right)\\
%(-2-7\epsilon + \frac{(3-2\epsilon) a}{\sqrt{3}\pi})\leq  & e^{a/3 }a_{0,2} \leq \frac49e^{-a/3}(-2+2\epsilon+\delta +\frac{(3+3\epsilon)a}{\sqrt{3}\pi})\\
%%%%%%%%%%%%%%%%%%%%%%%%%%%
0\leq \frac32-\frac29(1+\epsilon)+\frac{\sqrt{3}a(1-\epsilon)}{9\pi}
\leq &
e^{a/3 }b_{0,0}
\leq 
\frac32(1+3\epsilon)+\frac14 \epsilon_2 -\frac29(1+\epsilon) +\frac{\sqrt{3}a(1+\epsilon)}{9\pi}
%0\leq e^{-a/3}(\frac32 +\frac19(-2-7\epsilon)+ \frac{(3-2\epsilon) a}{\sqrt{3}\pi})\leq & e^{a/3 }b_{0,0}\leq e^{-a/3}(\frac32 +4\epsilon+ \frac19 (-2+2\epsilon+\delta) +\frac{(3+3\epsilon)a}{\sqrt{3}\pi}). 
\end{align}
\end{lemma}
\begin{proof}
We begin by using lemmas \ref{basicvalues} and \ref{thetaderivin[-1/2,1/2]} to multiply our bounds on $\tfo(1),\tft'(-\frac12)$ to obtain our bound on $a_{0,1}$.

Next, using this bound, combined with Lemma \ref{basicvalues} and our definition 
\[
a_{1,0}:=\frac{\tfo(1)\tft(-1/2)-\tfo(-1)\tft(1/2)}{2}+\frac{a_{0,1}}{2},
\] we obtain the bounds for $a_{1,0}$. We also use the fact that $e^2<e$ so that $(1+\epsilon)^2<1+3\epsilon$. 

The bounds for 
\[a_{0,0}:=\frac{\tfo(1)\tft(-1/2)+\tfo(-1)\tft(1/2)}{2}
\]
follow in the same manner. For $a_{0,2}$, we next compute that 
\[
0 \leq \frac49 \tfo(-1)\tft(-1)\leq\frac49 e^{-a}4(1+\epsilon)^2\leq e^{-a/3}\epsilon_2.
\]
Now using the definition,
\[
a_{0,2}=\frac49\tfo(-1)(\tft(-1)-\tft(1/2)+\frac32 \tft'(1/2))
\]
 we obtain the bounds for $a_{0,2}$.

Finally, the $b_{0,0}$ bound follows immediately from previous bounds and the definition $b_{0,0}:=a_{0,0}+a_{0,2}/4$.
\end{proof}
We use the notation $a_{i,j}^{u}$ $a_{i,j}^{l}$ for the upper and lower bounds respectively, and we note that the bounds are linear in $a$ with positive slope, up to a factor of $e^{-a/3}$, and so the nonnegativity follows from simply checking when $a=9.6$.

%%%%%%%%%%%%%%%%%%%%%%%%%%
\subsubsection{Satisfying Necessary Conditions for large $a$}
%%%%%%%%%%%%%%%%%%%%%%%%%
Recall we aim to show 
\begin{align}
\frac{\partial (\tf-\tg)}{\partial t_1}(-1,-1)> 0\\
\frac{\partial (\tf-\tg)}{\partial t_1}(-1,1/2) > 0\\
\frac{\partial (\tf-\tg)}{\partial t_1}(1,-1/2)< 0.
\end{align}
\begin{proof}
For the first condition, using Lemma \ref{abounds} and the absolute monotonicity of $\tth$, we have:
\[
    e^{a/3}\left(\frac{\partial (\tf-\tg)}{\partial t_1}(-1,-1)\right)\geq 
     -e^{a/3}\left(\frac{\partial \tg}{\partial t_1}(-1,-1)\right)
     %&=
    %e^{a/3}(a_{1,0}-a_{0,2})\\
    \geq e^{a/3}a_{0,2}^{l}-e^{a/3}a^{u}_{1,0}> 0,
\]
and this last inequality is easy to check since $e^{a/3}a_{0,2}^{l}-e^{a/3}a^{u}_{1,0}$ is linear in $a$. \\

Now incorporating lemmas \ref{basicderivs} and \ref{thetaderivpm1},
\begin{align}
e^{a/3}\left( \frac{\partial (\tf-\tg)}{\partial t_1}(-1,1/2)\right) &\geq
e^{a/3}\frac{a(a-2)}{2\pi^2}e^{-a/4}e^{-a/12}-(e^{a/3}a^{u}_{1,0}+\frac{e^{a/3}}{2} a^{u}_{0,2})\\
&=\frac{a(a-2)}{2\pi^2}-(e^{a/3}a^{u}_{1,0}+\frac{e^{a/3}}{2} a^{u}_{0,2})> 0,\\
\end{align}
because this last quantity is quadratic in $a$ and convex, so it suffices to check it is positive with positive slope at $a=9.6$,
Finally,
\begin{align}
e^{a/3}\left( \frac{\partial (\tf-\tg)}{\partial t_1}(1,-1/2)\right) &\leq
e^{a/3}\frac{a}{2\pi^2}e^{-a/3}(1+\epsilon)-(e^{a/3}a^{l}_{1,0}-\frac{e^{a/3}}{2} a^{u}_{0,2})\\
&=\frac{(1+\epsilon)a}{2\pi^2}-(e^{a/3}a^{l}_{1,0}-\frac{e^{a/3}}{2} a^{u}_{0,2})< 0\\
\end{align}
and this last quantity is linear in $a$, so again the final check is straightforward.
\end{proof}

%%%%%%%%%%%%%%%%%%%%%%%%%%%%%%%%%%%%%%%%%%%%%%%%%%
\subsubsection{Bounds for proofs of Lemmas~\ref{linearization1} and \ref{linearization2}}
\label{proofsoflinearization}
%%%%%%%%%%%%%%%%%%%%%%%%%%%%%%%%%%%%%%%%%%%%%%%%%
It remains to show that 
%\[
%\tf_T\geq \tg_{a,b}'
%\hspace{.25cm}
%\text{ and } 
%\hspace{.25cm}
%\frac{\partial \left[e^{a/3}(\tf_T- \tg_{a,b}'\right]}{\partial a}\bigg|_{a=9.6}\geq 0
%\]
$\tf_T\geq \tg_{c,d}^*$
for various $c,d$ and 
various line segments in the critical region $[-1,1]\times [-1/2,1/2]$, where $\tg_{c,d}^*$ is an upper bound for $\tg_{c,d}$ obtained by replacing $a_{i,j}$'s with upper and lower bounds from Lemma \ref{abounds}. \\

In particular, when $c<0$ and $0<d$,   we define 
\begin{align}
\tg_{c,d}^*(t_1,t_2) &:= b^{u}_{0,0}+a_{1,0}^{l} t_1 a_{0,1}^{u}t_2+a_{0,2}^{u} t_2^2+ a_{0,2}^{l}(d t_1 +c t_2 -c d),
\end{align}
and then $\tg_{c,d}(t_1,t_2)\le \tg_{c,d}^*(t_1,t_2)$ if $c<t_1<0$ and $0<t_2<d$. 
For $c<0$ and $d\leq 0$, we define
\begin{align}
\tg_{c,d}^*(t_1,t_2) &:= b^{u}_{0,0}+a_{1,0}^{l} t_1 a_{0,1}^{l}t_2+a_{0,2}^{u} t_2^2+ a_{0,2}^{u}(d t_1 +c t_2 -c d),
\end{align}
which gives $\tg_{c,d}(t_1,t_2)\le \tg_{c,d}^*(t_1,t_2)$ if $c<t_1<0$ and $-1\le t_2<d\le 0$. 
%Note in both of these cases, $\tg_{c,d}^*$  is linear in $a$ up to a factor of $e^{a/3}$. Specifically, 
To complete the proof of Lemmas~\ref{linearization1} and \ref{linearization2} we show inequalities \eqref{partialaCond}  hold with
\begin{enumerate}
    \item $\tg_{-1,1/2}^*$ on the segments $\{(-\sqrt{2}/2,t_2)\mid t_2\in[\frac14,\frac12]\} \cup \{(t_1,\frac14)\mid t_1\in[-1,-\sqrt{2}/2]\}$
    \item $\tg_{-1,1/4}^*$ on the segments  $\{(-\sqrt{2}/2,t_2)\mid t_2\in[0,\frac14]\} \cup \{(t_1,0)\mid t_1\in[-1,-\sqrt{2}/2]\}$
    \item $\tg_{-\sqrt{2}/2,0}^*$ on the segments  $\{(-\sqrt{2}/2,t_2)\mid t_2\in[-.1,0]\} \cup
    \{(0,t_2)\mid t_2\in[-.1,0]\}
    \cup
    \{(t_1,-.1) \mid t_1\in[-\sqrt{2}/2,0]\}$
    \item $\tg_{-\sqrt{2}/2,-.1}^*$ on the segments  $\{(-\sqrt{2}/2,t_2)\mid t_2\in[-.2,-.1]\} \cup
    \{(0,t_2)\mid t_2\in[-.2,-.1]\}
    \cup
    \{(t_1,-.2)\mid t_1\in[-\sqrt{2}/2,0]\}$,
\end{enumerate}
in \cite{Mathematica} with the same procedure as used in Section \ref{largea4Pt}.
%Note: For the horizontal line segments with $t_2=\frac14,-.1,-.18$, we use $u=\frac{21}{100},\frac{107}{400},\frac{7}{25}$, respectively in the formula for $\tf_T$ rather than the exact $u$ values like $\frac{\arccos{\frac14}}{2\pi}$.
%But because $\tf_T$ is decreasing in $u$, it suffices to take these upper bounds on the exact $u$-values, and it makes the calculations easier. 

%So then the computer check needs to go as follows:
%We take each of the segments which has either $t_1$ or $t_2$ varying in $[p_0,p_n]$. 
%Abusing notation, we treat $\tf_T$ and $\tg_T$ as one variable functions of $t$ since along each segment one of $t_1,t_2$ is fixed. 
%Then we partition the interval up into $p_0,p_1,..,p_n$ (not arbitrary, we specify how). 
%Both $\frac{\partial \tg^{u}_{a,b}}{\partial a}\bigg|_{a=9.6}$ and $\tg_{a,b}^{u}$ are convex in $t$ (linear in the case of $t_2$ being fixed), 
%thus attaining maximums only at endpoints. 
%Meanwhile, as we have shown in the previous lemma, $\frac{\partial e^{a/3}\tf_T}{\partial a}\bigg|_{a=9.6}$ and $\tf_T$ are both increasing in $t$, 
%so their minimum occurs at left endpoints. 
%Thus, we show that for each $i=0,1,...,n-1$, it suffices to show \[
%\tf_T(p_i)\geq \max\{\tg_{a,b}(p_i),\tg_{a,b}(p_{i+1})\}
%\] and similarly,
%\[
%\frac{\partial e^{a/3}\tf_T}{\partial a}\bigg|_{a=9.6,t=p_i}\geq
%\max \{ \frac{\partial \tg^{u}_{b,c}}{\partial a}\bigg|_{a=9.6,t=p_i},\frac{\partial \tg^{u}_{b,c}}{\partial a}\bigg|_{a=9.6,t=p_{i+1}} \}
%\]
\subsubsection{Computations for proof of Lemma \ref{t1deriv}}
\label{proofoft1deriv}
Recall it remains to show $\frac{\partial (\tf-\tg)}{\partial t_1}(-\sqrt{2}/2,0)\geq 0$. By Lemma \ref{basicderivs}, \eqref{thetaderivin[-1/2,1/2]}, \ref{basicvalues} and our coefficient estimates, we have
\begin{align}
    \frac{\partial (\tf-\tg)}{\partial t_1}(-\sqrt{2}/2,0)&= \tfo'(-\sqrt{2}/2)\tft(0)-a_{1,0}\\
    &\geq 
    \frac{(ae^{-a(3/8)^2}(3/8-5/8e^{-a/4}))}{\pi (1/\sqrt{2})}e^{-3a/16}-e^{-a/3}(-1/2+\frac32\epsilon+\frac{(1+\epsilon)a}{\sqrt{3\pi}})\\
    &=e^{-a/3}\left[\frac{\sqrt{2}a e^{a/192}(3-5e^{-a/4})}{8\pi}-(-1/2+\frac32 \epsilon+\frac{(1+\epsilon)a}{\sqrt{3\pi}})\right].
\end{align}
We claim that 
$$\frac{\sqrt{2}ae^{a/192}(3-5e^{-a/4})}{8\pi}-(-1/2+\frac32 \epsilon+\frac{(1+\epsilon)a}{\sqrt{3\pi}})$$
is positive for $a=9.6$ and increasing in $a$ for $a\geq 9.6$. The first of these conditions is a simple check. For the latter,
\begin{align}
\frac{d}{da}&\left[    \frac{\sqrt{2}ae^{a/192}(3-5e^{-a/4})}{8\pi}-(-1/2+2\epsilon+\frac{(1+\epsilon)a}{\sqrt{3\pi}})\right]
\\
&=\frac{ e^{-47 a/
  192} (-960 + 235 a + 576 e^{a/4} + 3 a e^{a/4})}{768\sqrt{2} \pi}-\frac{(1+\epsilon)}{\sqrt{3}\pi}\\
  &> \frac{ e^{-47 a/
  192} (576 e^{a/4} + 3 a e^{a/4})}{768\sqrt{2} \pi}-\frac{(1+\epsilon)}{\sqrt{3}\pi}\\
  &=\frac{ (576 e^{a/192} + 3 a e^{a/192})}{768\sqrt{2} \pi}-\frac{(1+\epsilon)}{\sqrt{3}\pi} >0
\end{align}
where this last quantity is greater than $0$ because it's true for $a=9.6$ and clearly increasing in $a$.

%%%%%%%%%%%%%%%%%%%%%%%%%%%%%%%%%%%%%%%%%%%%%% 
\subsubsection{Positivity of $L_1$ for Proposition \ref{gradientcomp}
\label{proofofgradientcomp}}
%%%%%%%%%%%%%%%%%%%%%%%%%%%%%%%%%%%%%%%%%%%%
Recall to prove Lemma \ref{gradientcomp}, it suffices to show that 
\[
L_1(t_1,t_2):=\frac{\tft'\tfo}{\tfo'\tft}-\frac{\frac{\partial \tg}{\partial t_2}}{\frac{\partial \tg}{\partial t_1}}> 0 
\]
on $[0,1]\times [-1/2,0]$. 
We first bound $\frac{\tfo\tft'}{\tfo'\tft}$below. 
Using the Lemma \ref{basicvalues}, $\tfo\geq e^{-ax^2}$. Since $\lvert t_2\rvert \leq  \frac12$, by \eqref{basicvaluesapplication} it follows that
\begin{align}
    \tft < e^{-3a u^2}(1+\epsilon). 
\end{align}
With Lemma \ref{basicderivs}, $\tfo'\leq \frac{a x e^{-ax^2}}{\pi\sin(2\pi x)}$. Finally, with $u\in[1/4,1/3]$, we have
by \eqref{thetaderivin[-1/2,1/2]},
\begin{align}
\tft' \geq 
\frac{3a u e^{-3au^2}(1-\epsilon)}{\pi \sin(2\pi u)}.
\end{align}
Combining these bounds, we obtain:
\begin{align}
    \frac{\tfo\tft'}{\tfo'\tft}\geq \frac{e^{-3a x^2}3a e^{-3au^2}u (1-\epsilon)}{\pi \sin(2\pi u)}\cdot \frac{\pi\sin(2\pi x)}{a x e^{-ax^2}e^{-3a u^2}(1+\epsilon)}=\frac{3u(1-\epsilon)\sin(2\pi x)}{(1+\epsilon)x\sin(2\pi u)}.
\end{align}
Next, a couple of observations:
\begin{lemma}
\label{sinfacts}
Recall $x=\frac{\arccos(t_1)}{2\pi}$ and $u=\frac{\arccos(t_2)}{2\pi}$. The function $t_1\rightarrow \frac{2\pi\sqrt{1-t_1^2}}{\arccos(t_1)}=\frac{\sin(2\pi x)}{x}$ is concave for $t_1\in (-1,1)$. Also, $t_2 \rightarrow \frac{u}{\sin(2\pi u)}$ is decreasing for $t_2 \in(-1,1)$. 
\end{lemma}
\begin{proof}
Let $\phi(t)=\frac{\sqrt{1-t^2}}{\arccos{t}}$. 
Then
\[
\phi''(t)=-\frac{(-2 + 2 t^2 + t \sqrt{1 - t^2}\arccos(t) + \arccos(t)^2)}{((1 - t^2)^{
  3/2}\arccos(t)^3)}.
\]
Since the denominator is positive, it suffices to show positivity of the numerator for $t\in[-1,1]$. Equivalently, letting $y=\arccos[t]$, $y\in[0,\pi]$, we'll show positivity of 
\[
N_1(y):=-2 \sin(y)^2 + \cos(y)\sin(y)y + y^2.
\]
Now $N_1(0)=0$, $N_1'(0)=0$
and $N_1''(y)=4 \sin(y) (-y \cos(y) + \sin(y))\geq0$ since 
$1\geq y\cot(y)$ for $y\in(0,\pi]$ as $y\leq \tan(y)$ for $y\in(0,\pi/2)$ and $y\cot(y)\leq 0$ for $y\in (\pi/2,\pi)$. Now for the second part of the lemma, it suffices to show $\frac{u}{\sin(2\pi u)}$ is increasing in $u$ for $u\in [0,1/2]$.
We have
\[
\left(\frac{u}{\sin( u)}\right)'=(1 - u \cot( u))\csc( u)\geq 0 
\]
since $1\geq y \cot( y)$ for $y\in (0,\pi)$ as shown above.
\end{proof}
Returning to the proof, we have 
\begin{align}
L_1&\geq \frac{3u(1-\epsilon)\sin(2\pi x)}{(1+\epsilon)x\sin(2\pi u)}-\frac{\frac{\partial \tg}{\partial t_2}}{\frac{\partial \tg}{\partial t_1}}\\
&=\frac{3u(1-\epsilon)\sin(2\pi x)}{(1+\epsilon)x\sin(2\pi u)}-\frac{a_{0,1}+a_{0,2}(t_1+2t_2)}{a_{1,0}+a_{0,2}t_2}\\ 
\label{L1bound}
&=\frac{3u(1-\epsilon)\sin(2\pi x)}{(1+\epsilon)x\sin(2\pi u)}-\left(2+\frac{\tfo(-1)\tft(1/2)-\tfo(1)\tft(-1/2)+a_{0,2}t_1}{{a_{1,0}+a_{0,2}t_2}}\right).
\end{align}
By Lemma \ref{sinfacts} and the linearity of $\frac{\frac{\partial \tg}{\partial t_2}}{\frac{\partial \tg}{\partial t_1}}$ in $t_1$ for fixed $t_1$, if $L_1(0,t_2), L_1(1,t_2)\geq 0$, then $L_1(t_1,t_2)\geq 0 $ for all $t_1\in[-1,1]$. We'll also be using the bound developed in the proof of bounding $a_{1,0}$ in Lemma \ref{abounds} that 
\[
e^{-a/3}(1-3\epsilon) \leq \tfo(-1)\tft(1/2)-\tfo(1)\tft(-1/2)\leq e^{-a/3}(1+6\epsilon)
\]
\\
If $t_1=1$, then for $t_2\in[-1/2,0]$, we obtain
\begin{align}
\frac{\tfo(-1)\tft(1/2)-\tfo(1)\tft(-1/2)+a_{0,2}t_1}{a_{1,0}+a_{0,2}t_2}
%&=\frac{a_{0,1}+a_{0,2}(1+2t_2)}{a_{1,0}+a_{0,2}t_2}\\
&=
\frac{\tfo(-1)\tft(1/2)-\tfo(1)\tft(-1/2)+a_{0,2}}{a_{1,0}+a_{0,2}t_2}\\
&\leq \frac{1+6\epsilon +e^{a/3}a^{u}_{0,2}}{e^{a/3}(a^{l}_{1,0}+t_2a^{u}_{0,2})}
%&=\frac{1+4\epsilon +\frac49((3+3\epsilon)y-2+\epsilon+\delta)}{-1/2-2\epsilon+(1-\epsilon)y+\frac49((3+3\epsilon)y-2+2\epsilon+\delta)t_2}.
\end{align} 
Technically, this lower bound requires that $e^{a/3}(a^{l}_{1,0}+t_2a^{u}_{0,2})>0$ for all $t_2\in [-\frac12,0]$, but this quantity is linear in $a$ so it's easy to check.   
Thus, for $t_2\in[-1/2,0]$,
\[
L_1\geq \frac{3u(1-\epsilon)2\pi}{(1+\epsilon)\sin(2\pi u)}-2-
\frac{1+6\epsilon +e^{a/3}a^{u}_{0,2}}{e^{a/3}(a^{l}_{1,0}+t_2a^{u}_{0,2})}
\]
and
we claim $\displaystyle{\frac{1+6\epsilon +e^{a/3}a^{u}_{0,2}}{e^{a/3}(a^{l}_{1,0}+t_2a^{u}_{0,2})}}$ is decreasing in $a$, from which it would follow that it suffices to check that the bound is at least 0 only when $a=9.6$ (as the other terms have no dependence on $a$).

To that end, note that as a function of 
$a$
\[
\frac{1+6\epsilon +e^{a/3}a^{u}_{0,2}}{e^{a/3}(a^{l}_{1,0}+t_2a^{u}_{0,2})}
\]%$\frac{1+4\epsilon +\frac49(3+3\epsilon)y-2+2\epsilonlon+\delta)}{-1/2-2\epsilon+(1-\epsilon)y+\frac49(3+3\epsilon)y-2+2\epsilonlon+\delta)t_2}$ 
is rational (with numerator and denominator both linear) and so the sign of its derivative is independent of $a$. Thus, checking it is negative for all $a$ is simple. 
%\begin{align}
%&(1+5\epsilon +e^{a/3}b^{u}_{0,2})'(e^{a/3}(b^{l}_{1,0}+t_2b^{u}_{0,2}))-(1+5\epsilon +e^{a/3}b^{u}_{0,2})(e^{a/3}(b^{l}_{1,0}+t_2b^{u}_{0,2}))'\\
%&=\frac19 (-7 + 4 \delta (-1 + e) + \epsilon^2 (17 - 60 t_2) - 12 t_2 - 2 \epsilon (47 + 36 t_2))\\
%&\leq1/9 (-1 + 4 \delta (-1 + e) - 58\epsilon + 47 \epsilon^2)\\
%&\leq \frac19(-1+47\epsilon^2)<0
%\end{align}

%where he have plugged in $t_2=-\frac12$ because the quantity is clearly decreasing in $t_2$.

In summary, for any $a\geq 9.6$, we have $L_1(1,t_2)> 0$ when $t_2\in[-1/2,0]$ if 
\[
\frac{3u(1-\epsilon)2\pi}{(1+\epsilon)\sin(2\pi u)}-2-
\frac{1+6\epsilon +e^{a/3}a^{u}_{0,2}}{e^{a/3}(a^{l}_{1,0}+t_2 a^{u}_{0,2})}> 0
\]
holds for $a=9.6$.
%\frac{1+4\epsilon +\frac49((3+3\epsilon)y-2+\epsilon+\delta)}{-1/2-2\epsilon+(1-\epsilon)y+\frac49((3+3\epsilon)y-2+2\epsilon+\delta)t_2}\geq 0 
The $x$ terms have disappeared as we took the limit $x\rightarrow0$. As we have shown this inequality is a difference of two increasing functions in $u$, this inequality is verified in \cite{Mathematica} using the same interval partition approach described in Section \ref{largea4Pt}\, which reduces our check to a finite number of point evaluations.\\

The case where $t_1=0$ is more simple. Here we again use \eqref{L1bound} with $t_1=0$ to obtain
\begin{align}
L(0,t_2)
%&\geq 
%\frac{3u(1-\epsilon)\sin(2\pi x)}{(1+\epsilon)x\sin(2\pi u)}-\left(2+\frac{\tfo(-1)\tft(1/2)-\tfo(1)\tft(-1/2)+a_{0,2}t_1}%{a_{1,0}+a_{0,2}t_2}\right)\\
&\geq\frac{12u(1-\epsilon)}{(1+\epsilon)\sin(2\pi u)}-\left(2+\frac{\tfo(-1)\tft(1/2)-\tfo(1)\tft(-1/2)}{a_{1,0}+a_{0,2}t_2}\right)\\
&\geq\frac{12u(1-\epsilon)}{(1+\epsilon)\sin(2\pi u)}-\left(2+\frac{1+6\epsilon}{e^{a/3}(a^{l}_{1,0}+a^{u}_{0,2}t_2)}\right)
\end{align}
and with the definitions of $a^l_{1,0}$ and $a^u_{0,2}$ this quantity is clearly increasing in $a$ and so it suffices to check when $a=9.6$ (the denominator is positive and increasing in $a$). Again, this inequality is handled in Mathematica with finitely many point evaluations as we have a difference of increasing functions in $u$. \\

 %%%%%%%%%%%%%%%%%%%%%%%%%%%%%%%%%%%%%%%%%%%%%
\subsubsection{Log Derivative Estimates for the proof of Lemma \ref{logderivprop}}
\label{proofoflogderivprop}
%%%%%%%%%%%%%%%%%%%%%%%%%%%%%%%%%%%%%%%%%%
It remains to show:
\begin{enumerate}
    \item $N(-1/2)<0$
    \item $L_2(0,-\frac15)> 0$
    \item $L_2(-\frac{\sqrt{2}}{2},0)>0$
\end{enumerate}

First, 
\begin{align*}
N(-1/2)&=b_{0,0}a_{0,2}-a_{0,1}a_{1,0}+\frac{a_{0,2}}{2}(2a_{1,0}-\frac12a_{0,2})\\
%&=a_{0,2}(b_{0,0}-\frac14a_{0,2}+a_{1,0})-a_{0,1}a_{1,0}\\
&=a_{0,2}(a_{0,0}+a_{1,0})-a_{0,1}a_{1,0}\\
&\leq a^{u}_{0,2}(a^{u}_{0,0}+a^{u}_{1,0})-a^{l}_{0,1}a^{l}_{1,0}\\
%&=e^{-2a/3}\left[\frac49(-2+2\epsilon+\delta +\frac{(3+3\epsilon)a}{\sqrt{3}\pi})(\frac{(3+3\epsilon)}{2}-1/2+2\epsilon+\frac{(1+\epsilon)a}{\sqrt{3}\pi} )-\frac{(2-\epsilon)a}{\sqrt{3}\pi}(-1/2-2\epsilon +\frac{(1-\epsilon)a}{\sqrt{3}\pi})\right]\\
%&=e^{-2a/3}\bigg[ (\frac{a}{\sqrt{3}\pi})^2((4/3+4/3\epsilon)(1+\epsilon)-(2-\epsilon)(1-\epsilon)))\\
%&+\frac{a}{\sqrt{3}\pi}(\frac49(1+\epsilon)(-2+2\epsilon+\delta)+ (4/3+4/3\epsilon)(1+7/2\epsilon)-(2-\epsilon)(-1/2-2\epsilon)\\
%&+\frac49(-2+2\epsilon+\delta)(1+7/2\epsilon)
%\bigg]
\end{align*}
Now this last quantity (up to a factor of $e^{2a/3}$) is quadratic in $a$, concave down, along with negative and decreasing for $a$, as shown in the notebook.\\

Next, we show $L_2(t_1,t_2):= \frac{\tfo(t_1)}{\tfo'(t_1)}-t_1-\frac{b_{0,0}+a_{0,1}t_2+a_{0,2}t_2^2}{a_{1,0}+a_{0,2}t_2} \geq 0$ at $(0,-1/5)$, or equivalently that 
\[
e^{a/3+a/16}\left[\tfo(t_1)(a_{1,0}+a_{0,2}(-1/5)) - \tfo'(t_1)(b_{0,0}+a_{0,1}(-1/5)+a_{0,2}(-1/5)^2) \right]\geq 0
\]
Using lemmas \ref{basicderivs} and \ref{basicvalues}, we have
\begin{align}
e^{a/3+a/16}\left[\tfo(t_1)(a_{1,0}+a_{0,2}(-1/5))- \tfo'(t_1)(b_{0,0}+a_{0,1}(-1/5)+a_{0,2}(-1/5)^2)\right]\geq \\
e^{a/3+a/16} \left[ e^{-a/16}(a^{l}_{1,0}+a^{u}_{0,2}(-1/5)) - a e^{-a/16}/(4\pi) (b^{u}_{0,0}+a^{l}_{0,1}(-1/5)+a^{u}_{0,2}(-1/5)^2) \right]\\
=
 e^{a/3}(a^{l}_{1,0}+a^{u}_{0,2}(-1/5)) - a/(4\pi) e^{a/3}(b^{u}_{0,0}+a^{l}_{0,1}(-1/5)+a^{u}_{0,2}(-1/5)^2) \geq 0,
\end{align}
and this last inequality is easy to show because the lower bound is quadratic and convex in $a$ with positive value and derivative, at $a=9.6$.

The case where $t_1=-\sqrt{2}/2,t_2=0$ is similar but requires a little more care. 
Unfortunately, the bounds 
\[
\tfo(t_1)\geq e^{-a x^2},  \hspace{.5cm}  \tfo'(t_1)\leq a x e^{-ax^2}/(\pi \sin(2\pi x)
\]
are too coarse to work for all $a\geq 9.6$. Instead, we must truncate one fewer term to get our lower bound for $\tfo$ and use two more terms for our upper bound of $\tfo'$ (see the proof of Lemma \ref{basicderivs}), to obtain 
\begin{align}
\frac{\tfo(t_1)}{\tfo'(t_1)}&\geq 
( e^{-ax^2}+e^{-a(x-1)^2} )
\frac{\pi \sin(2\pi x)}
{ a\left(xe^{-ax^2}+ (x-1)e^{-ax^2}+(x+1)e^{-a(x+1)^2}\right) }\\
&=\frac{\pi \sin(2\pi x)}{a} \frac{1+e^{-a(1-2x)}}{x-(1-x)e^{-a(1-2x)}+(x+1)e^{-a(1+2x)}}.
\end{align}
So for $t_1=-\sqrt{2}/2$, we obtain
\begin{align}
\frac{\tfo(-\sqrt{2}/2)}{\tfo'(-\sqrt{2}/2)}&\geq 
\frac{\pi \sqrt{2}}{2a} \frac{1+e^{-a/4}}{\frac38-\frac58 e^{-a/4}+\frac{11}{8} e^{-7a/4}}\\
&= \frac{4\pi \sqrt{2}}{a}
\frac{1+e^{-a/4}}{3-e^{-a/4}(5-11 e^{-3a/2})}\\
&\geq \frac{4\pi \sqrt{2}}{a}
\frac{1+e^{-a/4}}{3-(5-\epsilon)e^{-a/4}}.
\end{align}
Thus, 
\[
L_2(-\sqrt{2}/2,0)\geq \frac{4\pi \sqrt{2}}{a}
\frac{1+e^{-a/4}}{3-(5-\epsilon)e^{-a/4}}+\sqrt{2}/2-\frac{e^{a/3}b_{0,0}^{u}}{e^{a/3}a_{1,0}^{l}}
\]
which is nonnegative if and only if
\[
4\pi \sqrt{2}(e^{a/3}a^{l}_{1,0})\frac{1+e^{-a/4}}{3-(5-\epsilon)e^{-a/4}}  +\sqrt{2}/2 a (e^{a/3}a^{l}_{1,0})-a(e^{a/3}b^{u}_{0,0}) \geq 0
\]

Now if $9.6 \leq  a\leq c$, then we have the inequality
\begin{align}
\label{L2 inequality}
4\pi \sqrt{2}(e^{a/3}a^{l}_{1,0})\frac{1+e^{-c/4}}{3-(5-\epsilon)e^{-c/4}}   \leq 
4\pi \sqrt{2}(e^{a/3}a^{l}_{1,0})\frac{1+e^{-a/4}}{3-(5-\epsilon)e^{-a/4}}  
\end{align}
since 
\[
\frac{1+e^{-c/4}}{3-(5-\epsilon)e^{-c/4}}
\]
is decreasing in $c$ for $c\geq 9.6$. So it remains to show in the accompanying Mathematica document that for every $a\in \R$, there is some choice of $c\geq a$ such that
\[
L_3(a,c):= 4\pi \sqrt{2}(e^{a/3}a^{l}_{1,0})\frac{1+e^{-c/4}}{3-(5-\epsilon)e^{-c/4}}+\sqrt{2}/2 a (e^{a/3}a^{l}_{1,0})-a(e^{a/3}b^{u}_{0,0})> 0 
\]
In particular, we do so by showing that for each $i=1,\dots,6$ and the sequence $a_0=9.6,9.8,10,10.2,11,12,\infty=a_6$, we have $L_3(a,a_i)\geq 0$ for $a\geq a_{i-1}$. 
Each of these checks is easy since $L_3(a,c)$ is quadratic in $a$ for fixed $c$. 
As with the $a\leq \pi^2$ case, all of these checks and algebraic simplifications are verified in \cite{Mathematica}.

%%%%%%%%%%%%%%%%%%%%%%%%%%%%%%%%%%%%%%
%END OF TEDIOUS BOUNDS AND COMPUTATIONS :):)
%%%%%%%%%%%%%%%%%%%%%%%%%%%%%%%%%%%%%%%%%

\section*{Acknowledgments}
The  authors thank Henry Cohn, Denali Relles, Larry Rolen, Ed Saff, and Yujian Su for helpful discussions at various stages of this project and also thank the anonymous reviewers for their careful reading and suggestions.
%%%%%%%%%%%%%%%%%%%%%%%%%%%
%APPENDIX
%%%%%%%%%%%%%%%%%%%%%%%%%%%%
\nocite{*}

\bibliographystyle{amsplain}
\bibliography{A2citations.bib}

\end{document}